\documentclass[a4paper]{article}
\usepackage[utf8]{inputenc}
\usepackage{hyperref}
\usepackage{amsmath, amsthm}
\usepackage{amsfonts}
\usepackage{amssymb}
\usepackage{mathtools}
\usepackage[dvipsnames]{xcolor}
\hypersetup{
  colorlinks=true,
  citecolor=blue,
  linkcolor=purple!80
}
 \usepackage{algorithm}%
\usepackage{algorithmicx}%
\usepackage{algpseudocode}%
\usepackage{pgfplots}
\pgfplotsset{compat=1.17}
\usepackage{subcaption}
\usepackage{enumitem}
\usepackage{microtype}
\usepackage[nameinlink,noabbrev,capitalize]{cleveref}

\usepackage{aligned-overset}

\usepackage{booktabs} %
\usepackage{array}  

 \usepackage[left=2cm,right=2cm,top=2cm,bottom=2cm]{geometry}
\usepackage{comment}

\usepackage{hhline}
\usepackage{multirow}
\usepackage{colortbl}
\usepackage{multicol}
\usepackage[isbn=false,maxbibnames=99,backref=true,giveninits=true
]{biblatex}
\newcommand{\grad}{\nabla}
\DeclareMathOperator*{\argmin}{arg\,min}
\newcommand{\norm}[2]{\left\lVert #1\right\rVert_{#2}}

\DeclareMathOperator{\dom}{dom}
\DeclareMathOperator{\dist}{dist}

\DeclareMathOperator{\interior}{int}
\DeclareMathOperator{\parm}{par}
\DeclareMathOperator{\aff}{aff}
\DeclareMathOperator{\relint}{ri}

\definecolor{mydarkgreen}{RGB}{39,130,67}
\newcommand{\green}{\color{mydarkgreen}}

\newcommand{\algname}[1]{{\green \small \sf #1}}

\newcommand{\algnametitle}[1]{{ \sf #1}}

\newcommand{\HH}{\mathcal{H}}

\newcommand{\F}{\mathcal{F}}
\newcommand\so{
  \mathchoice
    {{\scriptstyle\mathcal{O}}}%
    {{\scriptstyle\mathcal{O}}}%
    {{\scriptscriptstyle\mathcal{O}}}%
    {\scalebox{.7}{$\scriptscriptstyle\mathcal{O}$}}%
  }

\newcommand\R{\mathbb{R}}

\renewcommand\d{\mathop{}\!\mathrm{d}}
\newcommand\dt{\d t}

\newcolumntype{b}{X}
\newcolumntype{s}{>{\hsize=.5\hsize}X}

\definecolor{darkorange}{rgb}{8.,.4,0.}

\newtheorem{theorem}{Theorem}[section]
\newtheorem{prop}[theorem]{Proposition}
\newtheorem{lem}[theorem]{Lemma}
\newtheorem{cor}[theorem]{Corollary}

\newtheorem{ass}{Assumption}
\crefname{ass}{Assumption}{Assumptions}

\newtheorem{remark}{Remark}
\crefname{remark}{Remark}{Remarks}

\theoremstyle{definition}
\newtheorem{defn}[theorem]{Definition}
\usepackage{booktabs}
\usepackage{tabularx}
\title{\algnametitle{LeAP-SSN}: A Semismooth Newton Method \\with Global Convergence Rates}
\author{Amal Alphonse  \and Pavel Dvurechensky \and Ioannis P.~A. Papadopoulos \and Clemens Sirotenko}
\date{}
\begin{document}
\maketitle
\makeatletter
\begingroup
\renewcommand{\thefootnote}{}           %
\renewcommand{\@makefntext}[1]{#1}      %
\footnotetext{Weierstrass Institute, Mohrenstrasse 39, 10117 Berlin, Germany (\href{mailto:alphonse@wias-berlin.de}{alphonse@wias-berlin.de}, \href{mailto:dvurechensky@wias-berlin.de}{dvurechensky@wias-berlin.de}, \href{mailto:papadopoulos@wias-berlin.de}{papadopoulos@wias-berlin.de}, \href{mailto:sirotenko@wias-berlin.de}{sirotenko@wias-berlin.de})}
\endgroup
\makeatother%
\begin{abstract}
We propose \algname{LeAP-SSN} (\underline{Le}venberg--Marquardt \underline{A}daptive \underline{P}roximal \underline{S}emi\underline{s}mooth \underline{N}ewton method), a semismooth Newton-type method with a simple, parameter-free globalisation strategy that guarantees convergence from arbitrary starting points in nonconvex settings to stationary points, and under a Polyak--\L{}ojasiewicz condition, to a global minimum,  
in Hilbert spaces. The method employs an adaptive Levenberg--Marquardt regularisation for the Newton steps, combined with backtracking, and does not require knowledge of problem-specific constants. We establish global nonasymptotic rates: $\mathcal{O}(1/k)$ for convex problems in terms of objective values, $\mathcal{O}(1/\sqrt{k})$ under nonconvexity in terms of subgradients, %
and linear convergence under a Polyak--\L{}ojasiewicz condition. %
The algorithm achieves superlinear convergence under mild semismoothness and Dennis--Mor\'e or partial smoothness conditions, even for non-isolated minimisers. By combining strong global guarantees with superlinear local rates in a fully parameter-agnostic framework, \algname{LeAP-SSN} bridges the gap between globally convergent algorithms and the fast asymptotics of Newton's method. The practical efficiency of the method is illustrated on representative problems from imaging, contact mechanics,  and machine learning.
\end{abstract}

\section{Introduction}
In this paper, we consider the composite minimisation problem
\begin{equation}
\label{eq:opt}
   \min_{x \in \HH} \{F(x)\coloneqq f(x) + \psi(x) \}, \tag{P}
\end{equation}
where $\HH$ is a %
Hilbert space,  $\psi\colon \HH \to \overline{\mathbb{R}}$ is a potentially nonsmooth proper convex lower semicontinuous function, and $f \in C^{1,1}(\HH; \mathbb{R})$ is a (possibly nonconvex) function such that the Fréchet derivative $f'$ is not necessarily smooth. 
In particular, we are interested in the setting when $f'$ possesses some semismoothness-like properties and our main goal is to propose a variant of the semismooth Newton (SSN) method  that enjoys both global convergence rate guarantees and local superlinear convergence.
Note that when $\psi \equiv 0$ and $f'$ is semismooth, the optimality condition for \eqref{eq:opt} is given by the nonlinear and nonsmooth equation $f'(x)=0$ with the semismooth operator $f'\colon \HH \to \HH^*$.

The SSN method is a generalisation of the classical Newton method used for solving (systems of) nonsmooth equations and for the minimisation of functions with non-Lipschitz second derivatives, see for example the pioneering works \cite{MR1972649, MR1786137, MR1250115, MR1216791, MR1972217, hintermuller2002primal}. SSN methods are especially prominent when solving infinite-dimensional problems, e.g., in PDE-constrained optimisation, due to their mesh-independent and fast local convergence properties. There are recent theoretical results for the method in the context of optimisation problems; these include local superlinear convergence under (local) strong convexity or non-degeneracy  and global asymptotic convergence for its variants with a Levenberg--Marquardt regularisation and/or Armijo-type linesearch strategy \cite{potzl2022second,potzl2024inexact,khanh2024globally,mordukhovich2024second,wachsmuth2025globalized}. At the same time, global convergence rates remain unknown for the SSN method and its variants, even in the simpler setting where $\psi \equiv 0$. In this paper, we close this theoretical gap and propose a novel variant of the globalised SSN method which we abbreviate as \algname{LeAP-SSN} (\underline{Le}venberg--Marquardt \underline{A}daptive \underline{P}roximal \underline{S}emi\underline{s}mooth \underline{N}ewton, \cref{alg:proximal_newton})  that combines the best of both worlds. Namely, it allows us to prove global convergence rates in several settings, including the convex setting and a  nonconvex setting (with better rates under a Polyak--\L{}ojasiewicz (PL) condition), and to prove %
superlinear convergence under semismoothness and a Dennis--Moré-type condition.  In addition, in finite dimensions, we also prove that our algorithm identifies a $C^2$-manifold in which eventually all iterates lie, and use this and the theory of partial smoothness and active manifolds to prove superlinear convergence without requiring a Dennis--Moré assumption. We conclude the paper with some numerical experiments.

Our work, in a sense, lies in the interface of research questions studied in the (PDE-constrained) nonsmooth optimisation and (smooth) convex optimisation communities, and we hope that this paper can start an avenue for other global convergence rate results for SSN and its variants. 

\subsection{Contributions}
On a high level, for the problem \eqref{eq:opt}, we propose a SSN-type algorithm \algname{LeAP-SSN} (see \cref{alg:proximal_newton}) with guaranteed global convergence rates and local superlinear convergence. Our approach employs  proximal steps and a Levenberg--Marquardt globalisation strategy with backtracking inspired by \cite{doikov2024super}, that avoids relying on second-order semismoothness assumptions as required in \cite{potzl2022second,potzl2024inexact,wachsmuth2025globalized}. Unlike classical trust-region methods, our globalisation mechanism ensures a practical and simple implementation while retaining theoretical guarantees. 
In each iteration, our algorithm constructs a model $\widetilde{F}$ of the objective $F$ combining a quadratic model of $f$, the functional $\psi$, and a quadratic regularisation. Our main assumption on these subproblems, which holds even in some cases when $F$ is nonconvex, is that it is possible to find a stationary point of $\widetilde{F}$ that does not increase $\widetilde{F}$, see \cref{ass:subproblem} for a formal statement. 
We also introduce a mild assumption on $f'$, see \cref{ass:main_inequality}, which is key to global convergence rates, and which holds for a large class of functions and settings. 
Under \cref{ass:main_inequality}, we obtain a series of global and local nonasymptotic convergence rates:
\begin{enumerate}[label=(\arabic*)] 
    \item In the nonconvex case, we prove that the sequence $(x_k)_{k \in \mathbb{N}}$ generated by our algorithm from arbitrary $x_0 \in \dom F$ satisfies\footnote{Here and below $a_k=\mathcal{O}(b_k)$ means that there exists an explicit constant $c$ such that $a_k\leq cb_k$ for all $k\geq 0$. 
    } $\min_{0\leq i \leq k-1}\norm{F'(x_{i+1})}{*}=\mathcal{O}(1\slash \sqrt{k})$, see \cref{thm:global_convergence_sublinear_nonconv}.  
    \item If additionally $F$ satisfies a Polyak--Łojasiewicz (PL) condition (see \cref{ass:PL}), we prove that $x_k$ converges to a global solution $x^*$ of \eqref{eq:opt} with a linear rate from arbitrary $x_0 \in \dom F$. Moreover, $F(x_{k})-F(x^*)$ converges to $0$ linearly, see \cref{thm:global_convergence_nonconv_PL}.
    \item If $F$ is convex, we prove that for arbitrary $x_0 \in \dom F$ our algorithm guarantees $F(x_{k})-F^*=\mathcal{O}(1\slash k)$, where $F^*$ is the global minimal value in \eqref{eq:opt}, see \cref{thm:global_convergence_sublinear}.
    \item If for some isolated local minimum $x^*$ in \eqref{eq:opt}, \cref{ass:main_inequality} holds only locally around $x^*$, the PL  and quadratic growth (QG) conditions (see \eqref{eq:QG_local_1}) also hold locally around $x^*$, and the model $\widetilde{F}$ is strongly convex, we prove that  $x_k \to x^*$ and $F(x_{k})-F(x^*) \to 0$ linearly provided that $x_0$ is sufficiently close to $x^*$, see \cref{thm:local_convergence_1}. In other words, we prove that our algorithm has local linear convergence. 
\end{enumerate}
When additionally $f'$ has higher regularity and $F$ is locally convex, we show that the above bounds asymptotically improve. To be precise, the linear convergence in the above results improves to superlinear in each of the following cases:
\begin{enumerate}[label=(\arabic*)] 
\setcounter{enumi}{4}
\item If $\HH$ is finite-dimensional and $f$ is $C^2$ (i.e., twice continuously differentiable) 
around $x^*$, %
see \cref{thm:global_convergence_c2}.
\item If $f'$ is semismooth at $x^*$, and a  Dennis--Moré-type condition is satisfied, %
see \cref{thm:global_superlinear} and \cref{thm:local_superlinear}.
\item If $f$ is $C^2$ and convex, %
see \cref{cor:c2_convex_case}.
\item If $\HH$ is finite-dimensional, $f'$ is semismooth around $x^*$, $F$ is partially smooth at $x^*$ relative to a $C^2$ manifold, and a non-degeneracy (or strict complementary) condition is satisfied, %
see \cref{thm:fast_convergence_partial_smoothness}. %
\end{enumerate}
Thus, in particular, we obtain local superlinear convergence of our algorithm as for classical SSN, see \cref{thm:local_superlinear}.  The critical ingredient for these superlinear convergence results is to show that the regularisation parameter adaptively chosen in the algorithm asymptotically goes to zero, meaning that asymptotically, our method becomes classical SSN. 
As a result, we show that semismoothness allows for obtaining stronger local and global asymptotic convergence properties. Our results are interesting also in finite dimensions since they establish faster asymptotic convergence under higher regularity. The above global results are summarised for convenience in \cref{tab:table_results}. 
\begin{table}[ht]
\centering
\begin{tabularx}{\textwidth}{bs}%
\midrule
\textbf{Assumption} & \textbf{Global convergence rate %
}  \\
\midrule
- & $\mathcal{O}(1\slash \sqrt{k})$ for $\min_{i \leq k}\norm{F'(x_{i})}{*}$ \\
$f \in C^2$, $\dim(\HH) < \infty$, $x_k \to x^*$ & $\so(1\slash \sqrt{k})$ for $\min_{i \leq k}\norm{F'(x_{i})}{*}$ \\
$f$ convex &$\mathcal{O}(1\slash k)$ for $F(x_k)$\\
$f$ convex,  $f \in C^2$, $\dim(\HH) < \infty$, $x_k \to x^*$ & $\so(1\slash k)$ for $F(x_k)$\\
PL  & Linear for $x_k$ and $F(x_k)$\\
PL, $f \in C^2$, $\dim(\HH) < \infty$, $x_k \to x^*$ & Superlinear\\
PL, $f \in C^2$, $f$ convex & Superlinear\\
PL, $f$ semismooth, DM-type condition, $F$ locally convex, $H \succeq 0$ & Superlinear\\
PL, $f$ semsimooth, partial smoothness, $F$ locally convex, $H \succeq 0$ & Superlinear\\
\bottomrule
\end{tabularx}
\caption{High-level summary of our main results under \cref{ass:basic,ass:main_inequality,ass:subproblem}.}
\label{tab:table_results}
\end{table}
We highlight that, crucially, our algorithm achieves  
\begin{itemize}
\item \textbf{superlinear convergence under the PL condition to \emph{non-isolated}\footnote{By a \emph{non-isolated} minimiser, we mean a minimiser that we do not assume to be isolated.} minima}, whereas most existing works guarantee superlinear convergence only under strong convexity at the limit point. 
Convergence to non-isolated solutions is particularly relevant in applications such as inverse problems, where non-isolated stationary points naturally arise (for instance, when the forward operator is non-injective, as is common in e.g.~medical image reconstruction). Local superlinear convergence in such settings is an active area of research; see, e.g., \cite{rebjock2024fast} for recent results in the smooth case. To the best of our knowledge, no existing work establishes \emph{superlinear} convergence \emph{to non-isolated solutions} in the setting we consider. 

\item \textbf{automatic strong convergence of the iterates in Hilbert spaces under a PL condition}, without requiring additional assumptions commonly found in the literature, such as compactness or weak-to-strong continuity of \( f' \). While such results are expected based on prior works on the Kurdyka–Łojasiewicz (KL) inequality (e.g., \cite{attouch2013convergence,frankel2015splitting}), these typically rely on non-vanishing step sizes, i.e., they exclude the case \( \lambda_k \to 0 \).

    \item \textbf{active manifold identification in finite time}, that is, in finite dimensions, all iterates of our algorithm eventually lie on a $C^2$-manifold where we can harness partial smoothness and active manifold theory to obtain \emph{automatic satisfaction of a Dennis--Moré-type (DM-type) condition in an entirely nonsmooth setting} and deduce superlinear convergence. %
    
    \item \textbf{the obtainment of the above results without assuming Lipschitz Hessians}, which is a standard assumption made in the literature on global rates for second-order and quasi-Newton methods (see the literature review in the next subsection).
    
    \item \textbf{adaptivity to parameters and function classes}, in the sense that our method is agnostic to knowledge about convexity, validity of the PL condition, semismoothness, etc., for the specific problem being tackled. %
\end{itemize}
We close this subsection by noting that we implement and apply \algname{LeAP-SSN} to total variation image restoration, support vector machine classification, and Signorini contact problems in \cref{sec:numerics}. The implementation is publicly available \cite{github:leapssn}.

\subsection{Related literature}
\paragraph{Related work on globalisation of SSN.}  
The globalisation of semismooth Newton methods is not a new endeavour. Numerous works have focused on globalising both classical and nonsmooth Newton-type methods, however, the majority of these approaches are restricted to the finite-dimensional setting and typically rely on linesearch procedures to ensure global convergence without proving rates. Though we are not aware of any existing approaches that establish global convergence rates for semismooth Newton methods in our setting, let us briefly review the works most closely related to ours. 
\begin{enumerate}[label=(\arabic*)] \itemsep=0em
\item  \cite{potzl2022second} (and the followup \cite{potzl2024inexact} with inexactness) proposes a similar regularised approach; however, the analysis there relies on the notion of second-order semismoothness (which our work does not need) and %
essentially requires $F$ to be strongly convex to establish fast local convergence, %
    and does not establish global convergence rates. %
    \item   \cite{Alphonse} includes an inexact  globalised semismooth Newton method to tackle fixed-point equations in Banach spaces with an emphasis on equations arising from the obstacle problem; the globalisation there is via the Banach fixed point theorem.  \cite{wachsmuth2025globalized} addresses a strongly convex optimal control problem and employs an Armijo-type line search strategy, and also uses second-order semismoothness for its local superlinear result. %
    Armijo linesearch methods have the advantage of requiring only a single linear system solve per iteration, even during backtracking. Moreover, it can be shown that full Newton steps are eventually accepted, leading to superlinear convergence under semismoothness without requiring additional assumptions such as a Dennis--Moré condition. Note that neither work establishes global rates. 
\item The work in \cite{pang1995globally} (see also \cite{facchinei2003finite,cui2021modern}) was among the first to globalise a semismooth Newton-type method in a setting closely related to ours. They consider finite-dimensional problems of the form $\min_{x \in X} f(x)$, where $X$ is a polyhedron and $\nabla f$ is semismooth (also called an $SC^1$ function). Similar to our approach, their method employs a Levenberg--Marquardt--type regularization of the (Newton) Hessian of the form $\nabla^2_N f(x_k) + \mu_k I$, involving the Newton derivative $\nabla_N^2 f(x_k)$. The choice of $\mu_k$ is not specified beyond the assumption $\mu_k \to 0$, and they establish superlinear convergence under suitable regularity conditions.
\item \cite{khanh2024globally} and the recent book \cite{mordukhovich2024second} present results on minimizing  $C^{1,1}$ functions $f$ in finite dimensions, closely related to the present work. They use a Levenberg--Marquardt--type scheme which, instead of solving a regularised semismooth Newton system, tackles the inclusion $-\nabla f(x_k) \in (\partial^2 f(x_k) + \mu_k I )[d_{k}]$ to find the Newton direction $d_k$, where $\partial^2 f$ is the second-order subdifferential and $\mu_k = c \|\nabla f(x_k)\|$. This inclusion can sometimes be easier to solve than the semismooth Newton equation \cite{khanh2024globally}. Combined with an Armijo-type linesearch, they prove superlinear convergence under semismoothness at the limit point, without Dennis--Moré--type conditions. A notable difference is that $\mu_k \to 0$ follows automatically from their algorithm, while our method needs extra smoothness to ensure this. %
\end{enumerate}
None of the above works obtain global convergence rates for their versions of SSN where applicable. Moreover, they either consider the finite-dimensional or (strongly) convex setting, or utilise stronger assumptions like second-order semismoothness. 

\paragraph{Related work on global rates in the smooth, finite-dimensional setting.} There is a vast literature establishing global convergence rates for regularised Newton methods for finite-dimensional convex and nonconvex optimisation problems, with the main assumption being that the Hessian is Lipschitz. For the nonconvex setting the classical results are \cite{nesterov2006cubic,cartis2011adaptive}, where global rates are obtained for convex and nonconvex settings based on the cubic regularisation of the Newton method. The works \cite{nesterov2008accelerating,monteiro2013accelerated,gasnikov2019near,carmon2022optimal,kovalev2022first} further improve the rates by proposing accelerated versions in the convex setting. The Lipschitz Hessian condition was relaxed to a H\"older condition in \cite{grapiglia2017regularized,grapiglia2019accelerated}, where universal regularised Newton methods (including accelerated versions) were proposed. The above methods mainly use a regularisation of power greater than 2 in auxiliary subproblems in each iteration, which makes this subproblem problematic to solve in practice. The most recent advances are based on the use of quadratic regularisation akin to Levenberg--Marquardt, but with the regularisation parameter proportional to some power of the gradient norm. In this avenue, global convergence rates are obtained in the convex case for the Levenberg--Marquardt method in \cite{mishchenko2023regularized,doikov2024gradient} under a Lipschitz Hessian assumption and in \cite{doikov2024super} under either H\"older Hessians or a third-order derivative assumption. All the above papers focus on global rates in finite dimensions and either in the convex setting with H\"older Hessians or in the nonconvex setting, but with Lipschitz Hessians. They also do not consider the behaviour of their methods under any form of semismoothness. Despite the fact that these approaches are not directly applicable in our setting of semismooth derivatives and infinite dimensions, our inspiration largely comes from the paper \cite{doikov2024super} which introduces a universal Newton method that adapts to the H\"older exponent of various smoothness classes automatically. Note that assuming H\"older or Lipschitz Hessians is out of the question in our setting because it would imply, by a simple argument, that $f'$ is not only semismooth but differentiable.

\paragraph{Related work on Newton methods for non-strongly convex functions.}
Most works in the paragraph above rely on local strong convexity to obtain superlinear convergence, but recent research has focused on globalising Newton-type methods for nonconvex composite problems. In this setting, globalisation typically relies on Hessian regularisation, as in the present work. A central trend is replacing strong convexity with the \emph{Kurdyka--Łojasiewicz (KL)} inequality, a generalisation of the  PL condition \cite{polyak1963gradient}, which enables fast convergence even to non-isolated stationary points. Foundational results by Attouch et al.~\cite{attouch2009convergence,attouch2010proximal,attouch2013convergence} established KL-based convergence for descent methods. The first Newton-type extension \cite{frankel2015splitting} achieved linear convergence. In the convex case, Mordukhovich  et al.~\cite{mordukhovich2023globally} recently obtained superlinear convergence under metric subregularity, closely related to PL-type inequalities. Inspired by this line of work, Liu et al.~\cite{liu2024inexact,liu2025inexact} analyse a related composite setting, which is nonconvex, and establish superlinear convergence under the KL inequality and a local Hölder error bound, thereby allowing for non-isolated minima. Kanzow et al.~\cite{vom2024inexact} extends \cite{liu2024inexact} to a linesearch-free variant with a similar analysis.
To our knowledge, \cite{wang2025convergence,liu2025inexact} is the first to prove fast local convergence to \emph{non-isolated} stationary points in the PL/KL framework. Importantly, \emph{all of the above results are in finite dimensions} and \emph{rely critically on the Lipschitz continuity of the Hessian}.
 
\paragraph{Related work on globalisation via trust regions.}
Trust-region methods are classical and highly effective for smooth (\(C^2\)) minimisation \cite{conn2000trust,cartis2022evaluation}. 
In the setting of functions admitting a \emph{semismooth Newton derivative}, the majority of results concern the solution of semismooth equations coming from the corresponding first order optimality system \cite{ulbrich2011semismooth,facchinei2003finite}, with recent advances providing guarantees under the KL condition \cite{ouyang2025trust}.
For minimisation of functions with Newton differentiable derivatives the literature is still emerging with early work by Qi~\cite{qi1995trust} and recent extensions to minimisation on manifolds by Zhang et~al.~\cite{zhang2024riemannian}. Overall, trust-region schemes are robust to negative curvature, naturally accommodate inexact subproblem solves, and obtain global convergence and local superlinear convergence under different regularity assumptions, but are typically  designed for finite-dimensional settings, demand stronger assumptions and more elaborate implementations than our regularised Newton-type globalisation. 
Moreover, the question of global rates remains open. Incorporating trust-region mechanisms in the spirit of \cite{baraldi2023proximal} may further enhance our method; we leave this to future work.

\paragraph{Summary.} To conclude, to the best of our knowledge, \emph{no parameter-free globalisation strategy for the minimisation of functions with Newton differentiable derivatives on Hilbert spaces exists that achieves global sublinear rates together with a transition to superlinear convergence in the PL setting}. Our work closes this gap.

\subsection{Organisation of the paper}
The paper is organised as follows. In \cref{sec:technical_prelim} we introduce basic notation and definitions, our main assumptions and preliminary results. In \cref{sec:algorithm}, we formulate our algorithm \algname{LeAP-SSN} and detail some inequalities that come in use in the rest of the paper. We prove the above-mentioned global convergence results, local linear convergence, and superlinear convergence results in \cref{sec:global rates}, \cref{sec:local convergence}, and \cref{sec:fast}, respectively. Finally, we conclude with some numerical experiments in \cref{sec:numerics}.

\section{Technical preliminaries}\label{sec:technical_prelim}
We start with some notation and necessary definitions. The dual to the space $\HH$ is denoted by $\HH^*$. We use $\|\cdot\|$ to denote the norm on $\HH$ and $\|\cdot\|_*$ to denote the norm on $\HH^*$. We write $\langle g, x \rangle$ to denote the duality pairing of $g\in \HH^*$ and $x \in \HH$. By $\mathcal{R}\colon \HH \to \HH^*$ we mean the Riesz map. %
For a point $x\in \HH$ and a closed set $A \subseteq \HH$, we denote ${\rm dist}(x,A)\coloneqq \inf_{u\in A}\|x-u\|_{\HH}$. Given $x\in \HH$, we set $B_R(x)\coloneqq \{u\in \HH:\|x-u\|\leq R\}$ to be the closed ball of radius $R$ around $x$. For two points $x,y\in \HH$, we denote $[x,y]\coloneqq \{tx+(1-t)y : t \in [0,1]\}$.
For a function  $G \colon \HH \to \overline{\mathbb{R}}\coloneqq \R \cup \{\+\infty\}$, we denote its effective domain by $\dom G\coloneqq \{x\in \HH: G(x)<+\infty\}$. 
For the function $F$ in \eqref{eq:opt} and $x_0\in \dom F$, we denote the sublevel set of $F$ as 
\begin{equation}
    \label{eq:sublevel_set_def}
    \F_0\coloneqq \{ x \in \dom F : F(x) \leq F(x_0)\}.
\end{equation}
For a proper convex function $\psi \colon \HH \to \overline{\mathbb{R}}$, we denote by $\partial \psi$ its convex subdifferential, i.e., for any $x \in \dom \psi$,
\[
\partial \psi(x) \coloneqq  \{ h \in \HH^* : \psi(x)-\psi(y) \leq \langle h, x-y \rangle, \quad \forall y \in \HH\}.
\] 
We also use $\psi'(x)$ to denote an element of $\partial \psi(x)$.
For a function $F\colon \HH \to \overline{\mathbb{R}}$, we denote by $\hat\partial F$ the Fréchet subdifferential
\[
\hat \partial F(x) \coloneqq  \left\{ g \in \HH^* : \liminf_{u \to x} \frac{F(u)-F(x)-\langle g, u-x\rangle}{\norm{u-x}{}} \geq 0\right\}.
\]
The limiting (Mordukhovich) subdifferential $\partial_L F(x)$ at a point $x \in \mathrm{dom}(F)$ is defined as the set of elements $p \in \HH^*$ for which there exist sequences $(x_k)_{k \in \mathbb{N}} \subset \HH$ and $(p_k)_{k \in \mathbb{N}} \subset \HH^*$ such that $p_k\in \hat{\partial} F(x_k)$ for all $k \in \mathbb{N}$ and the following limits hold: %
\begin{equation*}
    x_k \to x, \qquad p_k \overset{\ast}{\rightharpoonup} p, \qquad F(x_k) \to F(x) \qquad \text{as $k \to \infty$.}
\end{equation*}
Note that if $F$ is convex, we have $\partial_L F(x) = \hat \partial F(x) = \partial F(x) $ \cite[Proposition 1.2]{Kruger}.  

\noindent For a function $f\colon \HH \to \overline{\mathbb{R}}$ that is $C^{1}$ on $\dom f$, i.e., continuously Fréchet differentiable, we denote by $f'(x)\in \HH^*$ its Fréchet derivative at $x\in \dom f$. %
\begin{remark}[On the setting of \eqref{eq:opt}]
If $F = f + \psi$ with $f$ continuously Fr\'echet differentiable and $\psi$ convex, the sum rule \cite[Proposition~1.107 (ii)]{Mordukhovich1} yields, for all $x \in \dom(F)$,
\[
\partial_L F(x) = f'(x) + \partial_L \psi(x) =  f'(x) + \hat{\partial} \psi(x) = \hat{\partial} F(x).
\]
In other words, the limiting subdifferential and the Fréchet subdifferential coincide in this setting. Thus, throughout the paper, we write $\partial F(x)$ for the limiting subdifferential, which coincides with the Fr\'echet subdifferential in this setting. We also use $F'(x)$ to denote an element of $\partial  F(x)$.
\end{remark}
\noindent Note that, by Fermat's rule \cite[Proposition 1.114] {Mordukhovich1}, we have at a local minimum $\bar x \in \dom\psi$ that 
\[0 \in \partial(f+\psi)(\bar x) = f'(\bar x) + \partial \psi(\bar x).\]
For a function $f \in C^{1,1}(\dom f, \overline{\mathbb{R}})$, i.e., $f \in C^{1}$ and has Lipschitz Fréchet derivative,  we say that $f'$ is \textit{Newton differentiable} (or generalised differentiable) at a point $x\in \dom f$ if there exists an open subset $U\subset \dom f $ with $x \in U$ and a family of maps $H\colon U \to \mathcal{L}(\HH, \HH^*)$ such that 
\begin{equation}
\label{eq:semismoothness_def}
    \lim_{h \in \HH, h \to 0} \frac{\norm{f'(x+h)-f'(x)-H(x+h)h}{*}}{\norm{h}{}} = 0.
\end{equation}
(Note the difference to the Fréchet derivative, where $H(x+h)$ above would instead be $H(x)$.) We also say that $f'$ is Newton differentiable at $x$ with respect to $H$. If $f'$ is Newton differentiable for every $x\in U$, we say that it is Newton differentiable on $U$. We slightly abuse notation and use $H(x)$ to denote the corresponding linear operator from $\HH$ to $\HH^*$ which we call the Newton or generalised derivative of $f'$ at $x$.  Finally, we adopt---by slight abuse of notation, but consistent with the convention in finite-dimensional optimisation---the shorthand
\begin{equation}
    H(x) \succeq \mu I 
    \;\; \underset{\mathrm{def}}{\Longleftrightarrow}\;\;
    \langle (H(x) - \mu \mathcal{R})d, d \rangle \geq 0 
    \qquad \forall d \in \HH,
\end{equation}
where $\mu \in \mathbb{R}$ and, as before, $\mathcal{R} \colon \HH \to \HH^*$ denotes the Riesz map.
\begin{remark}
\label{rem:semismooth}
A few words are in order regarding our definition of semismoothness.
\begin{itemize}
    \item To be as general as possible, we do not specify which particular generalised derivatives $H(x)$ are used in the above definition. For example, choosing a single-valued selection of Clarke’s Jacobian as generalised derivative $H(x)$, we obtain the classical definition of semismooth functions. Moreover, if $f'$ is semismooth, every single-valued selection of Clarke’s Jacobian works as $H(x)$ in the above definition. Following a tradition \cite{hintermuller2002primal, MR1972649}, we still use the word ``semismoothness" to denote that $f'$ satisfies the above definition. Yet, we underline that our paper covers a much more general setting of Newton differentiability that extends the notion of differentiability even further than semismooth functions.  

\item In some communities, %
Newton differentiability is often defined in terms of set-valued derivatives. Indeed, a map $T\colon X \to Y$ is said to be Newton differentiable with Newton derivative $G\colon X \rightrightarrows \mathcal{L}(X,Y)$ if
\[\lim_{h \to 0} \sup_{M \in G(x+h)}\frac{\norm{T(x+h)-T(x)-Mh}{}}{\norm{h}{}} =0.\]
The definition of Newton derivative we use in our work amounts to taking a selection from the set-valued $G$ at each point. Thus, our Newton derivative can be viewed as an abuse of notation where we pick a realisation of one of the elements of the set $G$ at each point. In this sense, our algorithm and results can handle set-valued derivatives. 
\end{itemize}
We refer to \cite{RickmannHerzogHeberg} for a survey relating different notions of generalised derivatives.
\end{remark}

\subsection{Main assumptions}
In this subsection, we present and discuss the main assumptions on the regularity of $F$ in \eqref{eq:opt} that are used to obtain our theoretical results. We emphasise that we do not need all assumptions in all results: we will refer specifically to the assumptions that are needed in each result. However for convenience, we introduce the important assumptions here.

The following structural assumption about problem \eqref{eq:opt} is used for all the results of this paper and thus one should bear in mind that it is in effect throughout.
\begin{ass}
\label{ass:basic}
\begin{enumerate}[label=(\roman*)]\itemsep=0em
    \item The function $\psi\colon \HH \to \overline{\mathbb{R}}$ is convex, proper, and lower semicontinuous.
    \item The function $f\colon \HH \to \overline{\mathbb{R}}$ is continuously Fréchet differentiable on $\dom f$.
    \item The function $F\coloneqq f+\psi$ is bounded from below by the global minimal value $F^*$ and has a nonempty solution set $S\coloneqq \{x^* \in \HH: F(x^*)=F^*\}$.
\end{enumerate} 
\end{ass}
Note that, in general, we do not assume $f$ and $F$ to be convex. 
The next assumption is fundamental and is used to prove global convergence rates of our algorithm.
\begin{ass}
\label{ass:main_inequality}
There exists a constant $L\geq 0$ such that at any point $x\in \dom f$ it is possible to evaluate a symmetric operator $H(x) \in \mathcal{L}(\HH, \HH^*)$ such that %
$f'$ satisfies
\begin{equation}
\label{eq:main_inequality}
    \norm{f'(y) -  f'(x) - H(x)(y-x)}{*} \leq L\norm{x-y}{}, \quad \forall x,y \in \dom f.
\end{equation}
\end{ass}
We assumed symmetricity of $H$ for simplicity, but the theory below would still work with relevant modifications (for example, $H(x)$ in \eqref{eq:prox_step_alg_FO} would need to be replaced with $\frac 12(H(x) + H(x)^*)$). 
Let us give some typical scenarios to illustrate the connection between $H$ and $f$:
\begin{itemize}
\item \textbf{If $f$ is $C^2$}: the smooth case when $f$ is twice continuously differentiable offers the natural choice $H = f''$. 
\item \textbf{If $f'$ is semismooth}: if $f'$ is semismooth or Newton differentiable, we can use its Newton derivative as $H$.
\item \textbf{The quasi-Newton setting}: here an approximation $H$ of the Hessian (or Newton derivative) is used for implementation purposes. Typical assumptions in the literature include $f''$ being Lipschitz continuous and constructing $H$ positive semi-definite.
\end{itemize}
We remark that an assumption similar to \cref{ass:main_inequality} is made in \cite[\S 4]{YamakawaYamashita} in the $C^2$ setting. Moreover, it can be seen as a nonasymptotic counterpart of $0\textsuperscript{th}$-order semismoothness \cite[Definition 2.6]{hintermuller2010semismooth}.  To further motivate \cref{ass:main_inequality}, note that the literature on the analysis of SSN and its variants often requires (see e.g. \cite[\S 1]{potzl2022second}, \cite[Assumption 1]{wachsmuth2025globalized}, \cite[Assumption 3.6]{MR3812973}, \cite[Assumption 2.2]{Alphonse}) that the operator $H(x)$ is uniformly bounded, i.e., for some $M>0$,
\begin{align}
\label{eq:bounded_Hessian}
    \norm{H(x)}{\mathrm{op}}\coloneqq \sup_{h\in \HH, h \ne 0} \frac{\|H(x)h\|_*}{\|h\|}\leq M.
\end{align} 
As the next lemma shows, \cref{ass:main_inequality} follows from \eqref{eq:bounded_Hessian}, and thus the former can be seen as a relaxation of the latter.%
\begin{lem}\label{lem:relation_bdd_hess_and_main_inequality}
Let \eqref{eq:bounded_Hessian} hold. 
\begin{enumerate}[label=(\roman*)]\itemsep=0em
\item  If $f'$ is Lipschitz with constant $L_1$, then, \cref{ass:main_inequality} holds with constant $M+L_1$.
\item If $f'$ is Newton differentiable with respect to $H$, %
then, \cref{ass:main_inequality} holds with constant $2M$.
\end{enumerate}
\end{lem}
\begin{proof}
In case (ii), by \cite[Proposition 2.3]{MR4431970}, it holds that $f'$ is Lipschitz with constant $L_1=M$.
The rest of the proof follows by the triangle inequality:
\[
\norm{f'(y) -  f'(x) - H(x)(y-x)}{*}\leq \norm{f'(y) -  f'(x)}{*} +   \norm{H(x)(y-x)}{*}\leq (L_1+M)\norm{x-y}{}.
\]
\end{proof}
Finally, we also remark that many semismooth functions appearing in applications have bounded Newton derivative %
and, hence, satisfy \cref{ass:main_inequality}. 

For some of the results, we use the following relaxation of strong convexity known as the Polyak--\L{}ojasiewicz (PL) condition/inequality, which does not even require $F$ to be convex. We use it in two contexts: sometimes globally for obtaining certain global convergence rates and locally for obtaining local convergence rates. The assumption for our global results is formulated as follows. %
\begin{ass}
\label{ass:PL}
$F$ satisfies the Polyak--\L{}ojasiewicz (PL) condition on $\dom F$, i.e., for some $\mu >0$,
\begin{equation}
    \label{eq:PL}
        \frac{1}{2\mu} {\rm dist}(0, \partial F(x))^2 \geq F(x)-\min_{x\in \dom F} F(x), \quad \forall x \in \dom F.
\end{equation}
\end{ass} 
Originally introduced in \cite{polyak1963gradient}, the PL inequality has since been widely employed to analyse the linear convergence of first-order methods (see, e.g., \cite{karimi2016linear}) and was later extended to nonsmooth and nonconvex settings under the broader framework of the Kurdyka--Łojasiewicz (KL) inequality, see \cite{attouch2009convergence,attouch2010proximal,attouch2013convergence} and the discussion in \cite{bento2025convergence} for a historical account. 
Another relaxation of strong convexity that does not require $F$ to be convex that we will use for one of our results on local convergence rates is the quadratic growth (QG) condition $\frac{\mu}{2} \dist(x,S)^2 \leq F(x) -\min_{x\in U} F(x)$.%

Previous works on local and global convergence of SSN and its variants either use a strong convexity assumption \cite{potzl2022second,wachsmuth2025globalized} or that  the inverse of the Newton derivative is bounded from above \cite{hintermuller2010semismooth,ulbrich2011semismooth},
which is a similar assumption to strong convexity since in that case the minimum eigenvalue of the Newton derivative is positive. 
For results on the relationship between strong convexity (SC), PL and QG conditions, we refer to \cite{liao2024error} where in finite dimensions it is shown that SC implies PL which in turn implies QG for weakly convex functions and if additionally $F$ is convex, PL and QG are equivalent but still weaker than SC; to \cite{apidopoulos2022convergence}, where the same is shown in infinite dimensions but with $\psi \equiv 0$; to \cite{bolte2017error} where the equivalence between PL and QG is established in infinite dimensions, but for convex functions; to \cite{jin2023growth} where for nonconvex functions and infinite dimensions it is shown that PL implies QG. Thus, our results are stronger than previously known since we use the weaker assumptions of PL and QG instead of SC as in other works.

Our next assumption is on the semismoothness of $f'$ that is used exclusively to prove superlinear convergence of our algorithm.
\begin{ass}
\label{ass:semismooth}
The Fréchet derivative $f'$ is semismooth (i.e., Newton differentiable)
at a local solution $x^*$ to problem \eqref{eq:opt}  with respect to $H$ (see \eqref{eq:semismoothness_def}). %
\end{ass}

We would like to underline that, despite \cref{ass:main_inequality,ass:PL} being stated above as global assumptions, it is possible to prove some of our results by making these assumptions only locally on some ball $B_R(x^*)$ around a local solution $x^*$. These local counterparts of \cref{ass:main_inequality,ass:PL} are simply obtained via changing $\dom f$ and $\dom F$ to $B_R(x^*)$.

\subsubsection{Further results related to \texorpdfstring{\cref{ass:main_inequality}}{Assumption \ref{ass:main_inequality}}}
Let us briefly discuss some basic results related to \cref{ass:main_inequality}.  We first see how in some cases, \cref{ass:main_inequality} and \eqref{eq:bounded_Hessian} are equivalent, cf. \cref{lem:relation_bdd_hess_and_main_inequality}.
\begin{lem}\label{lem:ass_2_and_4_equiv}
Let $\dom f = \mathcal{H}$, and $f'$ be Lipschitz with constant $L_1$. If \cref{ass:main_inequality} holds, then, \eqref{eq:bounded_Hessian} holds with constant $L+L_1$. %
In fact, in this case, %
\cref{ass:main_inequality} and \eqref{eq:bounded_Hessian} are equivalent.
\end{lem}
\begin{proof}%
Using \cref{ass:main_inequality} and the triangle inequality, we obtain   $\norm{H(x)(y-x)}{*} \leq L\norm{x-y}{} + \norm{f'(y)-f'(x)}{*}.$ Taking $y=x+tz$ for arbitrary $z, x \in \HH$ and $t>0$, we further obtain
\[
\norm{H(x)z}{*} \leq L\norm{z}{} + \frac{\norm{f'(x+tz)-f'(x)}{*}}{t}.
\]
Taking $t=1$ and using the Lipschitz continuity of $f'$, we directly obtain $\norm{H(x)z}{*} \leq (L+L_1)\norm{z}{}$. Hence \eqref{eq:bounded_Hessian} is satisfied with $M=L+L_1$.
The equivalence follows from
\cref{lem:relation_bdd_hess_and_main_inequality}, which gives the reverse direction and that \cref{ass:main_inequality} holds with constant $L=M+L_1$.
\end{proof}
We will use the following basic lemma, which holds not only under \cref{ass:main_inequality}, but also when the inequality \eqref{eq:main_inequality} from the assumption holds only on line segments. Below, recall the notation $[x,y]\coloneqq \{tx+(1-t)y :  t \in [0,1]\}$. %
\begin{lem}
\label{lem:upper_bound_f}
Let $x,y\in {\rm dom} f$ and suppose that
\begin{equation}
\label{eq:for_lemma_semismooth_main_ass}
    \norm{f'(\tilde y) -  f'(\tilde x) - H(\tilde x)(\tilde y-\tilde x)}{*} \leq L\norm{\tilde x-\tilde y}{}, \quad \forall \tilde x, \tilde y \in [x,y].
\end{equation}
Then, we have 
    \begin{equation}
    \label{eq:upper_bound_f}
        \left|f(y)-[f(x)+\langle f'(x),y-x \rangle + \frac 12 \langle H(x)(y-x),y-x \rangle] \right| \leq \frac{L}{2} \norm{x-y}{}^2.
    \end{equation}
\end{lem}
\begin{proof}
    Indeed,
    \begin{align*}
        &\left|f(y)-[f(x)+\langle f'(x),y-x \rangle + \frac 12 \langle H(x)(y-x),y-x \rangle] \right| = \left|\int_0^1 \langle  f'(x+t(y-x)) -  f'(x) - tH(x)(y-x), y-x \rangle \dt\right|\\
        &\leq \int_0^1 \norm{f'(x+t(y-x)) -  f'(x) - tH(x)(y-x)}{*}\norm{y-x}{} \dt \overset{\eqref{eq:for_lemma_semismooth_main_ass}}{\leq} \int_0^1 t L\norm{y-x}{}^2 \dt = \frac{L}{2} \norm{x-y}{}^2.
    \end{align*}
\end{proof}

\subsection{Regularised Newton step}
Our algorithm relies on the proximal Levenberg--Marquardt Newton step defined for some $x \in {\rm dom} F$ and constant $\lambda \geq 0$ as
\begin{equation}
    \label{eq:PLMSN_step}
    x_+(\lambda,x)\coloneqq \argmin_{y \in \HH} \left\{\widetilde{F}(y;x,\lambda)\coloneqq f(x)+\langle f'(x), y-x \rangle + \frac 12 \langle H(x) (y-x), y-x \rangle + \frac{\lambda}{2}\norm{y-x}{}^2 + \psi(y)\right\}.\tag{$\mathrm{P}_{\mathrm{N}}$}
\end{equation}
For brevity, if it is clear from the context, we use $x_+$ for $x_+(\lambda,x)$. Our assumption about the properties of a solution $x_+$ to the above auxiliary problem is as follows. This assumption is in effect throughout the rest of the paper.
\begin{ass}\label{ass:subproblem}
 For all $\lambda \geq L$ (where L is from \cref{ass:main_inequality}) and all $x \in \dom F$ we can compute a stationary point $x_+ = x_+(\lambda;x) \in \HH$ of the model function $\widetilde{F}(y;x,\lambda)$ in \eqref{eq:PLMSN_step} that satisfies:
\begin{align}
        -f' (x) - H(x) (x_+ - x) -\lambda\mathcal{R}(x_+ - x)  &\in \partial \psi(x_+) &\text{ (first-order condition)}, \label{eq:prox_step_alg_FO}\\
        \widetilde{F}(x_+;x,\lambda) \leq \widetilde{F}(x;x,\lambda)&=F(x) &\text{ (model nonincrease)}. \label{eq:prox_step_alg_decr}
\end{align}
\end{ass}
Note that we do not require finding a global solution to \eqref{eq:PLMSN_step}. Also, it should be possible to relax the exact first-order condition \eqref{eq:prox_step_alg_FO} to allow for some inexactness, but we defer this to future research.
\begin{remark}
\label{rem:model-descent}
    Let us briefly discuss some cases in which \cref{ass:subproblem} holds.%
    \begin{itemize}
        \item If $f$ is convex and $H(x) \succeq 0$ for all $x \in \dom F$, then the minimisation problem in \eqref{eq:PLMSN_step} is strongly convex for any $\lambda > 0$ and, hence, always admits a global unique minimiser $x_+$ that satisfies \eqref{eq:prox_step_alg_FO}--\eqref{eq:prox_step_alg_decr}, independently of the constant $L$ in \cref{ass:main_inequality}. The global minimum can either be found explicitly if $\psi$ is simple or by some auxiliary algorithm for convex optimisation.
        \item If $f$ is $\mu$--weakly convex, i.e., $f + (\mu/2)\|\cdot\|^2$ is convex, and $H(x) \succeq -\mu I$ %
        for some $\mu \geq 0$ and all $x \in \dom F$, then \cref{ass:subproblem} still holds provided the constant $L>0$ in \cref{ass:main_inequality} satisfies $L \geq \mu  $ since in this case the model $\widetilde{F}(\cdot;x,\lambda)$ is convex for all $x \in \dom F$ and $\lambda \geq L \geq \mu$. Thus, again \eqref{eq:PLMSN_step} admits a global minimiser $x_+$ that satisfies \eqref{eq:prox_step_alg_FO}--\eqref{eq:prox_step_alg_decr}.
        The  condition $\mu \leq L$ can always be satisfied since inequality \eqref{eq:main_inequality} holds also for any $\tilde{L}\geq L.$
        \item If the term $\frac12 \langle H(x)(y-x),\, y-x \rangle$ in \eqref{eq:PLMSN_step} is nonconvex (and there is no such $\mu>0$ that $H(x) \succeq -\mu I$), but the non-regularised model $\widetilde{F}(y;x,0)$ is bounded from below for all $x \in \dom F$, then $\widetilde{F}(y;x,\lambda)$ is bounded from below for any $\lambda \geq 0$ and   $x \in \dom F$. In this case, one can run the proximal gradient method (cf.~\cite{beck2017first}) starting from $x$. This method is monotone with respect to $\widetilde{F}$ and guarantees finding an (approximate) stationary point, so we may expect that both \eqref{eq:prox_step_alg_FO} and \eqref{eq:prox_step_alg_decr} also hold in this setting. The model $\widetilde{F}(y;x,0)$ is bounded from below, for example, if $\dom \psi$ is compact or if $\psi$ grows faster than quadratically as $\|x\| \to \infty$.

        \item Both conditions \eqref{eq:prox_step_alg_FO} and \eqref{eq:prox_step_alg_decr} are satisfied if $\widetilde{F}(y;x,\lambda)$ satisfies the PL condition.
    \end{itemize}
    Note that this list is not exhaustive and there may be other situations where \cref{ass:subproblem} is satisfied.
\end{remark} 
To reiterate, not only does \cref{ass:subproblem} cover the convex setting often considered in other works, but also many nonconvex settings. We note also that
\cref{ass:subproblem} is actually needed only on the trajectory of the algorithm.

We next move to several technical results that characterise the properties of the step \eqref{eq:PLMSN_step}.
The first-order condition \eqref{eq:prox_step_alg_FO} means that there exists $\psi'(x_{+}) \in \partial \psi(x_+)$ such that
\begin{equation}
\label{eq:PLMSN_step_optimality}
    \psi'(x_{+}) = - f'(x) - H(x)(x_+ - x)  - \lambda\mathcal{R}(x_+ - x).
\end{equation}
Thus, defining 
\begin{equation}
    \label{eq:F_div_def}
    F'(x_+) \coloneqq  f'(x_+)+\psi'(x_+)  = f'(x_+) - f' (x) - H(x)(x_+ - x)  - \lambda\mathcal{R}(x_+ - x) ,
\end{equation}
we immediately see that $F'(x_+) \in \partial F(x_+)$.\\
The following lemma gives us control over the length of the step via the subgradient norm at a point $x$. It is essentially \cite[Lemma 1]{doikov2024super}, but we give a proof for completeness.
\begin{lem}%
\label{lem:inequality_lemma}
Let $x_+ = x_+(x,\lambda)$ be a stationary point of \eqref{eq:PLMSN_step}, where  $H(x) \succeq \mu I$ for some $\mu \in \mathbb{R}$. Let also $\lambda+\mu>0$. Then,
\begin{equation} 
    \norm{x_+-x}{} \leq  \frac{\norm{F'(x)}{*}}{\lambda + \mu}.\label{eq:rk_leq_gk_over_lambda_k_new}
\end{equation}
\end{lem}
\begin{proof}
Rearranging the first-order optimality condition \eqref{eq:PLMSN_step_optimality} and multiplying it by $x_+-x$, we have, setting $r=\norm{x_+-x}{},$
\[  
\langle H(x)(x_+ - x), x_+ - x \rangle = \langle -\psi'(x_{+}) - f'(x), x_+ - x\rangle - \lambda r^2. 
\]
Taking $\psi'(x) \in \partial\psi(x)$, we have by the assumption of the lemma and the monotonicity of $\partial \psi$,
\begin{align*}
\mu r^2=\mu\norm{x_+ - x}{}^2 {\leq} \langle H(x)(x_+ - x), x_+ - x \rangle &\leq \langle -\psi'(x)  -f'(x), x_+ - x\rangle  - \lambda r^2 \\
&\leq \norm{\psi'(x)+f'(x)}{*}r  - \lambda r^2.
\end{align*}
Rearranging, dividing by $r$ and $\lambda +\mu>0$, we get \eqref{eq:rk_leq_gk_over_lambda_k_new}.
\end{proof}
In the next result, we establish a sensitivity estimate of how the point $x_+(\lambda,x)$ varies for fixed $x$ when $\lambda$ changes.
This result will be fundamental for opening the path to show superlinear convergence.
\begin{lem}\label{lem:lambda_dependence}
Let $x \in \dom f$ be such that $H(x) \succeq 0$. Then, the following inequality holds for all $0 < \lambda \leq \lambda'$:
  \begin{equation}
       \|x_+(\lambda,x) - x_+(\lambda',x) \| \leq \frac{\lambda' - \lambda}{\lambda' } \| x -x_+(\lambda,x)\|.\label{eq:lambda_dependence}
  \end{equation}
\end{lem}
\begin{proof}
We denote for brevity $x_+(\lambda) \coloneqq  x_+(\lambda, x)$ and $\tilde{f}(y)\coloneqq  f(x)+\langle f'(x), y-x \rangle + \frac 12 \langle H(x) (y-x), y-x \rangle$. Setting $y_1 = x_+(\lambda)$ and $y_2 = x_+(\lambda')$, using \eqref{eq:PLMSN_step_optimality}, we have that
\begin{align*}
    \psi'(y_1) &= -  \tilde{f}'(y_1) - \lambda\mathcal{R}(y_1 - x), \quad
    \psi'(y_2)= -  \tilde{f}'(y_2) -\lambda'\mathcal{R}(y_2 - x),
\end{align*}
where $\psi'(y_1) \in \partial \psi(y_1)$ and $\psi'(y_2) \in \partial \psi(y_2)$. Multiplying the first equality by $y_1-y_2$, the second by $y_2-y_1$, and summing, we arrive at
\begin{align*}
    \langle \psi'(y_1) - \psi'(y_2), y_1-y_2 \rangle &= \langle -\tilde{f}'(y_1) - \lambda\mathcal{R}(y_1 - x), y_1-y_2\rangle + \langle  \tilde{f}'(y_2) + \lambda'\mathcal{R}(y_2 - x), y_1-y_2\rangle\\
    &= \langle  \tilde{f}'(y_2) + \lambda'\mathcal{R}(y_2 - x) -\tilde{f}'(y_1) - \lambda\mathcal{R}(y_1 - x) , y_1-y_2\rangle\\
    &=\langle  H(x)(y_2-y_1)  + \lambda'\mathcal{R}(y_2 - x) - \lambda\mathcal{R}(y_1 - x) , y_1-y_2\rangle,
\end{align*}
where in the last equality we used that %
$\tilde{f}'(y_2)-\tilde{f}'(y_1)=H(x)(y_2-y_1)$. Rearranging, using $H(x) \succeq 0$ and the monotonicity of $\partial \psi(x)$, we obtain
\begin{align*}
       0 
       &\leq\langle \psi'(y_1) - \psi'(y_2), y_1-y_2 \rangle  + \langle H(x)(y_1-y_2), y_1-y_2 \rangle\\
        &= \langle   \lambda'\mathcal{R}(y_2 - x) - \lambda\mathcal{R}(y_1 - x) , y_1-y_2\rangle\\
        &=\langle   \lambda'\mathcal{R}(y_2 - x) - \lambda'\mathcal{R}(y_1 - x) + \lambda'\mathcal{R}(y_1 - x) - \lambda\mathcal{R}(y_1 - x) , y_1-y_2\rangle \\
        &=-\lambda'\norm{y_1-y_2}{}^2+(\lambda'-\lambda) \langle \mathcal{R}(y_1-x),y_1 - y_2 \rangle \leq -\lambda'\norm{y_1-y_2}{}^2 + (\lambda'-\lambda) \norm{y_1-x}{}\norm{y_1 - y_2}{},
    \end{align*}
where for the last inequality we used Cauchy--Schwarz. Rearranging, %
we obtain \eqref{eq:lambda_dependence} as desired.
\end{proof}

The final result of this subsection will allow us to establish per-iteration decrease of $F$ by our algorithm.
\begin{lem}%
\label{lem:inequality_lemma_2}
Consider the update step \eqref{eq:PLMSN_step} and assume that there exists a function $\varphi: \R \to \R$ such that
\begin{equation}
    \label{eq:asmpt_phi}
    \norm{f'(x_+) -  f'(x) - H(x)(x_+-x)}{*} \leq \varphi(\norm{x-x_+}{}).
\end{equation}
Then, 
\begin{equation}
\label{eq:progress_prelim}
         \langle F'(x_+),x - x_+ \rangle
         \geq \frac{1}{2 \lambda}\norm{F'(x_+)}{*}^2 +  \frac{\lambda}{2} \norm{x - x_+}{}^2 - \frac{\varphi(\norm{x - x_+}{})^2}{2\lambda}.
    \end{equation}
\end{lem}

\begin{proof}
We have that
\begin{align*}
\varphi(\norm{x - x_+}{})^2 \overset{\eqref{eq:asmpt_phi}}&{\geq} \norm{f'(x_+) -  f'(x) - H(x)(x_+-x)}{*}^2 \overset{\eqref{eq:F_div_def}}{=} \norm{F'(x_+) + \lambda \mathcal{R}(x_+ - x)}{*}^2 \\
&= \norm{F'(x_+)}{*}^2 + 2\lambda\langle F'(x_+),x_+ - x \rangle + \lambda^2\norm{x - x_+}{}^2.
\end{align*} 
Dividing by $2 \lambda$ and rearranging, we obtain the desired inequality \eqref{eq:progress_prelim}.
\end{proof}

\section{The proposed algorithm}
\label{sec:algorithm} 
In this section, we present \algname{LeAP-SSN} --- our  Levenberg--Marquardt Adaptive Proximal Semismooth Newton method.
The algorithm is based on the regularised proximal step \eqref{eq:PLMSN_step} and in each iteration we use a simple linesearch to adaptively find the regularisation parameter $\lambda$ that guarantees sufficient progress of the algorithm. The resulting algorithm is listed as \cref{alg:proximal_newton} below. %
\begin{algorithm}[H]
\caption{\algname{LeAP-SSN} (\underline{Le}venberg--Marquardt \underline{A}daptive \underline{P}roximal \underline{S}emi\underline{s}mooth \underline{N}ewton method)}\label{alg:proximal_newton}
\begin{algorithmic}[1]
\State \textbf{input:} Parameters $\alpha \in (0,1/2]$, $\beta \in (0,(m-1)/(2m)]$ for some $m \geq 1$ (standard choice: $\alpha=1/2$, $\beta=1/4$, $m=2$), $x_0\in \dom F$, $\psi'(x_0) \in \partial \psi(x_0)$, and $\Lambda_0 >0$.  \label{step:input} 
\State \textbf{for} $k=0,1,\ldots,$ \textbf{do} \label{step:outer_loop}
\State \hspace{0.7cm} \textbf{for} $j_k = 0,1,\ldots$, \textbf{do} \label{step:inner_loop}
\State \hspace{1.4cm} Set $\lambda=2^{j_k}\Lambda_k$. \label{step:main_step}
\State \hspace{1.4cm} Attempt to compute
\begin{align}
        \qquad
        x_+ = \argmin_{y \in \mathcal{H}} \left\{f(x_k)+\langle f'(x_k), y-x_k \rangle + \frac 12 \langle H(x_k) (y-x_k), y-x_k \rangle + \frac{\lambda}{2}\norm{y-x_k}{}^2 + \psi(y)\right\}\label{eq:prox_step_alg}
\end{align}
\hspace{1.4cm} in the sense that $x_+$ satisfies \cref{ass:subproblem}. %
If not possible, go to next iteration of loop (\cref{step:inner_loop}). 
\State \hspace{1.4cm} Set $\psi'(x_+) = -f' (x_k) - H(x_k) (x_+ - x_k) -\lambda\mathcal{R}(x_+ - x_k) $. 
\State \hspace{1.4cm} Set $F'(x_+) = f' (x_+) + \psi'(x_+)$.
\State \hspace{1.4cm} \textbf{if} $x_+$ satisfies the acceptance conditions
\vspace{-3mm}
\begin{align}
\qquad \langle F'(x_+),x_k -x_+\rangle \geq \frac{\alpha}{\lambda}\norm{F'(x_+)}{*}^2 \quad\text{and}\quad
    F(x_k)-F(x_+)\geq \beta\lambda \norm{x_+-x_k}{}^2\label{eq:stop_cond_alg}
\end{align}
\vspace{-4mm}
\label{step:acceptance}
\State \hspace{2.1cm} Set $x_{k+1} = x_+$, $\lambda_k = \lambda=2^{j_k}\Lambda_k$, $\Lambda_{k+1}=\lambda_k/2 = 2^{j_k}\Lambda_k/2$.\label{step:k+1_update}
\State \hspace{2.1cm} Set $\psi'(x_{k+1})= \psi'(x_+)$, $F'(x_{k+1}) = f'(x_{k+1}) + \psi'(x_{k+1})$.
\State \hspace{2.1cm} \textbf{break} and go to next iteration of the outer loop (\cref{step:outer_loop}).
\State \hspace{1.4cm} \textbf{end if}
\State \hspace{0.7cm} \textbf{end for}
\State \textbf{end for}
\end{algorithmic}
\end{algorithm}

\begin{remark}[Solving \eqref{eq:prox_step_alg}]
Note that when $\psi \equiv 0$,  \eqref{eq:prox_step_alg} simplifies to computing
\[
x_+ = x_k-(H(x_k) + \lambda\mathcal{R})^{-1}f'(x_k).
\]
In other words, potentially after a discretisation, one requires a single linear system solve. This $x_+$ satisfies model nonincrease \eqref{eq:prox_step_alg_decr} under the various conditions discussed in \cref{rem:model-descent}.

When $\psi \ne 0$, $x_+$ can be expressed via the proximal operator of $\psi$ which in many cases is known explicitly \cite{chierchia2020proximity}. Alternatively, one may need to run a proximal gradient descent or another auxiliary optimisation method to compute $x_+$.
\end{remark}
\begin{remark}[The Riesz map and choice of norm]
\label{rem:Riesz}
For many problems considered in the context of finite-dimensional optimisation (like the example in \cref{sec:examples:svc}), the Riesz map $\mathcal{R}$ is the identity matrix $\mathcal{R} = I$ and $\norm{\cdot}{} = \norm{\cdot}{*} = \norm{\cdot}{\ell^2}$ are the Euclidean norm. However, in discretisations of infinite-dimensional problems, the Riesz map often takes the form of a matrix $A$ which is a combination of appropriate problem-dependent discretised operators. See \cref{sec:examples:signorini,sec:examples:image-restoration} for examples. In this case one defines $\norm{\cdot}{} = \norm{\cdot}{A}$ and $\norm{\cdot}{*} = \norm{\cdot}{A^{-1}}$. This ensures the correct scaling of the solution coefficient vectors to their respective representations as functions and often vastly improves solver performance \cite{schwedes2017mesh}.
\end{remark}

As we will show later, when proving superlinear convergence under semismoothness, we obtain $\lambda_k\to 0$, which means that our algorithm in a sense automatically interpolates between regularised proximal SSN and pure proximal SSN, and becomes closer to pure proximal SSN when $k\to \infty$. 

Our next goal is to show that the inner cycle to find an appropriate $\lambda_k$ is finite  . To that end, recall that we use $x_+(\lambda,x_k)$ to denote a solution to \eqref{eq:PLMSN_step} with $x=x_k$ and that  solution satisfies \eqref{eq:prox_step_alg_FO}--\eqref{eq:prox_step_alg_decr} of \cref{ass:subproblem}.
\begin{lem}[Acceptance and termination of the inner loop]
\label{lem:acceptance_inner_nonconv} 
Consider iteration $k \geq 0$ of \cref{alg:proximal_newton} applied to \eqref{eq:opt}. Let \cref{ass:basic,ass:subproblem}  hold and suppose that the following inequality holds:
\begin{equation}
        \norm{f'(\tilde y) -  f'(\tilde x) - H(\tilde x)(\tilde y-\tilde x)}{*} \leq L\norm{\tilde x-\tilde y}{}, \quad \forall \tilde x, \tilde y \in [x_k,x_+(\lambda,x_k)], \quad \forall \lambda\geq0.\label{eq:acc_lemma_ass_main_inequality}
\end{equation}
Pick $m \geq 1$, $0 < \alpha \leq 1\slash 2$, and $0<\beta \leq (m-1)\slash (2m)$.   Then, if $\lambda_k \geq mL$, both inequalities in \eqref{eq:stop_cond_alg} hold, i.e., the acceptance conditions of the inner loop of \cref{alg:proximal_newton} holds and the inner loop ends after a finite number of trials.
Furthermore, we have 
\begin{equation}
    \label{eq:lambda_bound}
    \lambda_k\leq \overline{\lambda}\coloneqq \max\{2mL,\Lambda_0\}, \quad \Lambda_{k+1}\leq \overline{\lambda}/2, \quad \forall k \geq 0.
\end{equation}
\end{lem}
Before we start the proof, let us remark that \eqref{eq:acc_lemma_ass_main_inequality} is of course implied by \cref{ass:main_inequality}.  We formulated the result with the localised version \eqref{eq:acc_lemma_ass_main_inequality} because we will need it for our local convergence results in a later section.
\begin{proof}[Proof of \cref{lem:acceptance_inner_nonconv}]
By \cref{ass:subproblem} (which is applicable since we assumed $\lambda_k \geq mL \geq L$), the subproblem~\eqref{eq:PLMSN_step} is well defined in the sense that it admits a stationary point  satisfying~\eqref{eq:prox_step_alg_FO}--\eqref{eq:prox_step_alg_decr}. 
By the assumption \eqref{eq:acc_lemma_ass_main_inequality} %
we have that in particular, the inequality \eqref{eq:asmpt_phi} holds with $x_+(\lambda,x_k)$ and $x_k$ instead of $x_+$ and $x$ respectively, and $\varphi(r)=Lr$. Thus, applying \cref{lem:inequality_lemma_2}, we obtain, for any $\lambda \geq mL \geq L$,
\begin{align*}
\langle F'(x_+(\lambda,x_k)),x_k - x_+(\lambda,x_k) \rangle
         \overset{\eqref{eq:progress_prelim}}&{\geq} \frac{1}{2 \lambda}\norm{F'(x_+(\lambda,x_k))}{*}^2 +  \frac{\lambda}{2} \norm{x_k - x_+(\lambda,x_k)}{}^2 - \frac{L^2\norm{x_k - x_+(\lambda,x_k)}{}^2}{2\lambda} \\
         &\geq \frac{1}{2 \lambda}\norm{F'(x_+(\lambda,x_k))}{*}^2,
\end{align*} 
where in the last inequality we used that $\lambda\geq mL \geq L$.
Thus, since $\lambda_k=2^{j_k}\Lambda_k$ is increasing when $j_k$ is increasing, we obtain that the first inequality in \eqref{eq:stop_cond_alg} holds after a finite number of trials when $\lambda_k\geq mL$ since $\alpha\leq 1/2$ by the assumptions of the lemma. 

We now move to the second inequality in \eqref{eq:stop_cond_alg}. 
By the assumption \eqref{eq:acc_lemma_ass_main_inequality}, %
we can apply \cref{lem:upper_bound_f} and obtain, for any $\lambda \geq mL $,
\begin{align*}
    F(x_+(\lambda,x_k)) &= f(x_+(\lambda,x_k)) + \psi(x_+(\lambda,x_k)) \nonumber\\
    \overset{\eqref{eq:upper_bound_f}}&{\leq} f(x_k)+\langle f'(x_k), x_+(\lambda,x_k)-x_k \rangle + \frac 12 \langle H(x_k) (x_+(\lambda,x_k)-x_k), x_+(\lambda,x_k)-x_k \rangle \nonumber\\
    & \hspace{1em}+ \frac{L}{2}\norm{x_+(\lambda,x_k)-x_k}{}^2  + \psi(x_+(\lambda,x_k))\nonumber\\
    \overset{\eqref{eq:PLMSN_step}}&{=}\widetilde{F}(x_+(\lambda,x_k);x_k,\lambda)-\frac{\lambda}{2}\norm{x_+(\lambda,x_k)-x_k}{}^2+ \frac{L}{2}\norm{x_+(\lambda,x_k)-x_k}{}^2 \nonumber\\
    \overset{\eqref{eq:prox_step_alg_decr}}&{\leq} F(x_k)+ \frac{L-\lambda}{2}\norm{x_+(\lambda,x_k)-x_k}{}^2 \nonumber\\
    \overset{L\leq \lambda/m}&{\leq}F(x_k)-  \frac{m-1}{2m}\lambda\norm{x_+(\lambda,x_k)-x_k}{}^2.%
\end{align*}
Thus, since $\lambda_k=2^{j_k}\Lambda_k$ is increasing when $j_k$ is increasing, by rearranging the last inequality, we obtain that the second inequality in \eqref{eq:stop_cond_alg} holds after a finite number of trials when $\lambda_k\geq mL$ since $\beta\leq \frac{m-1}{2m}$  by the assumptions of the lemma. 
Thus, both inequalities in \eqref{eq:stop_cond_alg} hold after a finite number of trials, and the step of \cref{alg:proximal_newton} is well defined. 

We now move to the proof of the first inequality in \eqref{eq:lambda_bound} by induction. Consider the base case, i.e., iteration $k=0$. If $j_0=0$, i.e., the acceptance condition holds for $\lambda_0=\Lambda_0$, we have that $\lambda_0\leq \max\{2mL,\Lambda_0\}=\overline{\lambda}$. If $j_0>0$, we in any case have that $\lambda_0 = \Lambda_0 2^{j_0}\leq 2mL$ since otherwise $\lambda_0/2 >mL$ would have already been accepted by the arguments in the first part of this lemma. Thus, again $\lambda_0\leq \max\{2mL,\Lambda_0\}=\overline{\lambda}$. Consider now iteration $k>0$ and proceed in the same way. If $j_k=0$, i.e., the acceptance condition holds for $\lambda_k=\Lambda_k$, we have by the induction hypothesis that $\lambda_k= \Lambda_k= \lambda_{k-1}/2 \leq\overline{\lambda}/2\leq \overline{\lambda}$. If $j_k>0$, we in any case have that $\lambda_k\leq 2mL$ since otherwise the previous iterate already satisfies $\lambda_k/2>mL$, which is sufficient for the acceptance conditions to hold by the first part of this lemma. Thus, again $\lambda_k\leq \overline{\lambda}$. The second inequality in \eqref{eq:lambda_bound} just follows by definition, since for any $k \in \mathbb{N}$, $\Lambda_{k+1}=\lambda_k/2\leq\overline{\lambda}/2$. 
\end{proof}
\subsection{Useful notation and inequalities}

Before we move to the main results, we introduce some technical preliminaries which are used throughout the rest of the paper. To simplify the derivations, we introduce the following notation for $k\geq 0$:
\begin{equation}
\label{eq:F_r_g_notaion}
    F_k\coloneqq F(x_k)-F^*, \qquad r_k\coloneqq \norm{x_k-x_{k+1}}{}, \qquad g_k\coloneqq \norm{F'(x_k)}{*}.
\end{equation}
Then, using the definition of the point $x_{k+1}$ in \cref{step:k+1_update} of \cref{alg:proximal_newton}, the acceptance conditions \eqref{eq:stop_cond_alg} can be written as 
\begin{align}   
\langle F'(x_{k+1}),x_k -x_{k+1}\rangle &\geq \frac{\alpha}{\lambda_k}g_{k+1}^2, \label{eq:stop_cond_alg_new_1} \\
    F_k-F_{k+1}=F(x_k)-F(x_{k+1})&\geq \beta\lambda_kr_k^2.\label{eq:stop_cond_alg_new_2}
\end{align}
We will use the following implications of the acceptance conditions \eqref{eq:stop_cond_alg}.
\begin{lem}\label{lem:functionValuesGradientRelation}
Let, at iteration $k\geq 0 $ of \cref{alg:proximal_newton}, acceptance criteria \eqref{eq:stop_cond_alg} hold. Then, we have 
    \begin{align}
    F^*\leq F(x_{k+1})\leq F(x_k) &\qquad\text{or equivalently}\qquad 0 \leq F_{k+1}\leq F_k, \label{eq:F_k_decr}\\
    \norm{F'(x_{k+1})}{*} \leq \frac{\lambda_k}{\alpha}\norm{x_k-x_{k+1}}{} &\qquad\text{or equivalently}\qquad g_{k+1} \leq \frac{\lambda_kr_k}{\alpha},\label{eq:g_k_leq_r_k}\\
    F(x_k)-F(x_{k+1}) \geq \frac{\beta\alpha^2}{\lambda_k}\norm{F'(x_{k+1})}{*}^2 &\qquad\text{or equivalently}\qquad F_k-F_{k+1}=F(x_k)-F(x_{k+1}) \geq \frac{\beta\alpha^2}{\lambda_k}g_{k+1}^2.\label{eq:functionValuesGradientRelation}
\end{align}
\end{lem}
\begin{proof}
The inequalities in \eqref{eq:F_k_decr} follow from \eqref{eq:stop_cond_alg_new_2}, \eqref{eq:F_r_g_notaion}, and \cref{ass:basic}. 
Applying Cauchy--Schwarz on the first inequality in \eqref{eq:stop_cond_alg} and using the definition of $x_{k+1}$ in \cref{step:k+1_update} of \cref{alg:proximal_newton} gives
\[
\frac{\alpha}{\lambda_k}\norm{F'(x_{k+1})}{*} \leq \norm{x_k-x_{k+1}}{},
\]
which together with \eqref{eq:F_r_g_notaion} clearly implies \eqref{eq:g_k_leq_r_k}.
Further, after plugging the previous inequality into the second inequality in \eqref{eq:stop_cond_alg} and using again the definition of $x_{k+1}$ in \cref{step:k+1_update} of \cref{alg:proximal_newton}, we get 
\begin{align*}
    F(x_k)-F(x_{k+1}) &\geq \beta \lambda_k \norm{x_k-x_{k+1}}{}^2 \geq \frac{\beta\alpha^2}{\lambda_k}\norm{F'(x_{k+1})}{*}^2,
\end{align*}
and \eqref{eq:functionValuesGradientRelation} follows by \eqref{eq:F_r_g_notaion}.
\end{proof}
\begin{lem}\label{lem:rk_identity_new} 
Given a set $U\subseteq \HH$,  suppose there exists $F_U^* \in \mathbb{R}$ such that $F(x) \geq F_U^*$  for all $x\in U$.
\begin{enumerate}[label=(\roman*)]\itemsep=0em
    \item At iteration $k\geq 0$ of \cref{alg:proximal_newton}, let the PL inequality hold at $x_{k} \in U$, i.e., 
    \begin{equation}
    \label{eq:PL_at_point}
    \frac{1}{2\mu} {\rm dist}(0, \partial F(x_{k}))^2 \geq F(x_{k})-F_U^*.%
    \end{equation}
    Then, we have
    \begin{align}
    g_{k}^2 \coloneqq  \norm{F'(x_{k})}{*}^2 &\geq 2\mu (F(x_{k})-F_U^*). \label{eq:F_k_g_k_bound} 
    \end{align}  
    \item At iteration $k+1\geq 1 $ of \cref{alg:proximal_newton}, let the assumption of the previous item hold and the acceptance criteria \eqref{eq:stop_cond_alg} hold. Let in addition $x_{k+2} \in U$. Then, we have   
    \begin{align}
    r_{k+1}^2 &\leq \frac{\lambda_k r_k^2}{ \alpha^2 \beta \mu}, \label{eq:rk_identity_new} \\
     g_{k+2}^2 & \leq \frac{\lambda_{k+1} g_{k+1}^2}{2\alpha^2\beta\mu}. \label{eq:gk_identity_new}
    \end{align} 
\end{enumerate}
\end{lem}
\begin{proof}
Since by construction $F'(x_{k}) \in \partial F(x_{k})$, we have that
\begin{equation*}
     g_{k}^2 \overset{\eqref{eq:F_r_g_notaion}}{\coloneqq } \norm{F'(x_{k})}{*}^2 \geq  {\rm dist}(0, \partial F(x_{k}))^2 \overset{\eqref{eq:PL_at_point}}{\geq}   2\mu (F(x_{k})-F_U^*),   
\end{equation*}
which is \eqref{eq:F_k_g_k_bound}. 
By \eqref{eq:stop_cond_alg_new_2} at iteration $k+1$, inequality $F(x_{k+2}) - F_U^* \geq 0$ (which follows from $x_{k+2} \in U$), inequality \eqref{eq:F_k_g_k_bound} at iteration $k+1$, and inequality \eqref{eq:g_k_leq_r_k}, we observe that
\begin{align}
\label{eq:F_k_g_k_bound_proof_1}
    \beta \lambda_{k+1} r_{k+1}^2 \overset{\eqref{eq:stop_cond_alg_new_2}}{\leq} (F(x_{k+1})- F_U^*)-(F(x_{k+2})- F_U^*) \leq F(x_{k+1})- F_U^* \overset{\eqref{eq:F_k_g_k_bound}}{\leq} \frac{g_{k+1}^2}{2\mu} \overset{\eqref{eq:g_k_leq_r_k}}{\leq} \frac{\lambda_k^2 r_k^2}{2\alpha^2 \mu}.
\end{align}
Using that $ \lambda_{k}/2 = \Lambda_{k+1} \leq \lambda_{k+1}$ by \cref{step:k+1_update} of \cref{alg:proximal_newton}, we obtain  \eqref{eq:rk_identity_new}. To obtain  \eqref{eq:gk_identity_new}, we use \eqref{eq:g_k_leq_r_k} at iteration $k+1$ as well as the first and the penultimate inequality in \eqref{eq:F_k_g_k_bound_proof_1}:
\begin{equation*}
    g_{k+2}^2 \overset{\eqref{eq:g_k_leq_r_k}}{\leq} \frac{\lambda_{k+1}^2}{\alpha^2} r_{k+1}^2 
    \overset{\eqref{eq:F_k_g_k_bound_proof_1}}{\leq} \frac{\lambda_{k+1}^2}{\alpha^2} \frac{g_{k+1}^2}{2\beta \mu\lambda_{k+1}}= \frac{\lambda_{k+1} g_{k+1}^2}{2\alpha^2\beta\mu}.
\end{equation*}
\end{proof}

\section{Global convergence rates}
\label{sec:global rates}
In this section, we obtain our main results on global nonasymptotic convergence rate results for \cref{alg:proximal_newton} in various settings including under a PL condition and a convex setting. The main tool for these results is \cref{ass:main_inequality}.

\subsection{Global \texorpdfstring{$\mathcal{O}(1\slash \sqrt{k})$}{O(1/sqrt(k))} convergence rate for gradient norm}
In this subsection, we focus on the setting where $F$ in \eqref{eq:opt} is not necessarily convex (but \cref{ass:subproblem} is satisfied). We show that under the global \cref{ass:main_inequality},  \cref{alg:proximal_newton} has a global nonasymptotic convergence rate $\mathcal{O}(1\slash \sqrt{k})$ in terms of the minimal subgradient norm on the trajectory. The formal result is as follows. 
\begin{theorem}[Global convergence rate]
    \label{thm:global_convergence_sublinear_nonconv}
    Let, for problem \eqref{eq:opt},  \cref{ass:basic,ass:subproblem,ass:main_inequality} hold. Let also $(x_k)_{k \in \mathbb{N}}$ be the iterates of \cref{alg:proximal_newton} with arbitrary starting point  $x_0\in \dom F$. Then, the sequence $(x_k)_{k \in \mathbb{N}}$ is well defined, $(x_k)_{k \in \mathbb{N}} \subset \F_0$, i.e., the iterates of the algorithm stay in the sublevel set $\F_0$ (see \eqref{eq:sublevel_set_def}), and 
    \[\norm{F'(x_{k})}{*} \to 0 \quad \text{as $k\to \infty$.}\] Moreover, the following global nonasymptotic convergence rate holds:
    \begin{equation}
    \label{eq:nonconv_rate}
        \min_{0\leq i \leq k-1}\norm{F'(x_{i+1})}{*}\leq \frac{\sqrt{\max\{2mL,\Lambda_0\}(F({x}_0) - F^*)}}{\sqrt{\beta\alpha^2 k}}, \quad k \geq 1.
    \end{equation}
    If in addition $\lambda_k \to 0$  as $k\to \infty$, we obtain that $\min_{0\leq i \leq k-1}\norm{F'(x_{i+1})}{*} = \so(1\slash \sqrt{k})$.
    Finally, the number of Newton steps \eqref{eq:prox_step_alg}
    up to the end of iteration $k$ does not exceed
    \[
    k+1+\log_2\max\{1/2,mL/\Lambda_0\}.
    \]
\end{theorem}

\begin{proof}
Since \cref{ass:main_inequality} holds, we have by \cref{lem:acceptance_inner_nonconv} that both inequalities in \eqref{eq:stop_cond_alg} hold in each iteration $k\geq 0$, i.e., the acceptance conditions of the inner loop of \cref{alg:proximal_newton} hold and the inner loop ends after a finite number of trials. Thus, the sequence $(x_k)_{k \in \mathbb{N}}$ is well defined. 
From \eqref{eq:F_k_decr}, it is clear by induction that for all $k\geq 0$, $F(x_k)\leq F(x_0)$ and, thus, $(x_k)_{k \in \mathbb{N}} \subset \F_0$, where $\F_0$ is defined in \eqref{eq:sublevel_set_def}. 

Summing \eqref{eq:functionValuesGradientRelation} from $0$ to $k-1$ and telescoping, using that $F$ is bounded from below by $F^*$ by \cref{ass:basic}, 
\cref{lem:acceptance_inner_nonconv} and the notation $g_{i+1}\coloneqq \norm{F'(x_{i+1})}{*}$ in \eqref{eq:F_r_g_notaion}, we obtain
\[
F({x}_0) - F^*  \geq F({x}_0) - F(x_{k})  \overset{\eqref{eq:functionValuesGradientRelation}}{\geq}  \sum_{i=0}^{k-1}\frac{\beta\alpha^2}{\lambda_i}g_{i+1}^2 \overset{\eqref{eq:lambda_bound}}{\geq} 
\sum_{i=0}^{k-1} \frac{\beta\alpha^2}{\overline{\lambda}}g_{i+1}^2 \geq\frac{k\beta\alpha^2}{\overline{\lambda}} \min_{0\leq i \leq k-1}\norm{F'(x_{i+1})}{*}^2.
\]
First, observe that this chain of inequalities implies that for all $k\geq 0$, $\sum_{i=0}^{k-1}g_{i+1}^2$ is bounded from above. Thus, $\norm{F'(x_{i})}{*}=g_{i} \to 0$ as $i\to \infty$.
Second, recalling from \eqref{eq:lambda_bound} that $\overline{\lambda}=\max\{2mL,\Lambda_0\}$, we immediately get \eqref{eq:nonconv_rate}. Now, suppose that $\lambda_i \to 0$ as $i \to \infty$. We also obtain from the above inequalities, by keeping $\lambda_i$, that
\[F(x_0)  - F^* \geq \beta\alpha^2\min_{0 \leq i \leq k-1}g_{i+1}^2 \sum_{i=0}^{k-1}\frac{1}{\lambda_i}\]
and hence
\[\min_{0 \leq i \leq k-1}g_{i+1}^2 \leq \frac{F(x_0)  - F^*}{\beta\alpha^2}\frac{1}{\sum_{i=0}^{k-1}\frac{1}{\lambda_i} }.\]
The inequality between arithmetic and geometric means (AM-GM) yields $\sum_{i=0}^{k-1}\frac{1}{\lambda_i} \geq k\left(\prod_{i=0}^{k-1}\frac{1}{\lambda_i}\right)^{1\slash k}$, whence
\[\min_{0 \leq i \leq k-1}g_{i+1}^2 \leq \frac{F(x_0)  - F^*}{\beta\alpha^2}\frac{1}{k\left(\prod_{i=0}^{k-1}\frac{1}{\lambda_i}\right)^{1\slash k}} = \frac{F(x_0)  - F^*}{\beta\alpha^2}\frac{\left(\prod_{i=0}^{k-1}\lambda_i\right)^{1\slash k}}{k},\]
and the right-hand side is $\so(1\slash k)$ because the geometric mean converges to zero. By taking the square root, we deduce the $\so(1\slash \sqrt{k})$ rate for $\min_{0\leq i \leq k-1}\norm{F'(x_{i+1})}{*}$.

It remains to estimate the number of Newton steps \eqref{eq:prox_step_alg}. By \cref{step:k+1_update} of \cref{alg:proximal_newton}, we have for all $k\geq 0$ that $\Lambda_{k+1} = 2^{j_k}\Lambda_k/2$, or equivalently $j_k=1+\log_2\frac{\Lambda_{k+1}}{\Lambda_k}$. Summing these equalities from $0$ to $k$, we obtain that the total number of Newton steps up to the end of iteration $k$ is
\[
N_k=\sum_{i=0}^k j_i = \sum_{i=0}^k \left(1+\log_2\frac{\Lambda_{i+1}}{\Lambda_i}\right)=k+1+\log_2\frac{\Lambda_{k+1}} {\Lambda_0}\overset{\eqref{eq:lambda_bound}}{\leq} k+1+\log_2\frac{\overline{\lambda}} {2\Lambda_0}\overset{\eqref{eq:lambda_bound}}{=} k+1+\log_2\max\{1/2,mL/\Lambda_0\}.
\]
\end{proof}

\begin{remark}[Choice of the parameters $m,\alpha,\beta$]
    As one can see, the rate in \eqref{eq:nonconv_rate} is proportional to $\sqrt{m/(\beta\alpha^2)}$ and the smaller this number, the better the rate. This means that $\alpha,\beta$ should be chosen as large as possible. Recall from \cref{lem:acceptance_inner_nonconv}  that $0 < \alpha \leq 1\slash 2$, and $0<\beta \leq (m-1)\slash (2m)$. Thus, the optimal choice is $\alpha = 1\slash 2, \beta = (m-1)\slash (2m)$. This leads to $\sqrt{m/(\beta\alpha^2)}=\sqrt{8m^2/(m-1)}$, which is minimised at $m^*=2$. This, in turn, implies $\beta=1/4$. This is the standard choice outlined in the input of \cref{alg:proximal_newton}. %
\end{remark}

We finish this subsection with a result establishing the stationarity of strong accumulation points of the sequence $(x_k)_{k \in \mathbb{N}}$ of the iterates of \cref{alg:proximal_newton}.
By \cref{thm:global_convergence_sublinear_nonconv}, we have $(x_k)_{k \in \mathbb{N}} \subset \F_0$. If the sublevel set $ \F_0$ is bounded
and the space $\HH$ is finite-dimensional, then there exists a strong accumulation point of $(x_k)_{k \in \mathbb{N}}$. Unfortunately, if  $\HH$ is infinite-dimensional, only weak accumulation points are guaranteed to exist. Following existing literature on nonconvex minimisation in Hilbert spaces \cite{geiersbach2021stochastic,potzl2022second}, we restrict ourselves here to strong accumulation points. The study of additional assumptions on $f$, e.g., strong compactness of sublevel sets, that allow establishing stationarity of weak accumulation points is left for future work.

\begin{prop}[Stationarity of strong accumulation points]
\label{prop:subseq_stationary_nonconv}
Let, for problem \eqref{eq:opt}, \cref{ass:basic,ass:subproblem,ass:main_inequality} hold. Then, all strong accumulation points of the sequence $(x_k)_{k \in \mathbb{N}}$ of the iterates of \cref{alg:proximal_newton} are stationary points of problem \eqref{eq:opt}.
\end{prop}
\begin{proof}
    Pick an accumulation point $\bar{x} \in \HH$ and a subsequence such that $x_{k_l} \to \bar{x} $ as $l\to \infty$. 
    From \cref{thm:global_convergence_sublinear_nonconv}, we know that $F'(x_k) \to 0$ as $k\to \infty$. Thus, $F'(x_{k_l}) = f'(x_{k_l})+\psi'(x_{k_l}) \to 0$  as $l \to \infty$. 
    Since $f$ is continuously Fréchet differentiable, we have that $f'(x_{k_l}) \to f'(\bar{x})$. Consequently $\psi'(x_{k_l}) = (\psi'(x_{k_l}) + f'(x_{k_l})) - f'(x_{k_l}) \to - f'(\bar{x})$. Since $\psi$ is convex lsc, $x\mapsto \partial \psi(x)$ has a closed graph and we deduce that $-f'(\bar{x}) =  \lim_{l\to \infty} \psi'(x_{k_l}) \in \partial \psi(\bar{x})$. Thus, equivalently, we have that $0 \in \partial F(\bar{x}) =  f'(\bar{x}) + \partial \psi(\bar{x})$ and so $\bar{x}$ is a stationary point.%
\end{proof}

\subsection{Global linear convergence rate for objective and iterates under PL}\label{subsec:global_linear_under_global_PL}
In this subsection, we still assume that $F$ in \eqref{eq:opt} may be nonconvex and \cref{ass:subproblem,ass:main_inequality} hold. Yet, we additionally make \cref{ass:PL} that $F$ satisfies the PL condition. %
This allows us to obtain a global nonasymptotic \emph{linear} convergence rate for \cref{alg:proximal_newton} for the objective functional values and for the iterates.
The formal result is as follows.

 \begin{theorem}[Global convergence rate, PL condition]
    \label{thm:global_convergence_nonconv_PL}
    Let, for problem \eqref{eq:opt}, \cref{ass:basic,ass:subproblem,ass:main_inequality,ass:PL} hold. Let also $(x_k)_{k \in \mathbb{N}}$ be the iterates of \cref{alg:proximal_newton} with arbitrary starting point $x_0\in \dom F$. %
    Then, the following statements hold. 
    \begin{enumerate}[label=(\roman*)]\itemsep=0em
    \item The objective values have the following global nonasymptotic linear convergence rate: %
        \begin{equation}
        \label{eq:nonconv_PL_rate}
        F(x_{k})-F^* \leq \exp\left(-\frac{2\beta\alpha^2\mu}{2\beta\alpha^2\mu+\max\{2mL,\Lambda_0\}} \cdot k\right)\cdot(F({x}_0)-F^*), \quad k \geq 0.
        \end{equation}
        If in addition $\lambda_k \to 0$ as $k\to \infty$, then $F(x_k)-F^*$ converges to $0$ superlinearly.
        \item %
        $x_k$ converges strongly to a global minimum %
        $x^*$ and the convergence is global, nonasymptotic, and linear:
        \begin{align}
        \norm{x_{k} - x^*}{}  \leq 
        \exp\left(-\frac{2\beta\alpha^2\mu}{4\beta\alpha^2\mu+2\max\{2mL,\Lambda_0\}} \cdot (k-2)\right)\cdot\frac{5\sqrt{F({x}_0)-F^*} }{\alpha\beta\mu^{1/2}}, \;\; k\geq 2.
        \label{eq:nonconv_PL_rate_x} 
        \end{align}
            If in addition $\lambda_k \to 0$ as $k\to \infty$, then $x_k$ converges to $x^*$ superlinearly.
        \item The number of Newton steps \eqref{eq:prox_step_alg} up to the end of iteration $k$ does not exceed
        \[
        k+1+\log_2\max\{1/2,mL/\Lambda_0\}.
        \]
    \end{enumerate}
\end{theorem}
\begin{proof}
Recall from \eqref{eq:F_r_g_notaion} the notation for $k\geq 0$
\[
F_k\coloneqq F(x_k)-F^*, \qquad r_k\coloneqq \norm{x_k-x_{k+1}}{}, \qquad g_k\coloneqq \norm{F'(x_k)}{*}.
\]
\begin{enumerate}[label=(\roman*), wide, labelwidth=!, labelindent=0pt]\itemsep=0em
\item  
From \cref{thm:global_convergence_sublinear_nonconv}, we know that the sequence $(x_k)_{k \in \mathbb{N}} \subset \F_0$ is well defined. We can apply the PL condition \cref{ass:PL} (and use that $F^*$ is the global minimal value by \cref{ass:basic}) to get
\begin{equation*}
     F_k - F_{k+1}  \overset{\eqref{eq:functionValuesGradientRelation}}{\geq}  \frac{\beta\alpha^2}{\lambda_k}g_{k+1}^2  \overset{\eqref{eq:PL}}{\geq} \frac{\beta\alpha^2}{\lambda_k} \cdot 2\mu F_{k+1}. 
\end{equation*}
Rearranging, we obtain
\begin{equation}
\label{eq:thm:global_convergence_nonconv_PL_proof_-1.5}
    F_{k+1}  \leq \frac{1}{1+\frac{2\beta\alpha^2\mu}{\lambda_k}}F_k=\frac{\lambda_k}{2\beta\alpha^2\mu+\lambda_k}F_k=\left(1-\frac{2\beta\alpha^2\mu}{2\beta\alpha^2\mu+\lambda_k}\right)F_k.
\end{equation}
We immediately see that if $\lambda_k \to 0$ as $k\to \infty$, then $1-\frac{2\beta\alpha^2\mu}{2\beta\alpha^2\mu+\lambda_k}\to 0$ and we obtain superlinear convergence. Using the upper bound $\lambda_k\leq \overline{\lambda}$ from \cref{lem:acceptance_inner_nonconv} (see \eqref{eq:lambda_bound}), we obtain 
\[
    \left(1-\frac{2\beta\alpha^2\mu}{2\beta\alpha^2\mu+\lambda_k}\right) \leq \left(1-\frac{2\beta\alpha^2\mu}{2\beta\alpha^2\mu+\overline{\lambda}}\right) \leq \exp\left(-\frac{2\beta\alpha^2\mu}{2\beta\alpha^2\mu+\overline{\lambda}}\right).
\]
This, by induction gives, for any $k\geq 0$,
\begin{equation}
\label{eq:thm:global_convergence_nonconv_PL_proof_0}
    F(x_k)-F^*=F_{k}  \leq \exp\left(-\frac{2\beta\alpha^2\mu}{2\beta\alpha^2\mu+\overline{\lambda}}\cdot k\right)(F({x}_0)-F^*),
\end{equation}
which is \eqref{eq:nonconv_PL_rate} since $\overline{\lambda}=\max\{2mL,\Lambda_0\}$ from \eqref{eq:lambda_bound}.

\item 
By \cref{lem:rk_identity_new}, \eqref{eq:stop_cond_alg_new_2}, and \eqref{eq:thm:global_convergence_nonconv_PL_proof_0},
we observe, for all $k\geq0$,
\begin{align} \label{eq:thm:global_convergence_nonconv_PL_proof_-2}
     r_{k+1}^2 \overset{\eqref{eq:rk_identity_new}}{\leq} %
     \frac{\lambda_k r_k^2}{\alpha^2 \beta  \mu} \overset{\eqref{eq:stop_cond_alg_new_2}}{\leq} \frac{  1}{\alpha^2 \beta^2\mu} (F_k - F_{k+1})\overset{\eqref{eq:thm:global_convergence_nonconv_PL_proof_0}}{\to} 0 \quad \text{as $k \to \infty$}.  
\end{align}
Thus, we have shown that $r_k \to 0 $ as $k \to \infty$.

Using the concavity of $t \mapsto t^{1/2}$, the fact that $F_k \geq F_{k+1} \geq 0$ by \eqref{eq:F_k_decr}, and inequality \eqref{eq:stop_cond_alg_new_2}, we get, for all $k\geq0$,
\begin{align}
    F_{k}^{1/2} -F_{k_{}+1}^{1/2} \geq \frac{1}{2F_{k_{}}^{1/2}}(F_{k_{}} - F_{k_{}+1}) \overset{\eqref{eq:stop_cond_alg_new_2}}{\geq} \frac{\beta \lambda_{k_{}} r_{k_{}}^2}{2 F_{k_{}}^{1/2}}. 
    \label{eq:thm:global_convergence_nonconv_PL_proof_1}
\end{align}
Applying again the PL condition, we get $g_k^2 \overset{\eqref{eq:PL}}{\geq} 2\mu F_k$.
Further, by \eqref{eq:g_k_leq_r_k} at iteration $k-1$ we have $g_k\leq \frac{\lambda_{k-1}}{\alpha}r_{k-1}$.
Combining the last two observations with \eqref{eq:thm:global_convergence_nonconv_PL_proof_1}, we obtain
\begin{align*}
    F_{k_{}}^{1/2} -F_{k_{}+1}^{1/2} \overset{\eqref{eq:PL}}{\geq} \frac{\beta \mu^{1\slash 2}\lambda_{k_{}} r_{k_{}}^2}{\sqrt{2} g_k}  
    \overset{\eqref{eq:g_k_leq_r_k}}{\geq}  \frac{\alpha \beta \mu^{1\slash 2}\lambda_{k_{}} r_{k_{}}^2}{\sqrt{2} \lambda_{k-1}r_{k-1}}.      
\end{align*}
By \cref{step:k+1_update} of \cref{alg:proximal_newton} we have $\lambda_k\geq \Lambda_k=\lambda_{k-1}/2$ and consequently
\begin{align}
     F_{k_{}}^{1/2} -F_{k_{}+1}^{1/2}            \geq  \frac{\alpha \beta \mu^{1\slash 2}r_{k_{}}^2}{2\sqrt{2} r_{k-1}}.      \label{eq:main_thm_PL_2}
\end{align}
Rearranging \eqref{eq:main_thm_PL_2} and using inequality $2\sqrt{ab} \leq a+b$ that holds for $a,b\geq0$, we get
\begin{align*}
    r_{k} &\leq  \sqrt\frac{2\sqrt{2} r_{k-1} (F_{k_{}}^{1/2} -F_{k_{}+1}^{1/2})}{\alpha\beta\mu^{1/2}}
    \leq 
    \frac{r_{k-1}}{2} + \frac{\sqrt{2} (F_{k_{}}^{1/2} -F_{k_{}+1}^{1/2})}{\alpha\beta\mu^{1/2}}.
\end{align*}
Summing from $l+1 \geq 1$ to $n \geq l+1$ we obtain 
\begin{align*}
        \sum_{k = l+1}^n r_{k} \leq \frac{1}{2}\sum_{k=l}^{n-1}r_k +  \frac{\sqrt{2}  (F_{l_{}+1}^{1/2} -F_{n_{}+1}^{1/2})}{\alpha\beta\mu^{1/2}}.
\end{align*}
Putting the sums on one side and using that $r_k\geq 0$, we obtain
\begin{align}
        \frac{1}{2}\sum_{k = l+1}^{n-1} r_{k} \leq r_l +  \frac{\sqrt{2}  (F_{l_{}+1}^{1/2} -F_{n_{}+1}^{1/2})}{\alpha\beta\mu^{1/2}} \to 0 \quad \text{as $l,n \to \infty$,} \label{eq:main_thm_PL_5}
\end{align}
where we used that $r_k\to 0$ by \eqref{eq:thm:global_convergence_nonconv_PL_proof_-2} and $F_k \to 0$ by \eqref{eq:thm:global_convergence_nonconv_PL_proof_0}.
Hence $(x_k)_{k \in \mathbb{N}} $ is a Cauchy sequence and strongly converges to a point $x^* \in \HH$, which by \cref{prop:subseq_stationary_nonconv}, is a stationary point. 
Further, the functional $F$ is a sum of continuously differentiable $f$ and convex proper lower semicontinuous $\psi$. Thus, $F$ is lower semicontinuous, which implies
\begin{equation*}
    F^* \overset{\eqref{eq:thm:global_convergence_nonconv_PL_proof_0}}{=} \lim_{k \to \infty}F(x_{k}) = \liminf_{k \to \infty} F(x_{k}) \geq F(x^*).
\end{equation*}
Therefore, $x^*$ is a global minimiser itself. 

The linear convergence of $(x_k)_{k \in \mathbb{N}} $ is shown as follows. Fixing some $l \geq 0$ and arbitrary $n \geq l+2$, we obtain using the bound in  \eqref{eq:main_thm_PL_5} and the triangle inequality, that 
    \begin{align}
        \norm{x_{l+1} - x_n}{}
        \leq \sum_{k = l+1}^{n-1} r_{k} 
        &\leq 2r_l +  \frac{2\sqrt{2}  (F_{l_{}+1}^{1/2} -F_{n_{}+1}^{1/2})}{\alpha\beta\mu^{1/2}} \label{eq:need_for_useful_bounds} %
         \\
        \overset{\eqref{eq:thm:global_convergence_nonconv_PL_proof_-2}}&{\leq}  2\sqrt{\frac{F_{l-1} - F_{l}}{\alpha^2 \beta^2\mu}} 
        +   \frac{2\sqrt{2}  (F_{l_{}+1}^{1/2} -F_{n_{}+1}^{1/2})}{\alpha\beta\mu^{1/2}}.\label{eq:main_thm_PL_6}
    \end{align}
    Sending $n\to \infty$ in \eqref{eq:main_thm_PL_6}, we obtain
    \begin{align*}
        \norm{x_{l+1} - x^*}{}        &\overset{\eqref{eq:main_thm_PL_6}}{\leq}  2\sqrt{\frac{F_{l-1} - F_{l}}{\alpha^2 \beta^2\mu}} 
        +   \frac{2\sqrt{2}  F_{l_{}+1}^{1/2} }{\alpha\beta\mu^{1/2}} 
        {\leq}
        \left( 2\sqrt\frac{1}{\alpha^2\beta^2\mu} 
        +   \frac{2\sqrt{2}   } {\alpha\beta\mu^{1/2}}\right)F_{l-1}^{1/2} \leq \frac{5F_{l-1}^{1/2} }{\alpha\beta\mu^{1/2}}, 
    \end{align*}
    where the penultimate inequality follows as $F_{l-1}-F_l \leq F_{l-1}$ and $0\leq F_{l+1} \leq F_l \leq F_{l-1}$ by \eqref{eq:F_k_decr}. Thus, the linear convergence of $x_l$ \eqref{eq:nonconv_PL_rate_x}  follows since $F_{l}$ converges to 0 linearly, see \eqref{eq:thm:global_convergence_nonconv_PL_proof_0}. In the same way, superlinear convergence when $\lambda_l \to 0$ as $l \to \infty$ follows since $F_{l}$ converges to 0 superlinearly in that case, see \eqref{eq:thm:global_convergence_nonconv_PL_proof_-1.5}.

\item Finally, the bound on the number of Newton steps follows from \cref{thm:global_convergence_sublinear_nonconv}.\qedhere
\end{enumerate}    
\end{proof}
    
As we see from the above result, the additional assumption of the PL condition \cref{ass:PL} allows us to obtain quite strong results, including global strong nonasymptotic linear convergence of the iterates to a global minimum and global nonasymptotic linear convergence of objective values. Observe also that it would have sufficed to assume \cref{ass:PL} only on $\F_0$ (instead of $\dom F$) since by \cref{thm:global_convergence_sublinear_nonconv} we know that $(x_k)_{k \in \mathbb{N}} \subset \F_0$, but we formulated it as it is to simplify the presentation. We also remark that we did not assume that $F$ is convex.   

We finish this subsection by a technical corollary of \cref{thm:global_convergence_nonconv_PL} that will be crucial later for the proof of superlinear convergence of \cref{alg:proximal_newton}. This result relates the distance to solution $\norm{x_k-x^*}{}$ to the distance between two consecutive iterates $\|x_{k+1} - x_k\|$, which in general turns out to be important for proving fast convergence of various algorithms, see, e.g., \cite{drusvyatskiy2018error,drusvyatskiy2021nonsmooth}. A similar result was proved in the recent work \cite{potzl2022second}, but under strong convexity. In our setting, we considerably relax this assumption by assuming that the PL condition holds instead of strong convexity.

    \begin{lem}\label{lem:useful_bounds_new}%
        Let the assumptions of \cref{thm:global_convergence_nonconv_PL} hold. Then, there exist constants $c_1,c_2>0$ such that, for all $k\geq 0$,
    \begin{equation}
        \norm{x_{k+1}-x^*}{} \leq c_1 \| x_{k+1}- x_k\| \quad \text{and} \quad  \norm{x_k-x^*}{} \leq c_2 \|x_{k+1} - x_k\|. \label{eq:useful_bounds_new}
    \end{equation}
    \end{lem}
    \begin{proof}
        Since the assumptions of \cref{thm:global_convergence_nonconv_PL} hold, we know that $x_k \to x^*$ as $k\to \infty$ and we can use \eqref{eq:need_for_useful_bounds} which holds for any $l \geq 0$ and $n \geq l+2$. Whence, changing in \eqref{eq:need_for_useful_bounds} $l$ to $k$ and sending $n \to \infty$, we have        
        \begin{align}
        \label{eq:lem:useful_bounds_new_proof_1}
            \|x_{k+1} - x^*\|   &\leq 2r_k + \frac{2\sqrt{2}  F_{k_{}+1}^{1/2}}{\alpha\beta\mu^{1/2}}. 
        \end{align}
        By \cref{ass:PL} and the  bound $\lambda_k \leq \overline{\lambda}$ from \cref{lem:acceptance_inner_nonconv}, we obtain
        \begin{equation*}
            F_{k+1}^{1/2}\overset{\eqref{eq:F_r_g_notaion}}{=} (F(x_{k+1}) - F^*)^{1/2} \overset{\eqref{eq:PL}}{\leq}  \frac{g_{k+1}}{(2\mu)^{1/2}} \overset{\eqref{eq:g_k_leq_r_k}}{\leq} \frac{\lambda_{k}r_{k}}{\alpha(2\mu)^{1/2}}    %
            \leq \frac{\overline{\lambda}}{\alpha(2\mu)^{1/2}} r_k. %
        \end{equation*}
        Combining this bound and \eqref{eq:lem:useful_bounds_new_proof_1} and recalling that, by \eqref{eq:F_r_g_notaion}, $r_k=\norm{x_k-x_{k+1}}{}$, we have that the first inequality in \eqref{eq:useful_bounds_new} holds with $c_1=2 + \frac{2\bar\lambda}{\alpha^2\beta\mu}$.
        The triangle inequality then yields the second inequality in \eqref{eq:useful_bounds_new} with $c_2=1+c_1$.
    \end{proof}

\subsection{Global \texorpdfstring{$\mathcal{O}(1\slash k)$}{O(1/k)} convergence rate for objective under convexity}\label{sec:global_convex_functions}
In this subsection, we focus on the setting where $F$ in \eqref{eq:opt} is convex on its sublevel set $\F_0$ defined in \eqref{eq:sublevel_set_def}. This assumption allows us to improve some results of the previous subsections. In particular, we obtain for \cref{alg:proximal_newton} a global nonasymptotic convergence rate $\mathcal{O}(1\slash k)$ for the objective residual $F(x_k)-F^*$. Convexity also allows us to show that all \textit{weak} accumulation points of the sequence generated by \cref{alg:proximal_newton} are global minima.

First, we can improve the bound \eqref{eq:functionValuesGradientRelation} in \cref{lem:functionValuesGradientRelation} by using convexity.
\begin{lem}
\label{lem:convex_g_condition}
Let at iteration $k\geq 0 $ of \cref{alg:proximal_newton} the acceptance condition \eqref{eq:stop_cond_alg} hold. Let also $F$ be convex between $x_k$ and $x_{k+1}$, i.e., $F(x_k) \geq F(x_{k+1}) + \langle F'(x_{k+1}), x_k - x_{k+1} \rangle$. 
Then, we have that
    \begin{align}
    F(x_k)-F(x_{k+1}) &\geq \frac{\alpha}{\lambda_k}\norm{F'(x_{k+1})}{*}^2 \quad \text{or equivalently} \quad  F_k-F_{k+1}=F(x_k)-F(x_{k+1}) \geq \frac{ \alpha}{\lambda_k}g_{k+1}^2.\label{eq:functionValuesGradientRelation_convex}
\end{align}
\end{lem}
\begin{proof}
 Using convexity, the definition of $x_{k+1}$ in \cref{step:k+1_update} of \cref{alg:proximal_newton}, and the first inequality  in \eqref{eq:stop_cond_alg} we obtain $F(x_k)-F(x_{k+1}) \geq \langle F'(x_{k+1}), x_k - x_{k+1} \rangle \overset{\eqref{eq:stop_cond_alg}}{\geq} \frac{\alpha}{\lambda_k}\norm{F'(x_{k+1})}{*}^2.$ The rest follows from the notation \eqref{eq:F_r_g_notaion}.
\end{proof}

Recall the definition \eqref{eq:sublevel_set_def} of the sublevel set $\F_0$ 
defined by the starting point $x_0$. We assume it to be bounded, i.e.,
\begin{equation}
\label{eq:bounded_sublevel_set}
     D_0\coloneqq  \sup_{x, y \in \mathcal{F}_0} \norm{x-y}{} < +\infty.  
\end{equation}
By \cref{thm:global_convergence_sublinear_nonconv}, we have that 
$(x_k)_{k \in \mathbb{N}} \subset \mathcal{F}_0$ is well defined. 
This, together with convexity of $F$ on $\F_0$ and the notation $g_k\coloneqq \norm{F'(x_k)}{*}$ in \eqref{eq:F_r_g_notaion} imply that, for any $\bar x \in \mathcal{F}_0$ and $k\geq 0$, we have
    \begin{equation}
    \label{eq:convexity_id} 
        F(x_k) - F(\bar x) \leq \langle F'(x_{k}), x_{k}-\bar x \rangle \leq g_{k}\sup_{x,y \in \mathcal{F}_0}\norm{x-y}{} = g_{k}D_0.   
    \end{equation}
As a direct consequence we obtain that \textit{weak} accumulation points of $(x_k)_{k \in \mathbb{N}}$ are global minima. 
\begin{prop}[Optimality of weak accumulation points]
\label{prop:weak_conv}
Let, for problem \eqref{eq:opt},  \cref{ass:basic,ass:subproblem,ass:main_inequality} hold. Let additionally $F$ have a bounded sublevel set $\mathcal{F}_0$ and be convex on $\mathcal{F}_0$. Then, all weak accumulation points of the sequence $(x_k)_{k \in \mathbb{N}}$ of the iterates of \cref{alg:proximal_newton} are global minima of problem \eqref{eq:opt}.
\end{prop}
\begin{proof}
    Note that since $\F_0$ is bounded and $(x_k)_{k \in \mathbb{N}} \subset \mathcal{F}_0$, the set of weak accumulation points is not empty. Pick an accumulation point $\bar{x} \in \mathcal{H}$ and a subsequence such that $x_{k_l} \rightharpoonup \bar{x} $ as $l\to \infty$. By \eqref{eq:convexity_id}, we conclude for any minimiser $x^* \in \mathcal{H}$, 
    \begin{equation}
    \label{eq:prop:weak_conv_proof_1}
        F(x_{k_l}) - F(x^*) \leq g_{k_l} D_0. 
    \end{equation}
    By \cref{thm:global_convergence_sublinear_nonconv}, we have that
    $g_k \to 0$ as $k\to \infty$ and consequently, by \eqref{eq:prop:weak_conv_proof_1},  
    also
    \begin{equation*}
        F(x^*) = \lim_{l \to \infty}F(x_{k_l}) = \liminf_{l \to \infty} F(x_{k_l}) \geq F(\bar{x})
    \end{equation*}
    by weak lower semicontinuity of the convex function $F$. Therefore, $\bar{x}$ is itself a global minimiser.    
\end{proof}
\begin{remark}
    If the minimiser is unique, we deduce weak convergence of the full sequence to the unique minimiser.
\end{remark}

We now present the main result of this subsection on the global nonasymptotic convergence rate of \cref{alg:proximal_newton} when $F$ is convex on $\F_0$.

\begin{theorem}[Global convergence rate, convex case]
\label{thm:global_convergence_sublinear}
Let, for problem \eqref{eq:opt},  \cref{ass:basic,ass:subproblem,ass:main_inequality} hold. Let also $(x_k)_{k \in \mathbb{N}}$ be the iterates of \cref{alg:proximal_newton} with arbitrary starting point $x_0\in \dom F$. Let additionally $F$  have a bounded sublevel set $\mathcal{F}_0$ and be convex on $\mathcal{F}_0$. 
Then, the following global nonasymptotic convergence rate holds:
\begin{equation}
    \label{eq:conv_rate}
        F(x_k) - F^* \leq \norm{F'(x_0)}{*}D_0\exp\left(-\frac{k}{4}\right) + \frac{2D_0^2 \max\{2mL,\Lambda_0\}}{\alpha k}, \quad k \geq 0.
\end{equation}
If in addition $\lambda_k \to 0$ as $k\to \infty$, we obtain that $F(x_k)-F^*=\so(1\slash k)$.
Moreover, the number of Newton steps \eqref{eq:prox_step_alg} up to the end of iteration $k$ does not exceed
\[
    k+1+\log_2\max\{1/2,mL/\Lambda_0\}.
\]

\end{theorem}
\begin{proof}
From \cref{thm:global_convergence_sublinear_nonconv}, we know that 
$(x_k)_{k \in \mathbb{N}} \subset \F_0$ is well defined. Hence, we have the convexity assumption at our disposal for the sequence $(x_k)_{k \in \mathbb{N}}$, by assumption. This, in particular, means that we can use inequalities \eqref{eq:functionValuesGradientRelation_convex} and \eqref{eq:convexity_id} for all $k\geq 0$.

The proof of the convergence rate is based on the per iteration decrease inequality \eqref{eq:functionValuesGradientRelation_convex} and follows the same lines as in  \cite{doikov2024super}. Recall from \eqref{eq:F_r_g_notaion} that we denote $g_k  \coloneqq \norm{F'(x_k)}{*}$ and $F_k \coloneqq  F(x_k)-F^* \geq 0$.
We start with the following chain of inequalities for any iteration $k\geq0$
\begin{align*}
    \frac{1}{F_{k+1}}-\frac{1}{F_k} &= \frac{F_k - F_{k+1}}{F_kF_{k+1}}
    \overset{\eqref{eq:functionValuesGradientRelation_convex}}{\geq} \frac{\alpha}{\lambda_kF_kF_{k+1}}g_{k+1}^2 
    = \frac{\alpha}{\lambda_k F_kF_{k+1}}\left(\frac{g_{k+1}}{g_k}\right)^2 g_k^2\\
        \overset{\eqref{eq:convexity_id}}&{\geq} \frac{\alpha}{D_0^2\lambda_k F_kF_{k+1}}\left(\frac{g_{k+1}}{g_k}\right)^2 F_k^2\geq \frac{\alpha}{D_0^2\lambda_k}\left(\frac{g_{k+1}}{g_k}\right)^2,  
\end{align*}
where in the last step we used that, by \eqref{eq:functionValuesGradientRelation_convex}, $F_k \geq F_{k+1}$.
Summing these inequalities for $j$ from $0$ to $k-1$ and using the inequality between arithmetic and geometric means (AM-GM), we obtain 
\begin{align*}
    \frac{1}{F_{k}}-\frac{1}{F_0}  
        &\geq \frac{\alpha}{D_0^2}\sum_{j=0}^{k-1}\frac{1}{\lambda_j}\left(\frac{g_{j+1}}{g_j}\right)^2 
        \geq \frac{\alpha}{D_0^2}k\left(\prod_{j=0}^{k-1} \frac{1}{\lambda_j}\left(\frac{g_{j+1}}{g_j}\right)^2\right)^{1/k}
        = \frac{\alpha k}{D_0^2} \left(\frac{g_{k}}{g_0}\right)^{2/k}\left(\prod_{j=0}^{k-1} \frac{1}{\lambda_j}\right)^{1/k}\\
         \overset{\eqref{eq:convexity_id}}&{\geq} \frac{\alpha k}{D_0^2}\left(\frac{F_{k}}{D_0g_0}\right)^{2/k}\left(\prod_{j=0}^{k-1}\lambda_j\right)^{-1/k} 
         = \frac{\alpha k}{D_0^2}\left(\frac{D_0g_0}{F_{k}}\right)^{-2/k}\left(\prod_{j=0}^{k-1}\lambda_j\right)^{-1/k}\\
         &= \frac{\alpha k}{D_0^2}\exp\left(-\frac{2}{k}\ln\left(\frac{D_0g_0}{F_{k}}\right)\right)\left(\prod_{j=0}^{k-1}\lambda_j\right)^{-1/k} 
         \geq \frac{\alpha k}{D_0^2}\left(1-\frac{2}{k}\ln\left(\frac{D_0g_0}{F_{k}}\right)\right)\left(\prod_{j=0}^{k-1}\lambda_j\right)^{-1/k},
\end{align*}
where in the last step we used that $\exp(x) \geq 1+x$.  
Consider two cases. If $\frac{2}{k}\ln\left(\frac{D_0g_0}{F_{k}}\right) \geq \frac 12$, we obtain $F_k \leq g_0D_0\exp\left(-\frac{k}{4}\right)$. Otherwise, if $\frac{2}{k}\ln\left(\frac{D_0g_0}{F_{k}}\right) < \frac 12$, we obtain
\begin{align}
\frac{1}{F_k} \geq  \frac{\alpha k}{2D_0^2}\left(\prod_{j=0}^{k-1}\lambda_j\right)^{-1/k} \quad\Longleftrightarrow \quad F_k \leq  \frac{2D_0^2}{\alpha k}\left(\prod_{j=0}^{k-1}\lambda_j\right)^{1/k}.
\label{eq:conv_rate_product}
\end{align}
Combining these two cases, we obtain
\[
 F(x_k) - F^* = F_k \leq  g_0D_0\exp\left(-\frac{k}{4}\right) + \frac{2D_0^2}{\alpha k}\left(\prod_{j=0}^{k-1}\lambda_j\right)^{1/k}.
\]
Applying the bound $\lambda_k \leq \bar\lambda = \max\{2mL,\Lambda_0\}$ from \cref{lem:acceptance_inner_nonconv}, we obtain the convergence rate result \eqref{eq:conv_rate} since $g_0  \coloneqq \norm{F'(x_0)}{*}$. Further, if $\lambda_k \to 0$ as $k\to \infty$, we obtain that $\left(\prod_{j=0}^{k-1}\lambda_j\right)^{1/k} \to 0$ since $\lambda_k\geq 0$ for $k\geq0$, and, hence, $F(x_k)-F^*=\so(1\slash k)$. The bound on the number of Newton steps is obtained in the same way as in  \cref{thm:global_convergence_sublinear_nonconv}.
\end{proof}

\begin{remark}[Choice of the parameters $m,\alpha,\beta$]
    First, observe that the proof of the convergence rate in \eqref{eq:conv_rate} was based on \eqref{eq:functionValuesGradientRelation_convex}, which in turn follows from the first inequality in \eqref{eq:stop_cond_alg} and convexity. Thus, if it is known that $F$ is convex, the second inequality in \eqref{eq:stop_cond_alg} can be omitted, and the algorithm would simplify. This also means that in the convex case $\beta$ can be chosen arbitrarily. Nevertheless, we keep the second inequality in \eqref{eq:stop_cond_alg} to make the algorithm agnostic to the knowledge of whether the problem is convex or not.
    Further, as one can see, the rate in \eqref{eq:conv_rate} is proportional to $m/\alpha$ and the smaller that number, the better the rate. This, in particular, means that $\alpha$ should be chosen as large as possible and $m$ should be as small as possible. Recall from \cref{lem:acceptance_inner_nonconv}  that $0 < \alpha \leq 1\slash 2$, and $1 \leq m$. Thus, the optimal choice in the convex case is $\alpha = 1\slash 2, m=1$. Yet, the standard choice $\alpha = 1\slash 2, m=2$ also works.
\end{remark}

We can show that the convexity of $F$ also allows for improving the global nonasymptotic convergence rate in terms of the minimal norm  of the subgradient 
on the trajectory of \cref{alg:proximal_newton}. Namely, in the next result, we obtain an improved $\mathcal{O}(1\slash k)$ rate compared to $\mathcal{O}(1\slash \sqrt{k})$ in \cref{thm:global_convergence_sublinear_nonconv}.

\begin{theorem}[Global convergence rate for subgradients, convex case] Let assumptions of \cref{thm:global_convergence_sublinear} hold. Then, the following global nonasymptotic convergence rate holds:
\begin{equation}
\label{eq:conv_rate_grad}
     \min_{i=k,\ldots,2k} \norm{ F'(x_{i+1})}{*} \leq 2\sqrt{\frac{\max\{2mL,\Lambda_0\}\norm{F'(x_0)}{*}D_0}{\alpha k}}\exp\left(-\frac{k}{8}\right) + \frac{4D_0 \max\{2mL,\Lambda_0\}}{\alpha k}, k\geq 0.
\end{equation}
\end{theorem}
\begin{proof}
The proof is an adaptation of \cite{NesterovSmallGradients} to our setting with a more complicated rate from \eqref{eq:conv_rate}.
Summing \eqref{eq:functionValuesGradientRelation_convex} from $k$ to $2k$, recalling from \eqref{eq:F_r_g_notaion} the notation $g_{i+1}  \coloneqq \norm{F'(x_{i+1})}{*}$, using  \eqref{eq:conv_rate}, and the bound $\lambda_k\leq \overline{\lambda}$ from \cref{lem:acceptance_inner_nonconv}, we obtain 
\begin{align*}
    &\frac{\alpha k}{\overline{\lambda}} \min_{i=k,\ldots,2k} \norm{ F'(x_{i+1})}{*}^2 \overset{\lambda_k\leq \overline{\lambda},\eqref{eq:F_k_decr}}{\leq} F(x_{2k+1})-F^* + \sum_{i={k}}^{2k}\frac{\alpha g_{i+1}^2}{\lambda_i} \\
    \overset{\eqref{eq:functionValuesGradientRelation_convex}}&{\leq} F(x_{2k+1})-F^* + \sum_{i={k}}^{2k} (F(x_{i}) - F(x_{i+1} )) =F(x_k) - F^* \overset{\eqref{eq:conv_rate}}{\leq}g_0D_0\exp\left(-\frac{k}{4}\right) + \frac{2D_0^2 \max\{2mL,\Lambda_0\}}{\alpha k}.
\end{align*}
Rearranging, using the inequality $\sqrt{a^2+b^2}\leq 2(a+b)$ for $a,b \geq0$, and recalling from \eqref{eq:lambda_bound} that $\overline{\lambda}=\max\{2mL,\Lambda_0\}$, we obtain \eqref{eq:conv_rate_grad}.

\end{proof}
Note that if in addition $\lambda_k \to 0$ as $k\to \infty$, the rate in the above result improves to $\so(1\slash k)$.

\section{Local linear convergence}
\label{sec:local convergence}
In this section, we return to the nonconvex setting and show local nonasymptotic linear convergence of \cref{alg:proximal_newton} to a local minimum of problem \eqref{eq:opt} under local versions of \cref{ass:main_inequality,ass:PL}, and a local QG condition. This result complements \cref{thm:global_convergence_nonconv_PL} in the sense that all the assumptions are made locally on a ball around a local minimum, whereas in \cref{thm:global_convergence_nonconv_PL} we made for example \cref{ass:main_inequality} globally on $\dom f$. Our motivation for this result comes from the classical results on SSN which make only local assumptions of semismoothness and non-degeneracy at a solution and prove local superlinear convergence, see, e.g., Theorem 2.11 in \cite{hintermuller2010semismooth}. The main result of this section is that under similar assumptions, but in a more general setting, \cref{alg:proximal_newton} converges to a local minimum. In the next section, we strengthen this result by showing that under additional assumptions including semismoothness the convergence is actually superlinear, like in the classical theory.

In this section, we let $x^*$ be an isolated local minimum in problem \eqref{eq:opt}.
This means that there exists some $R>0$ such that 
\begin{equation}
    \label{eq:isolated_min_prop}
    F(x) \geq F(x^*), \;\; \dist(x,S) =\norm{x-x^*}{} , \;\; \forall x \in  B_R(x^*).
\end{equation}

\begin{theorem}[Local linear convergence]
\label{thm:local_convergence_1}
    Let, for problem \eqref{eq:opt},   \cref{ass:basic,ass:subproblem} hold. Let $x^*$ be an isolated local minimum in \eqref{eq:opt} with the local minimal value $F(x^*)$ and let $R$ be such that \eqref{eq:isolated_min_prop} holds, \cref{ass:main_inequality} holds locally on $B_R(x^*)$ (instead of $\dom f$), the PL condition \cref{ass:PL} holds locally on $B_R(x^*)$ (instead of $\dom F$), %
    $H(x) \succeq \mu I$ for some  $\mu>0$  and for all $x \in B_R(x^*)$, 
    and the quadratic growth (QG) condition  holds locally on $B_R(x^*)$, i.e., 
        \begin{equation}
        \label{eq:QG_local_1}
        \frac{\mu}{2}\dist(x,S)^2 
        \leq F(x) - F(x^*), \quad \forall x  \in B_R(x^*) .
        \end{equation}
     Let also $(x_k)_{k \in \mathbb{N}}$ be the iterates of \cref{alg:proximal_newton} with starting point  $x_0\in \dom F$. Then, there exists $\Delta$ such that for any $x_0 \in B_R(x^*)$ with $\norm{F'(x_0)}{*}\leq \Delta$, $x_k$ converges to $x^*$ linearly and $F(x_k)$ converges to $F(x^*)$ linearly (explicit rates are given in the proof).  If in addition $\lambda_k\to 0$ as $k\to \infty$, the convergence is superlinear.
\end{theorem}
\begin{proof}
    Define constants $\varkappa\coloneqq \frac{\overline{\lambda}}{\overline{\lambda} + 2\mu\beta\alpha^2  } <1$,  $c\coloneqq 1+\sqrt{\frac{\overline{\lambda}}{\mu\beta\alpha^2}}$, where $\overline{\lambda}$ is defined in \cref{lem:acceptance_inner_nonconv} (see  \eqref{eq:lambda_bound}) and choose 
    \begin{equation}
    \label{eq:thm:local_convergence_1_proof_1}
        0 < \Delta \leq \frac{\mu R}{1+c}.%
    \end{equation}
    Recall from \eqref{eq:F_r_g_notaion} that $g_k\coloneqq \norm{F'(x_k)}{*}$. We slightly abuse notation for the purpose of this proof and define $F_k\coloneqq F(x_k)-F(x^*)$, where $F(x^*)$ is the minimal value of $F$ on $B_R(x^*)$, cf. with global minimal value $F^*$ in \eqref{eq:F_r_g_notaion}. 
    We also note that the local PL condition on $B_R(x^*)$ implies that whenever $x_k \in B_R(x^*)$, we have that
    \begin{equation}
        \label{eq:PL_local_1_refined} %
        F(x_k)-F(x^*) \leq \frac{1}{2\mu} \norm{F'(x_k)}{*}^2.
    \end{equation}    
    Let $x_0 \in B_R(x^*)$ be such that $ \norm{F'(x_0)}{*}\leq \Delta$. 
    We proceed by induction and prove that for $k\geq0$,
    \begin{align}
        \norm{F'(x_k)}{*} &\leq c\Delta \label{eq:thm:local_convergence_1_proof_ind_1}, \\
        F_k=F(x_k)-F(x^*) &\leq \varkappa^k \frac{\Delta ^2}{2 \mu} \label{eq:thm:local_convergence_1_proof_ind_2}, \\
        \norm{x_k-x^*}{} &\leq \varkappa^{k/2} \frac{\Delta }{ \mu} \label{eq:thm:local_convergence_1_proof_ind_3}.
    \end{align}
    Since $x_0 \in B_R(x^*)$, we can apply the assumptions of the theorem to show that the base case holds. Indeed, since $c\geq1$, we have that $\norm{F'(x_0)}{*}\leq \Delta \leq c\Delta$. Since $x^*$ is an isolated local minimum and the QG condition holds, we have  
    \begin{equation}
    \label{eq:thm:local_convergence_1_proof_2}    
        \frac{\mu}{2}\norm{x_0-x^*}{}^2 \overset{\eqref{eq:isolated_min_prop},\eqref{eq:QG_local_1}}{\leq} F(x_0)-F(x^*) =F_0\overset{\eqref{eq:PL_local_1_refined}}{\leq}  \frac{1}{2\mu} \norm{F'(x_0)}{*}^2\leq \frac{\Delta ^2}{2 \mu} 
        \Rightarrow \norm{x_0-x^*}{} \leq \frac{\Delta}{\mu}. %
    \end{equation}
    Thus, \eqref{eq:thm:local_convergence_1_proof_ind_1},  \eqref{eq:thm:local_convergence_1_proof_ind_2}, \eqref{eq:thm:local_convergence_1_proof_ind_3} hold for $k=0$.

    We now proceed with the induction step. Let \eqref{eq:thm:local_convergence_1_proof_ind_1},  \eqref{eq:thm:local_convergence_1_proof_ind_2}, \eqref{eq:thm:local_convergence_1_proof_ind_3} hold. Note that \eqref{eq:thm:local_convergence_1_proof_ind_3}, \eqref{eq:thm:local_convergence_1_proof_1}, and $\varkappa<1$ imply that $\norm{x_k-x^*}{} \leq \varkappa^{k/2} \frac{\Delta }{ \mu} \leq R$ and, hence, $x_k \in B_R(x^*)$.  
    Thus, by the assumption that $H(x) \succeq \mu I$ on $B_R(x^*)$ with $ \mu > 0$, %
     recalling the notation $x_+(\lambda,x)$ from \eqref{eq:PLMSN_step}, we obtain from \cref{lem:inequality_lemma} that 
    \begin{equation*}
        \norm{x_+(\lambda,x_k)-x_k}{} \overset{\eqref{eq:rk_leq_gk_over_lambda_k_new}}{\leq}  \frac{\norm{F'(x_k)}{*}}{\lambda + \mu} \overset{\eqref{eq:thm:local_convergence_1_proof_ind_1}}{\leq} \frac{c\Delta}{\mu}, \quad \forall \lambda \geq 0. %
    \end{equation*}
    This in combination with \eqref{eq:thm:local_convergence_1_proof_ind_3} and $\varkappa <1$ gives  
    \begin{equation*}
        \norm{x_+(\lambda,x_k)-x^*}{} \leq \norm{x_+(\lambda,x_k)-x_k}{} + \norm{x_k-x^*}{} \leq    \frac{c\Delta}{\mu} +\frac{\Delta}{\mu}\overset{\eqref{eq:thm:local_convergence_1_proof_1}}{\leq} R, \quad \forall \lambda \geq 0.
    \end{equation*}
    Hence, for all $\lambda \geq 0$, we have $x_+(\lambda,x_k) \in B_R(x^*)$ and $[x_k,x_+(\lambda,x_k)] \subset B_R(x^*)$. By the assumption that \cref{ass:main_inequality} holds on $B_R(x^*)$, we observe that it holds also on $[x_k,x_+(\lambda,x_k)]$ for all $\lambda\geq0$. Thus, \cref{lem:acceptance_inner_nonconv} applies. Whence, the next iterate $x_{k+1}$ is well defined, we have $\lambda_k\leq \overline{\lambda}$ by inequality \eqref{eq:lambda_bound}, and
    \begin{equation}
    \label{eq:thm:local_convergence_1_proof_5}    
        F_k-F_{k+1}=F(x_k)-F(x_{k+1}) \overset{\eqref{eq:functionValuesGradientRelation}}{\geq}  \frac{\beta\alpha^2}{ \lambda_k}\norm{F'(x_{k+1})}{*}^2 \overset{\lambda_k\leq \overline{\lambda}}{\geq} \frac{\beta\alpha^2}{ \overline{\lambda}}\norm{F'(x_{k+1})}{*}^2 \geq 0.
    \end{equation}
    Since $x_{k+1}= x_+(\lambda_k,x_k)\in B_R(x^*)$, we can use the localised PL condition \cref{ass:PL} at $x_{k+1}$ and obtain
    \begin{equation}
    \label{eq:thm:local_convergence_1_proof_6}    
        F_k-F_{k+1} \overset{\eqref{eq:thm:local_convergence_1_proof_5}}{\geq}   \frac{\beta\alpha^2}{  \lambda_k}\norm{F'(x_{k+1})}{*}^2  \overset{\eqref{eq:PL_local_1_refined}}{\geq} \frac{2\mu\beta\alpha^2}{ \lambda_k} (F(x_{k+1})-F(x^*)) = \frac{2\mu\beta\alpha^2}{\lambda_k} F_{k+1}.
    \end{equation}
    Rearranging and using the QG condition, we obtain
    \begin{align}   
        \frac{\mu}{2}\norm{x_{k+1}-x^*}{}^2 \overset{\eqref{eq:isolated_min_prop},\eqref{eq:QG_local_1}}&{\leq} F(x_{k+1})-F(x^*) =F_{k+1}\overset{\eqref{eq:thm:local_convergence_1_proof_6}}{\leq} \frac{F_{k}}{1+\frac{2\mu\beta\alpha^2}{ \lambda_k}} =\frac{\lambda_kF_{k}}{\lambda_k + 2\mu\beta\alpha^2  } \overset{\lambda_k\leq \overline{\lambda},\eqref{eq:thm:local_convergence_1_proof_ind_2}}{\leq} \frac{\overline{\lambda}}{\overline{\lambda} + 2\mu\beta\alpha^2  } \varkappa^k \frac{\Delta^2}{2\mu}       
        \notag \\ 
        &= \varkappa^{k+1} \frac{\Delta^2}{2\mu} \Rightarrow \norm{x_{k+1}-x^*}{}  \leq \varkappa^{(k+1)/2} \frac{\Delta}{\mu}. \label{eq:thm:local_convergence_1_proof_7}  
    \end{align}
    Thus, we obtain \eqref{eq:thm:local_convergence_1_proof_ind_2} and \eqref{eq:thm:local_convergence_1_proof_ind_3} for $k+1$.
    Finally, we show \eqref{eq:thm:local_convergence_1_proof_ind_1} for $k+1$ by contradiction.  Assume the opposite, i.e., 
    \begin{equation}
    \label{eq:thm:local_convergence_1_proof_8}
        g_{k+1}\overset{\eqref{eq:F_r_g_notaion}}{\coloneqq }\norm{F'(x_{k+1})}{*}>c\Delta. 
    \end{equation}
    Since $x_k,x_{k+1} \in B_R(x^*)$, by the first inequality in \eqref{eq:isolated_min_prop},  definition of $F_k,F_{k+1}$, and by induction from \eqref{eq:thm:local_convergence_1_proof_5}, we have  $0 \leq F_{k+1} \leq F_{k} \leq F_0$. Hence,
    \begin{align*}   
       0 \leq  F_{k+1}\overset{\eqref{eq:thm:local_convergence_1_proof_5}}&{\leq} F_k - \frac{\beta\alpha^2}{ \overline{\lambda}}g_{k+1}^2 \overset{F_k\leq F_0}{\leq} F_0 - \frac{\beta\alpha^2}{ \overline{\lambda}}g_{k+1}^2 \overset{\eqref{eq:thm:local_convergence_1_proof_2},\eqref{eq:thm:local_convergence_1_proof_8}}{<}     \frac{\Delta^2}{2\mu} - \frac{\beta\alpha^2}{ \overline{\lambda}} c^2\Delta^2\\
        & <   \frac{\Delta^2}{2\mu} - \frac{\beta\alpha^2}{ \overline{\lambda}} \Delta^2\left(\sqrt{\frac{\overline{\lambda}}{\mu\beta\alpha^2}} \right) ^2< 0,
    \end{align*}
    which is a contradiction. Thus,  we have \eqref{eq:thm:local_convergence_1_proof_ind_1} for $k+1$. 
    
    We finish the proof by observing that \eqref{eq:thm:local_convergence_1_proof_ind_3} implies that $x_k \to x^*$ linearly:
    \[
        \norm{x_k-x^*}{} \leq  \left(1-\frac{2\mu\beta\alpha^2}{\overline{\lambda} + 2\mu\beta\alpha^2  }\right)^{k/2}\frac{R }{ 1+c} \overset{\eqref{eq:lambda_bound}}{\leq} \exp\left(-\frac{2\beta\alpha^2\mu}{4\beta\alpha^2\mu+2\max\{2mL,\Lambda_0\}} \cdot k\right)\cdot\frac{R }{ 1+c},
    \]
    and that \eqref{eq:thm:local_convergence_1_proof_ind_2} implies $F(x_k)$ converges to $F(x^*)$ linearly:
    \[
        F(x_k)-F(x^*) \leq  \left(1-\frac{2\mu\beta\alpha^2}{\overline{\lambda} + 2\mu\beta\alpha^2  }\right)^{k}\frac{\mu R^2 }{ 2(1+c)^2} \overset{\eqref{eq:lambda_bound}}{\leq} \exp\left(-\frac{2\beta\alpha^2\mu}{\beta\alpha^2\mu+\max\{2mL,\Lambda_0\}} \cdot k\right)\cdot\frac{\mu R^2 }{ 2(1+c)^2}.
    \]
    Moreover, from the first two inequalities and first two equalities in \eqref{eq:thm:local_convergence_1_proof_7}, we see that if $\lambda_k\to 0$ as $k\to \infty$, the convergence of both $x_k$ and $F(x_k)$ is superlinear.
\end{proof}

We remark that assuming QG condition \eqref{eq:QG_local_1},  $H(x) \succ 0$ on $B_R(x^*)$ and that $x^*$ is an isolated minimum  may be considered as a price for making \cref{ass:main_inequality} only locally, cf. \cref{thm:global_convergence_nonconv_PL}, where it is made globally. Indeed, we use those assumptions to guarantee that the iterates of the algorithm stay in the ball $B_R(x^*)$ where we can apply local \cref{ass:main_inequality}. In future works, it would be interesting to avoid these assumptions and prove local convergence based only on the PL condition in the spirit of \cite{frankel2015splitting}. In particular, we hope that this will allow us to omit the assumption that $x^*$ is isolated, whose main purpose currently is to establish the   equality in \eqref{eq:isolated_min_prop} that allows for using the QG condition. In fact, we expect that the QG condition follows from the assumed PL condition \cite{jin2023growth}, however, we were not able to find a proof of this statement needed for our precise setting. We also note that the assumption of $x^*$ being an isolated minimum is, in a sense, standard since many existing papers use (local) strong convexity, see e.g., \cite{doikov2024super,doikov2024gradient}, or other assumptions \cite{facchinei1995minimization} that imply that the convergence is established to an isolated (local) minimum.

As already mentioned, our inspiration for the above result comes from the classical results on the local convergence of SSN (with $\psi \equiv 0$), see, e.g., Theorem 2.11 in \cite{hintermuller2010semismooth}, where all the assumptions are made locally. In particular, it is assumed that the generalised derivative $H(x)$ is non-singular at $x^*$ and $f'$ is semismooth at $x^*$. Since $H(x)$ is non-singular at $x^*$ it is non-singular in some $B_R(x^*)$ and there is a $\mu>0$ that lower bounds the smallest eigenvalue of $H(x)$. This justifies the assumption $H(x) \succ 0$ on $B_R(x^*)$. Moreover, this means that $f$ behaves like a strongly convex functional locally, which implies both local PL and QG conditions. Thus, our setting is close to the classical one, yet we allow for some nonconvexity.

We conclude this subsection with the following technical result, which is used in the proofs of  superlinear convergence in the next  section. This result provides a relation of the distance to solution $\norm{x_k-x^*}{}$ to the distance between two consecutive iterates $\|x_{k+1} - x_k\|$, analogous to  \cref{lem:useful_bounds_new}, but under local assumptions. 

\begin{lem}\label{lem:useful_bounds_local}%
Let assumptions of \cref{thm:local_convergence_1} hold.
Then, there exist constants $c_1,c_2>0$ such that, for all $k\geq 0$,
\begin{equation}
        \norm{x_{k+1}-x^*}{} \leq c_1 \norm{x_k -x_{k+1}}{} \quad \text{and} \quad  \norm{x_k-x^*}{} \leq c_2 \norm{x_k -x_{k+1}}{}. \label{eq:useful_bounds_local}
    \end{equation}
\end{lem}
\begin{proof}
As shown in the proof of \cref{thm:local_convergence_1}, we have that $x_k \in B_R(x^*)$ for all $k\geq 0$. Thus, 
\begin{equation*}
    \frac{\mu}{2}\norm{x_{k+1}-x^*}{}^2 \overset{\eqref{eq:isolated_min_prop},\eqref{eq:QG_local_1}}{\leq} F(x_{k+1})-F(x^*) \overset{\eqref{eq:PL_local_1_refined}}{\leq}  \frac{1}{2\mu} \norm{F'(x_{k+1})}{*}^2.
\end{equation*}
By \eqref{eq:g_k_leq_r_k} and the  bound $\lambda_k \leq \overline{\lambda}$ from \cref{lem:acceptance_inner_nonconv}, we have
\begin{equation*}
     \norm{F'(x_{k+1})}{*}\overset{\eqref{eq:g_k_leq_r_k}}{\leq} \frac{\lambda_k}{\alpha}\norm{x_k -x_{k+1}}{} %
     \leq \frac{\overline{\lambda}}{\alpha}\norm{x_k -x_{k+1}}{},
\end{equation*}
which in combination with the previous inequality gives the first inequality in \eqref{eq:useful_bounds_local} with the constant $c_1=\frac{\overline{\lambda}}{\alpha \mu}$. The second inequality with $c_2=1+c_1$ follows from the triangle inequality.
\end{proof}

\section{Superlinear convergence under higher regularity}
\label{sec:fast}
In this section, based on the additional assumptions on the regularity of $f'$, we further refine the convergence rates established in previous sections. We start with several technical lemmas that are used in what follows. Then, we show that if $\HH$ is finite dimensional and $f$ is twice continuously differentiable, i.e., $f \in C^2$, then in \cref{alg:proximal_newton} $\lambda_k\to0$ and the convergence rates proved in previous sections asymptotically improve. After that, we move on to the main setting of the paper and consider the behaviour of \cref{alg:proximal_newton} under a semismoothness assumption, which allows us to obtain its superlinear convergence. Finally, based on the active manifold identification framework, we show that superlinear convergence can also occur when $f$ is not locally $C^2$, but instead satisfies a combination of partial smoothness and semismoothness of the derivative. To simplify the proofs, we focus on the setting when $F$ is locally convex, but this assumption can be slightly relaxed; we defer the details to future work.

\subsection{Technical preliminaries}
\label{sec:further_tools}
We start with several technical results that are crucial for establishing superlinear convergence of \cref{alg:proximal_newton}. The first one establishes several useful inequalities for $x_+(\lambda,x_k)$.

\begin{lem}\label{lem:good_bounds}%
Consider iteration $k \geq 0$ of \cref{alg:proximal_newton}
and let $0<\lambda < \lambda_k$. Let  $H(x_k) \succeq 0 $ and, for some constant $c>0$, the following inequality hold:
\begin{equation}
    \norm{x_{k+1}-x^*}{} \leq c \norm{x_k -x_{k+1}}{}. 
        \label{eq:useful_bounds_abstract_first_1}
\end{equation}
Then, we have
         \begin{align}
             \| x_{+}(\lambda,x_k) - x^*\|  &\leq
             \frac{(2c+1)\lambda_k -(1+c)\lambda}{\lambda_k}\| x_{+}(\lambda,x_k)  - x_k\|,\label{eq:good_bound_1}\\
             \norm{x_k-x^*}{} &\leq  \frac{(2c+2)\lambda_k -(1+c)\lambda}{\lambda_k}\| x_{+}(\lambda,x_k)  - x_k\|.\label{eq:good_bound_2}
         \end{align}

\end{lem}
Before we start the proof, for the reader's convenience we note that when $\lambda$ is chosen as $\lambda_k\slash 2$, the constants in \eqref{eq:good_bound_1} and \eqref{eq:good_bound_2} are   $\frac{1+3c}{2}$ and $\frac{3+3c}{2}$ respectively.
\begin{proof}[Proof of \cref{lem:good_bounds}]
Applying two times the triangle inequality and one time \cref{lem:lambda_dependence}, recalling from \cref{step:k+1_update} of \cref{alg:proximal_newton} and \eqref{eq:PLMSN_step} that $x_{k+1}=x_+(\lambda_k, x_k)$, we obtain
\begin{align*}
    \| x_{+}(\lambda,x_k) - x^*\| &\leq \| x_{+}(\lambda,x_k) - x_{k+1}\| + \|x_{k+1} -x^*\|\\
    \overset{\eqref{eq:useful_bounds_abstract_first_1}}&{\leq} \| x_{+}(\lambda,x_k) - x_{k+1}\| + c\|x_{k+1} -x_k\|\\
    &\leq \| x_{+}(\lambda,x_k) - x_{k+1}\| + c(\| x_{+}(\lambda,x_k) - x_{k+1}\|+\|x_{+}(\lambda,x_k) -x_k\|)\\
    &=(1+c)\| x_{+}(\lambda,x_k) - x_{k+1}\|+c\|x_{+}(\lambda,x_k) -x_k\|)\\
   \overset{\eqref{eq:lambda_dependence}}&{\leq}  (1+c)\frac{\lambda_k-\lambda}{\lambda_k}\|x_{+}(\lambda,x_k)  - x_k\| + c\|x_{+}(\lambda,x_k) -x_k\|, 
\end{align*}
which is \eqref{eq:good_bound_1}.
Applying once again the triangle inequality, we obtain \eqref{eq:good_bound_2}:
\begin{align*}
    \norm{x_k-x^*}{} &\leq \norm{x_k-x_+(\lambda, x_k)}{} + \norm{x_+(\lambda, x_k)-x^*}{}\\
    \overset{\eqref{eq:good_bound_1}}&{\leq}   \norm{x_k-x_+(\lambda, x_k)}{}  + \frac{(2c+1)\lambda_k -(1+c)\lambda}{\lambda_k}\| x_{+}(\lambda,x_k)  - x_k\|.
\end{align*}
\end{proof} 
Note that the technical condition \eqref{eq:useful_bounds_abstract_first_1} is guaranteed by \cref{lem:useful_bounds_new} or \cref{lem:useful_bounds_local}. The next result establishes the convergence of non-accepted trial points $x_+(\lambda_k/2,x_k)$ to $x^*$ when accepted points $x_k$ converge to $x^*$.

\begin{lem}\label{lem:strange_convergence}
Let $(x_k)_{k \in \mathbb{N}}$ be the iterates of \cref{alg:proximal_newton} and $x_k \to x^*$ as $k \to \infty$. Let also $H(x_k) \succeq 0 $ 
for sufficiently large $k$. Then, $x_+(\lambda_k/2,x_k) \to x^*$ as $k \to \infty$.
\end{lem}
\begin{proof}
For any sufficiently large $k$ and $0<\lambda \leq \lambda_k$, by the triangle inequality, the definition $x_{k+1}=x_+(\lambda_k,x_k)$ in \cref{step:k+1_update} of \cref{alg:proximal_newton} and \eqref{eq:PLMSN_step}, and \cref{lem:lambda_dependence}, we have
 \begin{align*}
 \| x_k -x_+(\lambda,x_k)\| \leq \|x_{k+1} - x_+(\lambda,x_k)\| + \|x_{k+1} - x_k\| \overset{\eqref{eq:lambda_dependence}}{\leq} \frac{\lambda_k-\lambda}{\lambda_k}\| x_k -x_+(\lambda,x_k)\| + \|x_{k+1} - x_k\|.
 \end{align*}
Rearranging, we obtain
\begin{equation*}
    \| x_k -x_+(\lambda,x_k)\| \leq \frac{\lambda_k}{\lambda} \| x_{k+1} -x_k\|.
\end{equation*}
Now, utilising again the triangle inequality, we calculate
\begin{align*}
    \|x_+(\lambda,x_k) - x^*\| &\leq \|x_+(\lambda,x_k)-x_{k}\| + \|x_{k} -x^*\|\leq \frac{\lambda_k}{\lambda} \| x_{k+1} -x_k\| + \|x_{k} -x^*\|.
\end{align*}
The claim follows by taking $\lambda = \lambda_k/2$, using the upper bound $\lambda_k \leq \overline{\lambda}$ from \cref{lem:acceptance_inner_nonconv}, and taking the limit.%
\end{proof}

The following result will be crucial to show that in \cref{alg:proximal_newton} $\lambda_k \to 0$ when $f$ has additional regularity. On a high level, it can be interpreted as $\lambda_k \to 0$ if the constant $L$ in \cref{ass:main_inequality} goes to 0 around the trajectory of \cref{alg:proximal_newton}. It is, in some sense, natural to expect this from \eqref{eq:lambda_bound}, and we now prove it in rigour. 
\begin{lem}\label{lem:helper_lambdak_zero}
Let $F$ be locally convex around $x^*$. Let $(x_k)_{k \in \mathbb{N}}$ be the iterates of \cref{alg:proximal_newton} and $x_k \to x^*$ as $k \to \infty$. Let also $H(x_k) \succeq 0$  for sufficiently large $k$.  
         If there exists a sequence $(\varepsilon_k)_{k \in \mathbb{N}} \subset \R_+$ such that $\varepsilon_k \to 0$ and, for sufficiently large $k$,
         \begin{align}\label{eq:eps_condition}
            \norm{ f'(x_{+}(\lambda_k/2, x_k)) - f'(x_{k}) - H(x_k)(x_{+}(\lambda_k/2, x_k) - x_k) }{*} 
            &\leq \varepsilon_k \norm{x_k - x_+(\lambda_k/2,x_k)}{},
        \end{align}
        then, we have $\lambda_k \to 0$ as $k \to \infty$. 
        \end{lem}
        \begin{proof} For brevity, given an index $k \geq 0$, we denote by $x_{k,+} \coloneqq x_+(\lambda_k/2,x_k)$  the solution of the auxiliary problem \eqref{eq:PLMSN_step}   
        given current iterate $x_k$ and accepted regularisation parameter $\lambda_k>0$. Moreover, we set $r_{k,+}\coloneqq  \norm{x_k - x_{k,+}}{}$. 
        By assumption, there exists $K_1$ such that \eqref{eq:eps_condition} holds for all $k\geq K_1$. Thus, for any $k\geq K_1$, using \eqref{eq:eps_condition}, \cref{lem:inequality_lemma_2} with $x=x_k$, $x_+=x_{k,+}$, $\varphi(r) = \varepsilon_k r$, and $\lambda = \lambda_k/2$, and that $\alpha \leq 1/2$, we obtain
        \begin{align}
       \langle F'(x_{k,+}),x_k - x_{k,+} \rangle
        \overset{\eqref{eq:progress_prelim},\eqref{eq:eps_condition}}&{\geq} \frac{1}{\lambda_k}\norm{F'(x_{k,+})}{*}^2 + \left(\frac{\lambda^2_k- 4\varepsilon_k^2}{4\lambda_k}\right) r_{k,+}^2\nonumber\\ 
        &\geq \frac{2\alpha}{\lambda_k}\norm{F'(x_{k,+})}{*}^2 + \left(\frac{\lambda^2_k- 4\varepsilon_k^2}{4\lambda_k}\right) r_{k,+}^2.\label{eq:final_inequality_superlinear} 
    \end{align}
    By the local convexity of $F$ around $x^*$, there exists a ball $B_R(x^*)$ such that $F$ is convex on that ball. Since by the assumptions of the lemma $x_k \to x^*$ and, for sufficiently large $k$, $H(x_k) \succeq 0$, we have by \cref{lem:strange_convergence} that $ x_{k,+} \to x^*$ as well. Thus, there exists $K_2$ s.t. for all $k \geq K_2$, we have $x_k,x_{k,+}  \in B_R(x^*)$. Applying convexity of $F$ on $B_R(x^*)$, using that the first term on the right-hand side in \eqref{eq:final_inequality_superlinear} is nonnegative, and adding and subtracting $\beta\lambda_kr_{k,+}^2/2$, we obtain from \eqref{eq:final_inequality_superlinear}  for all $k \geq \max\{K_1,K_2\}$
    \begin{align}
       F(x_k) - F(x_{k,+}) \geq \langle F'(x_{k,+}),x_k - x_{k,+} \rangle
        &\geq  \frac{\beta\lambda_kr_{k,+}^2}{2} + \left(\frac{\lambda^2_k- 4\varepsilon_k^2 -2\beta\lambda_k^2}{4\lambda_k}\right) r_{k,+}^2\nonumber\\
        &=  \frac{\beta\lambda_kr_{k,+}^2}{2} + \left(\frac{(1-2\beta)\lambda^2_k- 4\varepsilon_k^2}{4\lambda_k}\right) r_{k,+}^2.
        \label{eq:final_inequality_superlinear_2}  
    \end{align}
    From now on, we assume $k \geq \max\{K_1,K_2\}$. From \cref{step:k+1_update} of \cref{alg:proximal_newton}, we know that $\lambda_k =2^{j_k}\Lambda_k$ for some $j_k\geq 0$. 
    Consider the two cases: when \( j_k = 0 \) and when \( j_k > 0 \). In the first case, i.e., when \( j_k = 0 \), acceptance occurs immediately in the first round of the inner loop. By \cref{step:k+1_update} of \cref{alg:proximal_newton}, \( \Lambda_k = \lambda_{k-1}/2 \), and hence in this case \( \lambda_k = 2^{j_k} \Lambda_k = 2^0 \Lambda_k = \Lambda_k =\lambda_{k-1}/2 \).
    In the second case, i.e., when \( j_k > 0 \), \( \lambda_k/2 \) must have been rejected. This implies that the candidate point \( x_+\left( \lambda_k/2, x_k \right) \) and candidate regularisation parameter $\lambda_k/2$ did not satisfy the acceptance conditions \eqref{eq:stop_cond_alg}. Thus, we have
    \begin{equation}
    \label{eq:lambda_not_accepted}
        \langle F'(x_{k,+}),x_k - x_{k,+} \rangle < \frac{2\alpha}{\lambda_k}\norm{F'(x_{k,+})}{*}^2 \quad \text{ or } \quad F(x_k) - F(x_{k,+}) < \frac{\beta\lambda_kr_{k,+}^2}{2}. 
    \end{equation}
    If the first inequality in \eqref{eq:lambda_not_accepted}  holds, from \eqref{eq:final_inequality_superlinear} we must have \( \lambda_k \leq 2\varepsilon_k \) since otherwise, we are led to a contradiction. If the second inequality  in \eqref{eq:lambda_not_accepted} holds, then in the same way, \eqref{eq:final_inequality_superlinear_2} and the fact $\beta \leq (m-1)/(2m) <1/2$ (see \cref{step:input} of \cref{alg:proximal_newton}) imply that $\lambda_k \leq 2\varepsilon_k\slash \sqrt{1-2\beta}$.
    Thus, in both cases we have $\lambda_k\leq \max\{2, 2\slash \sqrt{1-2\beta}\}\varepsilon_k$. Combining this with the case \( j_k = 0 \) and denoting $\gamma\coloneqq  \max\{2, 2\slash \sqrt{1-2\beta}\}$, we have, for $k\geq \max\{K_1,K_2\}$,
    \begin{align*}
        \lambda_k \leq \begin{cases}
            \frac{\lambda_{k-1}}{2}, & j_k = 0,\\
            \gamma\varepsilon_k, & j_k > 0.
        \end{cases}
    \end{align*}
    Now let us show that this implies $\lambda_k \to 0$. Define $\mathcal{I} \coloneqq  \{k \geq 0 : j_k > 0\} = \{k_1 <   k_2 <  \dots\}.$ Thus if $k \in \mathcal{I}$, $\lambda_k \leq \gamma\varepsilon_k$, whereas if $k \not\in \mathcal{I}$, we have $\lambda_k \leq \lambda_{k-1}\slash 2$. Now, if $|\mathcal{I}| < \infty$, then for sufficiently large $k$, we have $\lambda_k \leq \lambda_{k-1}\slash 2$ and hence $\lambda_k \to 0$. So let us suppose that $|\mathcal{I}| = \infty$. Fix $\delta >0$ and choose $i_0 $ such that $\gamma\varepsilon_k \leq \delta$ for all $k \geq k_{i_0}$. Take $k \geq k_{i_0}$. Then either $k$ belongs to $\mathcal{I}$ or not. If $k \not\in \mathcal{I}$, there exists $i\geq i_0$ such that $k_i < k < k_{i+1}$ with $k_i,k_{i+1} \in \mathcal{I}$ (recall $\{k_i\}$ from the definition of $\mathcal{I}$). Thus, we have
    \[
    \lambda_k \leq \frac{\lambda_{k_i}}{2^{k-k_i}} \leq \frac{\gamma\varepsilon_{k_i}}{2^{k-k_i}} \leq \gamma\varepsilon_{k_i} \leq \delta,
    \]
    where we used that  $k-k_i \geq 0$ and hence $2^{k-k_i}\geq 1$.
    On the other hand, if $k \in \mathcal{I}$, we have that $k=k_i$ for some $i \geq i_0$, and so  $\lambda_k \leq \gamma\varepsilon_k \leq \delta$. 
    In total, we infer \( \lambda_k \leq \delta \) for all \( k \geq k_{i_0} \). Since \( \delta > 0 \) was chosen arbitrarily, it follows that \( \lambda_k \to 0 \) as \( k \to \infty \).
\end{proof}
\begin{remark}
When  $F$ is globally convex, $\dim(\mathcal{H}) < \infty$, and the minimiser $x^*$ is unique, then $x_k \to x^*$ holds automatically in the above lemma.
\end{remark}

\subsection{Faster asymptotic rates in the finite-dimensional \texorpdfstring{$C^2$}{C\textasciicircum 2} setting}  \label{sec:c2_setting}

In this subsection, we present our first result on improved asymptotic rates of \cref{alg:proximal_newton} when $\HH$ is finite-dimensional and $f$ possesses additional regularity. Namely, we assume that $f\in C^2$, i.e., $f$ is twice continuously differentiable. The main idea is to prove that if $x_k \to x^*$ as $k \to \infty$, then we also have that $\lambda_k \to 0$ as $k \to \infty$. According to the results of the previous sections, the latter implies faster asymptotic rates, including superlinear convergence.

\begin{theorem}[Improved asymptotic rates, finite dimensions, $C^2$ case]\label{thm:global_convergence_c2} 
Let, for problem \eqref{eq:opt},  \cref{ass:basic,ass:subproblem} hold. 
Let $\dim(\mathcal{H}) < \infty$, $f$ be locally twice continuously differentiable and locally convex around some point $x^*$, and $H(x)\coloneqq  \nabla^2 f(x)$. Let $(x_k)_{k \in \mathbb{N}}$ be the iterates of \cref{alg:proximal_newton}. Then, the following %
hold.
\begin{enumerate}[label=(\roman*)]\itemsep=0em
    \item If $x_k \to x^*$ as $k \to \infty$, then %
    we have that $\lambda_k \to 0$ as $k \to \infty$.
    \item If the assumptions of \cref{thm:global_convergence_sublinear_nonconv} for the nonconvex setting hold and $x_k \to x^*$, then $\min_{0\leq i \leq k-1}\norm{F'(x_{i+1})}{*} = \so(1\slash \sqrt{k})$.
    \item If the assumptions of \cref{thm:global_convergence_nonconv_PL} for the PL setting hold, then $F(x_k)-F^*$ converges to $0$ superlinearly and $x_k$ converges to $x^*$ superlinearly, where $x^*$ is a global solution to \eqref{eq:opt}.
    \item If the assumptions of \cref{thm:global_convergence_sublinear} for the convex setting hold and $x_k \to x^*$, then $F(x_k)-F^*=\so(1\slash k)$.
    \item If the assumptions of \cref{thm:local_convergence_1} on local convergence hold, then $x_k$ converges to $x^*$ locally superlinearly and $F(x_k)$ converges to $F(x^*)$ locally superlinearly, where $x^*$ is a local solution to \eqref{eq:opt}.
    \end{enumerate}
\end{theorem}
\begin{proof}
\begin{enumerate}[label=(\roman*), wide, labelwidth=!, labelindent=0pt]\itemsep=0em
\item 
Again, given an index $k \geq 0$, we denote by $x_{k,+} \coloneqq x_+(\lambda_k/2,x_k)$ (see \eqref{eq:PLMSN_step}). Since we are in the finite dimensional setting, we write $f'\equiv \nabla f$.
We first show that
\begin{equation}
    \omega_k\coloneqq   \frac{ \norm{\nabla f(x_{k,+}) -  \nabla f(x_k) - \nabla^2 f(x_k)(x_{k,+}-x_k)}{*}}{\|x_{k,+} - x_k\|} \to 0 \quad \text{as $k \to \infty$.} \label{eq:higher_regularity_11}
\end{equation}
For this purpose, fix an arbitrary $\varepsilon>0$. By assumption, there exists a ball $B_R(x^*)$ on which $\nabla^2 f$ exists and is continuous on $B_R(x^*)$. Moreover, by local convexity and twice differentiability of $f$ we obtain $H(x)\coloneqq  \nabla^2 f(x)\succeq 0$ for all $x \in B_R(x^*)$. In addition, using compactness of $B_R(x^*)$ and the Heine--Cantor theorem, we get that $\grad^2 f$ is even uniformly continuous on $B_R(x^*)$. The latter means that there exists $\delta>0$ such that for every $x,y \in B_R(x^*)$ with $\| x -y \| \leq \delta$, we have 
\begin{equation*}
    \|\nabla^2 f(x) - \nabla^2 f(y)\|_{\mathrm{op}} \leq \varepsilon. %
\end{equation*}
As $x_k \to x^*$, there exist $K_1$ such that $k \geq K_1$ implies $x_k \in B_R(x^*)$. %
This in turn implies that, for $k \geq K_1$, $H(x_k) \succeq 0$. %
We conclude by \cref{lem:strange_convergence} that $ x_{k,+} \to x^*$ as well.
We further infer that there is an index $K_2 \geq K_1$ such that $x_k, x_{k,+} \in B_R(x^*)$ and 
$\norm{x_{k,+}-x_k}{} \leq \delta$ for all $k \geq K_2$. 
Consequently, we deduce from the fundamental theorem of calculus that, for all $k \geq K_2$,
\begin{align}
    &\norm{\nabla f(x_{k,+}) -  \nabla f(x_k) -  \nabla^2 f(x_k)(x_{k,+}-x_k)}{*}  \notag\\
    &= \norm{ \int_0^1 (\nabla^2f(x_k + t(x_{k,+} - x_k)) - \nabla^2 f(x_k))(x_{k,+} - x_k) \, \mathrm{d} t }{*} \notag \leq 
    \varepsilon \|x_{k,+} - x_{k}\|
\end{align}
as, clearly, for $x_t \coloneqq x_{k} + t(x_{k,+}-x_{k})$, we have $x_t \in B_R(x^*)$ and $\|x_t - x_{k}\| \leq \|x_{k,+} - x_{k}\| \leq \delta$ for any  $t\in (0,1)$. Since the choice of $\varepsilon>0$ was arbitrary, we infer \eqref{eq:higher_regularity_11}. 
Since $F$  is locally convex around $x^*$, 
recalling that $x_{k,+}=x_+(\lambda_k/2,x_k)$ and applying \cref{lem:helper_lambdak_zero} with $H(x_k) = \grad^2 f(x_k)$ and $\varepsilon_k = \omega_k$, we get $\lambda_{k} \to 0$. 
\item The claim follows by item (i) above and \cref{thm:global_convergence_sublinear_nonconv}.
\item The claim follows by item (i) above and \cref{thm:global_convergence_nonconv_PL}, which guarantees also that $x_k \to x^*$, where $x^*$ is a global solution to \eqref{eq:opt}.
\item The claim follows by item (i) above and \cref{thm:global_convergence_sublinear}.
\item The claim follows by item (i) above and \cref{thm:local_convergence_1}, which guarantees also that $x_k \to x^*$, where $x^*$ is a local solution to \eqref{eq:opt}.\qedhere
\end{enumerate}
\end{proof}

\subsection{Superlinear convergence under semismoothness}
In this subsection, we present one of our main results on superlinear convergence of \cref{alg:proximal_newton} in the semismooth setting, i.e., when \cref{ass:semismooth} holds. This setting is less restrictive than the $C^2$ setting in the previous subsection. Yet, it still allows us to obtain superlinear convergence in the same spirit as for classical SSN. Note that we again consider infinite-dimensional spaces $\HH$.

We start with the following result that establishes superlinear convergence if convergence itself is already established. Recall that \cref{thm:global_convergence_nonconv_PL} for the PL setting guarantees global convergence and \cref{thm:local_convergence_1} guarantees local convergence under PL and QG conditions. Thus, these results in combination with the following proposition establish respectively global superlinear and local superlinear convergence under semismoothness. As already mentioned, for simplicity we focus on the setting when $F$ is locally convex. 

\begin{prop}[Superlinear convergence under semismoothness]\label{thm:fast_local}  %
Let, for problem \eqref{eq:opt}, \cref{ass:basic,ass:subproblem} hold. Let $x^*$ be a local minimum in problem \eqref{eq:opt}, $F$ be locally convex around $x^*$, and the semismoothness at $x^*$ \cref{ass:semismooth} hold. Let $(x_k)_{k \in \mathbb{N}}$ be the iterates of \cref{alg:proximal_newton} and $x_k \to x^*$ as $k \to \infty$. Assume also that, for sufficiently large $k$, $H(x_k) \succeq 0$, 
the PL condition \eqref{eq:PL_at_point} holds at $x_{k}$ with $F^*_U=F(x^*)$, and there exists constants $c_1, c_2>0$  such that
\begin{equation}
    \norm{x_{k+1}-x^*}{} \leq c_1 \norm{x_k -x_{k+1}}{} \quad \text{and} \quad  \norm{x_k-x^*}{} \leq c_2 \norm{x_k -x_{k+1}}{} 
    \label{eq:useful_bounds_abstract}.
\end{equation}
    Consider the following statements.
    \begin{enumerate}[label=(\roman*)]\itemsep=0em
    \item\label{item:DM} The following Dennis--Moré-type (DM-type) condition holds: %
          \begin{equation}
        \norm{(H(x_+(\lambda_k\slash 2, x_k)) - H(x_k))(x_+(\lambda_k\slash 2, x_k) - x^*)}{*} = \so(\norm{ x_+(\lambda_k\slash 2, x_k) - x_{k}}{}) \quad \text{as $k \to \infty$.}\label{eq:DM_condition}
        \end{equation}
    \item\label{item:xk_xkplus1} The following condition holds: %
    \begin{equation}
       \frac{ \norm{f'(x_{+}(\lambda_k\slash 2,x_k)) - f'(x_{k}) - H(x_k)(x_{+}(\lambda_k\slash 2,x_k) - x_k) }{*} }{ \norm{x_{+}(\lambda_k\slash 2, x_k) - x_k}{}} \to 0 \quad \text{as $k \to \infty$}.\label{eq:uniform_ss_condition_convex}
    \end{equation}
    \item\label{item:lambdak_zero} The sequence of regularisation parameters goes to zero, i.e., $\lambda_k \to 0$ as $k \to \infty$.
    \item\label{item:local_superlinear} The convergence   $x_{k} \to x^*$ is superlinear, and in fact
     \begin{equation*}
         \|x_{k+1} - x^*\| \leq \frac{2c_2\lambda_{k}^{1/2}}{ \sqrt{\alpha^2 \beta \mu}}   \|x_k -x^*\| = \so(\|x_k -x^*\|)\quad \text{as $k \to \infty$}.
    \end{equation*}
    \item\label{item:gradients_superlinear} The subgradients converge superlinearly to zero, and in fact
   \begin{equation*}%
            \norm{ F'(x_{k+1})}{*} \leq \frac{\lambda_k^{1\slash 2}}{\sqrt{2\alpha\mu}} \norm{F'(x_k)}{*} = \so(\norm{F'(x_k)}{*}) \quad \text{as $k \to \infty$}.
        \end{equation*}   
    \end{enumerate}    
    Then, %
    we have the following relationship:
        \begin{align*}
        \ref{item:DM} &\iff \ref{item:xk_xkplus1} \implies \ref{item:lambdak_zero} \implies 
        \begin{cases}
            \ref{item:local_superlinear},\\
            \ref{item:gradients_superlinear}.
        \end{cases} %
    \end{align*} 

    \end{prop}

    \begin{proof}
    By the assumptions of the theorem, there exist a ball $B_R(x^*)$ such that 
    \begin{enumerate}[label=(\arabic*)] 
        \item $F(x)\geq F(x^*)$ for all $x \in B_R(x^*)$; 
        \item $F$ is convex on $B_R(x^*)$;
        \item By \cref{ass:semismooth} of semismoothness at $x^*$ it holds that
        \begin{equation}\label{eq:thm:fast_local_proof_1}
            \norm{ f'(x) - f'(x^{*}) - H(x)(x - x^*)}{*}= \so(\norm{ x -x^*}{}), \quad %
            \text{as } x \to x^*.
        \end{equation}
    \end{enumerate}
    For clarity, given an index $k \geq 0$, we again denote $x_{k,+} \coloneqq x_+(\lambda_k/2,x_k)$ (see \eqref{eq:PLMSN_step}). 
    By the assumptions, we also have the following.
    \begin{enumerate}[label=(\arabic*)] \setcounter{enumi}{3}
        \item There exists $K_0$ such that for all $k\geq K_0$, we have $H(x_k) \succeq 0$. %
        \item Combining the previous item with $x_k \to x^*$, we have by \cref{lem:strange_convergence} that $ x_{k,+} \to x^*$ as well. Thus, there exists $K_1$ s.t. for all $k \geq K_1$, we have $x_k,x_{k,+}  \in B_R(x^*)$ and, hence, we can use local optimality, convexity, and semismoothness.
        \item There exists $K_2$ such that for all $k\geq K_2$, the PL condition holds at $x_{k}$, i.e., \eqref{eq:PL_at_point} holds with $F^*_U=F(x^*)$. By the previous item, $x_k \in B_R(x^*)$ for all $k \geq K_1$. Also, we have $F(x)\geq F(x^*)$ for all $x \in B_R(x^*)$. Taking $U=B_R(x^*)$, $F_U^*=F(x^*)$ in \cref{lem:rk_identity_new} and combining with these observations, we can apply \cref{lem:rk_identity_new} for any $k\geq \max\{K_1,K_2\}$.
        \item There exists $K_3$ such that for all $k\geq K_3$ the inequalities \eqref{eq:useful_bounds_abstract} hold.
    \end{enumerate}

    From now on, we consider arbitrary $k \geq \max\{K_0,K_1,K_2,K_3\}$ so that we can apply all the above seven items.

    \begin{enumerate}[label=(\roman*), wide, labelwidth=!, labelindent=0pt]%
    \item[\ref{item:DM} $\implies$ \ref{item:xk_xkplus1}:]
    Thanks to the first inequality in \eqref{eq:useful_bounds_abstract} and $H(x_k) \succeq 0$, %
    we can use \cref{lem:good_bounds} with $\lambda=\lambda_k/2$ to obtain, recalling that $x_{k,+} =x_+(\lambda_k/2,x_k)$,
    \begin{align}
        \| x_{k,+} - x^*\|  \leq a\| x_{k,+}  - x_k\|,  \quad 
        \norm{x_k-x^*}{}  \leq b\| x_{k,+}  - x_k\|,\label{eq:good_bound_abstract_2}
    \end{align}
    where $a,b>0$ are constants.
    Using these inequalities, by the triangle inequality, semismoothness at $x^*$ \eqref{eq:thm:fast_local_proof_1}, and the Dennis--Moré-type condition \eqref{eq:DM_condition}, we have
    \begin{align*}
    \norm{ f'(x_{k,+}) - f'(x_{k}) - H(x_k)(x_{k,+} - x_k) }{*} 
        &\leq
        \norm{ f'(x_{k,+}) - f'(x^*) - H(x_{k,+})(x_{k,+} - x^*) }{*} \\
        &\quad +  \norm{ f'(x_k) - f'(x^{*}) - H(x_k)(x_{k} - x^*) }{*}\\
        &\quad +  \norm{ (H(x_{k,+}) - H(x_{k}))(x_{k,+} - x^*) }{*} \\
         \overset{\eqref{eq:thm:fast_local_proof_1},\eqref{eq:DM_condition}}&{=} \so(\| x_{k,+} -x^*\|) +\so(\| x_{k} -x^*\|) + \so(\| x_{k,+} -x_k\|) \\
        \overset{\eqref{eq:good_bound_abstract_2}}&{=} \so(\| x_{k,+} -x_k\|), 
    \end{align*}
    which is \eqref{eq:uniform_ss_condition_convex}.%

    \item[\ref{item:xk_xkplus1} $\implies$ \ref{item:DM}:] Analogous to the decomposition above, we have
     \begin{align*}
      \norm{ (H(x_{k,+}) - H(x_{k}))(x_{k,+} - x^*) }{*} 
        &\leq
        \norm{ f'(x_{k,+}) - f'(x^*) - H(x_{k,+})(x_{k,+} - x^*) }{*} \\
        &\quad +  \norm{ f'(x_k) - f'(x^{*}) - H(x_k)(x_{k} - x^*) }{*}\\
        &\quad + \norm{ f'(x_{k,+}) - f'(x_{k}) - H(x_k)(x_{k,+} - x_k) }{*} \\
         \overset{\eqref{eq:thm:fast_local_proof_1},\eqref{eq:uniform_ss_condition_convex}}&{=} \so(\| x_{k,+} -x^*\|) +\so(\| x_{k} -x^*\|) + \so(\| x_{k,+} -x_k\|) \\
        \overset{\eqref{eq:good_bound_abstract_2}}&{=} \so(\| x_{k,+} -x_k\|), 
    \end{align*}
    which is \eqref{eq:DM_condition}.
  
    \item[\ref{item:xk_xkplus1} $\implies$ \ref{item:lambdak_zero}:] 
    By the assumption of this item, there exists a sequence $(\varepsilon_k)_{k \in \mathbb{N}} \subset \R_+$ such that $\varepsilon_k \to 0$ and  
        \begin{align*}
            \norm{ f'(x_{+}(\lambda_k/2, x_k)) - f'(x_{k}) - H(x_k)(x_{+}(\lambda_k/2, x_k) - x_k) }{*} 
            &\leq \varepsilon_k \norm{x_k - x_+(\lambda_k/2,x_k)}{}.
        \end{align*}
        Thanks to local convexity of $F$ and $H(x_k) \succeq 0$, %
        we can use \cref{lem:helper_lambdak_zero}, which gives us that $\lambda_k \to 0$.

\item[\ref{item:lambdak_zero} $\implies$ \ref{item:local_superlinear}:]  
Due to the assumed PL condition at $x_k$ and the fact that $F(x^*) \leq F(x_{k})$, we can apply  \cref{lem:rk_identity_new} and obtain \eqref{eq:rk_identity_new}, where we recall from \eqref{eq:F_r_g_notaion} the notation $r_k\coloneqq \norm{x_k-x_{k+1}}{}$.
This together with the second inequality in \eqref{eq:useful_bounds_abstract} at iteration $k+1$ gives for some constants $c_1,c_2 >0$
\begin{equation*}
    \|x_{k+1} - x^*\| \overset{\eqref{eq:useful_bounds_abstract}}{\leq} c_2 \norm{x_{k+1}-x_{k+2}}{}, \qquad  \norm{x_{k+1}-x_{k+2}}{}\overset{\eqref{eq:rk_identity_new}}{\leq} \frac{1}{ \sqrt{\alpha^2 \beta \mu}} \lambda_{k}^{1/2} \norm{x_k-x_{k+1}}{}.
\end{equation*}
Combining these two inequalities and using the triangle inequality, we deduce
\begin{equation*}
    \|x_{k+1} - x^*\| \leq \frac{c_2}{ \sqrt{\alpha^2 \beta \mu}}  \lambda_{k}^{1/2} \norm{x_k-x_{k+1}}{} \leq \frac{c_2}{ \sqrt{\alpha^2 \beta \mu}}  \lambda_{k}^{1/2} (\|x_{k+1} - x^*\| + \|x_k -x^*\|).
\end{equation*}
Choosing $K_4 \geq \max\{K_0,K_1,K_2,K_3\}$ sufficiently large such that $k\geq K_4$ implies  $\frac{c_2}{ \sqrt{\alpha^2 \beta \mu}} \lambda_{k}^{1/2} \leq \frac12$, we have
\begin{equation*}
    \|x_{k+1} - x^*\| \leq \frac{2c_2}{ \sqrt{\alpha^2 \beta \mu}}  \lambda_{k}^{1/2}  \|x_k -x^*\|, \quad \forall k \geq K_4,
\end{equation*}
which implies the superlinear convergence of the iterates in \ref{item:local_superlinear} as $\lambda_k \to 0$ by assumption in this item. 
    
\item[\ref{item:lambdak_zero} $\implies$ \ref{item:gradients_superlinear}:] 
By the local convexity, \cref{lem:convex_g_condition} applies and we can use \eqref{eq:functionValuesGradientRelation_convex}. 
Due to the PL condition and the fact that $F(x^*) \leq F(x_{k})$, we can also apply  \cref{lem:rk_identity_new} and use \eqref{eq:F_k_g_k_bound}. Using additionally that $F(x^*) \leq F(x_{k+1})$, we finally obtain, recalling the notation $g_k\coloneqq \norm{F'(x_k)}{*}$ from \eqref{eq:F_r_g_notaion},
\begin{align*}
    g_{k+1}^2 \overset{\eqref{eq:functionValuesGradientRelation_convex}}&{\leq} \frac{\lambda_k}{\alpha}(F(x_k)-F(x_{k+1})) = \frac{\lambda_k}{\alpha}(F(x_k)-F(x^*) + F(x^*) - F(x_{k+1}))\overset{\eqref{eq:F_k_g_k_bound}}{\leq} \frac{\lambda_k}{2\alpha\mu}g_k^2.
\end{align*}
Since $\lambda_k \to 0$, we clearly have superlinear convergence of the subgradient norms. Note that due to the local convexity, the above bound is tighter than the one that can be obtained from \eqref{eq:gk_identity_new}. \qedhere
\end{enumerate}
\end{proof}
\cref{thm:fast_local} shows that the semismoothness \cref{ass:semismooth} and DM condition \eqref{eq:DM_condition} are sufficient for superlinear convergence of \cref{alg:proximal_newton}. We consider \eqref{eq:DM_condition} as a kind of Dennis--Moré condition, referring to the celebrated Dennis--Moré theorem \cite{MR343581}, which essentially gives a necessary and sufficient condition for superlinear convergence of quasi-Newton methods for smooth functions (in the nonsmooth setting, we refer to, e.g., \cite{cibulka2015inexact, MR3023752}). An interesting open question is to understand if \eqref{eq:DM_condition} is  also necessary to obtain superlinear convergence in our setting. 
We also note that a similar condition was used recently in \cite{potzl2022second,potzl2024inexact} to show superlinear convergence in the semismooth setting. %

The next result provides a sufficient condition for \eqref{eq:DM_condition} by showing that it is satisfied when $H$ is continuous.
\begin{lem}\label{lem:c2_implies_dm}
Let $H$ be continuous at $x^*$.
Let $(x_k)_{k \in \mathbb{N}}$ be the iterates of \cref{alg:proximal_newton} and $x_k \to x^*$ as $k \to \infty$. Let, for sufficiently large $k$, $H(x_k) \succeq 0$ and inequality \eqref{eq:good_bound_1} hold with $\lambda=\lambda_k/2$.
Then the Dennis--Moré-type condition \eqref{eq:DM_condition} is satisfied.
\end{lem}
\begin{proof}
Again, given an index $k \geq 0$, we denote by $x_{k,+} \coloneqq x_+(\lambda_k/2,x_k)$ (see \eqref{eq:PLMSN_step}). Then, by \cref{lem:strange_convergence}, we have $x_{k,+} \to x^*$ as $k\to \infty$.
This, together with triangle inequality and \eqref{eq:good_bound_1} gives us
\begin{align*}
\frac{\norm{(H(x_{k,+}) - H(x_k))(x_{k,+} - x^*)}{*}}{\| x_{k,+} - x_k\|} &= \frac{\norm{(H(x_{k,+}) -H(x^*) + H (x^*) - H(x_k))(x_{k,+} - x^*)}{*}}{\| x_{k,+}- x_k\|} \\
&\leq \frac{(\norm{H(x_{k,+}) -H(x^*)}{\mathrm{op}} + \norm{H(x^*) - H(x_k)}{\mathrm{op}})\norm{x_{k,+} - x^*}{}}{\| x_{k,+} - x_k\|} \\
\overset{\eqref{eq:good_bound_1}}&{\leq}\frac{1+3c}{2}(\norm{H(x_{k,+}) -H(x^*)}{\mathrm{op}} + \norm{H(x^*) - H(x_k)}{\mathrm{op}})\to 0,
\end{align*}
where convergence holds by continuity of $H$ at $x^*$ since both $x_k$ and $x_{k,+}$ converge to $x^*$ as $k\to \infty$.  
\end{proof}
Note that the inequality \eqref{eq:good_bound_1} is guaranteed, for example by \cref{lem:useful_bounds_new} or \cref{lem:useful_bounds_local} since they guarantee \eqref{eq:useful_bounds_abstract_first_1}.
\begin{remark}
It is possible to use the following condition instead of \eqref{eq:DM_condition}:
\begin{equation*}
\norm{(H(x_+(\lambda_k\slash 2,x_k)) - H(x_k))(x_+(\lambda_k\slash 2,x_k) - x^*}{*} = \so(\| x_{k} - x^*\|) \quad \text{as $x_k \to x^*$.}%
\end{equation*}
Yet, this requires additional technical steps using \eqref{eq:good_bound_abstract_2} to show the implication \ref{item:DM} $\implies$ \ref{item:xk_xkplus1}.
Also, \cref{lem:c2_implies_dm} would still go through with additional estimates, as done in the proof of \cite[Proposition 8]{potzl2022second}. %

As we also showed in \cref{thm:fast_local}, the condition \eqref{eq:uniform_ss_condition_convex} is sufficient for superlinear convergence as well. This condition can be interpreted as a kind of uniform semismoothness along the trajectory of \cref{alg:proximal_newton}.
\end{remark}

\begin{remark}
It is possible  to not assume local convexity of $F$ in \cref{thm:fast_local}. The main place where it is used is in the proof of \cref{lem:helper_lambdak_zero} to show the inequality \eqref{eq:final_inequality_superlinear_2}, which could also be shown by using the same arguments as in \cref{lem:acceptance_inner_nonconv} based on the model nonincrease condition \eqref{eq:prox_step_alg_decr}. At the same time, the assumption $H(x_k) \succeq 0$ is  %
is crucial to show inequalities \eqref{eq:good_bound_abstract_2} by \cref{lem:good_bounds} and is thus crucial to show superlinear convergence. %
\end{remark}
We finish this subsection by the main results that combine the convergence rate results of previous sections with the result of \cref{thm:fast_local}.
The first result is that in the PL setting local convexity, semismoothness, and DM condition imply global superlinear convergence of \cref{alg:proximal_newton}.
\begin{theorem}[Global superlinear convergence]
\label{thm:global_superlinear}
Let the assumptions of \cref{thm:global_convergence_nonconv_PL} hold. %
Let also $(x_k)_{k \in \mathbb{N}}$ be the iterates of \cref{alg:proximal_newton} with arbitrary starting point $x_0\in \dom F$. 
Let additionally $F$ be locally convex around $x^*$, and  the semismoothness at $x^*$ \cref{ass:semismooth} hold, where $x^*$ is the limit point of $(x_k)_{k \in \mathbb{N}}$. %
Let also either 
the DM condition \eqref{eq:DM_condition} or condition \eqref{eq:uniform_ss_condition_convex} hold. 
Let $H(x_k) \succeq 0$ for sufficiently large $k$.%

Then, $x^*$ is a global minimum and
\[x_k \to x^*\text{ superlinearly}, \qquad F(x_k)-F^* \to 0\text{ superlinearly}, \qquad\norm{F'(x_k)}{*} \to 0\text{ superlinearly as $k \to \infty$}.\]

\end{theorem}    
\begin{proof}
Since the assumptions of \cref{thm:global_convergence_nonconv_PL} hold, we obtain $x_k \to x^*$ as $k \to \infty$ where $x^*$ is a global minimum, and, hence, is also a local minimum. From \cref{thm:global_convergence_nonconv_PL} we also inherit \cref{ass:subproblem} and the PL inequality of \cref{ass:PL} at any $x_{k}$. Finally, inequalities \eqref{eq:useful_bounds_abstract} hold by \cref{lem:useful_bounds_new}. Thus, all the assumptions of \cref{thm:fast_local} hold and we have that $\lambda_k\to 0$, which by \cref{thm:global_convergence_nonconv_PL} and item \ref{item:gradients_superlinear} of  \cref{thm:fast_local} prove the claimed superlinear convergence.
\end{proof}   
For convenience, let us state our result in the $C^2$ convex case. 
\begin{cor}[Global superlinear convergence, $C^2$ case, convex setting, PL condition]\label{cor:c2_convex_case}
Let, for problem \eqref{eq:opt},  \cref{ass:basic,ass:main_inequality,ass:PL} hold, let $f$ be $C^2$ and convex, and set $H=f''$. Let also $(x_k)_{k \in \mathbb{N}}$ be the iterates of \cref{alg:proximal_newton} with arbitrary starting point $x_0\in \dom F$. %

Then, $x^*$ is a global minimum and
\[x_k \to x^*\text{ superlinearly}, \qquad F(x_k)-F^* \to 0\text{ superlinearly}, \qquad\norm{F'(x_k)}{*} \to 0\text{ superlinearly as $k \to \infty$}.\]
\end{cor}

The following result is a counterpart in our setting of classical results on local superlinear convergence of SSN, see, e.g., Theorem 2.11 in \cite{hintermuller2010semismooth}. In particular, if all the assumptions are fullfilled locally around a solution and the starting point of \cref{alg:proximal_newton} is sufficiently close to that solution, then the algorithm converges superlinearly to that solution.
\begin{theorem}[Local superlinear convergence]
\label{thm:local_superlinear}
Let the assumptions of \cref{thm:local_convergence_1} hold. Let also $(x_k)_{k \in \mathbb{N}}$ be the iterates of \cref{alg:proximal_newton} with a starting point $x_0\in \dom F$ which is sufficiently close to an isolated local minimum $x^*$ in problem \eqref{eq:opt}. Let additionally $F$ be locally convex around $x^*$, and the semismoothness at $x^*$ \cref{ass:semismooth} hold. Let also either 
DM condition \eqref{eq:DM_condition} or condition \eqref{eq:uniform_ss_condition_convex} hold. 

Then,
\[x_k \to x^*\text{ superlinearly}, \qquad F(x_k)-F^* \to 0\text{ superlinearly}, \qquad\norm{F'(x_k)}{*} \to 0\text{ superlinearly as $k \to \infty$}.\]

\end{theorem}    
\begin{proof}
Since the assumptions of \cref{thm:local_convergence_1} hold, we obtain $x_k \to x^*$ as $k \to \infty$ with $x^*$ being a local minimum. From \cref{thm:local_convergence_1} we also inherit \cref{ass:subproblem}, the PL inequality of \cref{ass:PL} at any $x_{k}$, and $H(x_k) \succeq 0$ for any $k$. %
Finally, inequalities \eqref{eq:useful_bounds_abstract} hold by \cref{lem:useful_bounds_local}. Thus, all the assumptions of \cref{thm:fast_local} hold and we have that $\lambda_k\to 0$, which by \cref{thm:local_convergence_1} and item \ref{item:gradients_superlinear} of \cref{thm:fast_local} prove the claimed superlinear convergence.
\end{proof}

\begin{remark}[On the assumption of $H(x)\succeq 0$]
\label{rem:Hgeq0}
\cref{thm:global_superlinear} and \cref{thm:local_superlinear} assume that $H(x)\succeq 0$ to show the transition to superlinear convergence. 
We would like to underline that this assumption, without loss of generality, covers the setting of a more flexible assumption that
\[
H(x)\succeq \mu_H I
\quad\text{and}\quad
\psi \ \text{is $\mu_\psi$-strongly-convex,}
\]
i.e., $\psi - \tfrac{\mu_\psi}{2}\norm{\cdot}{}^2$ is convex, where $\mu=\mu_H+\mu_\psi\geq 0$.
A slightly stronger assumption with $\mu>0$ was made, e.g., in~\cite{potzl2022second}. 
Indeed, if $\mu_H<0$, we can consider the transformation
\[
\widetilde{\psi} := \psi - \tfrac{\mu_\psi}{2}\|\cdot\|^2,
\qquad
\widetilde{H}(x) := H(x) + \mu_\psi I.
\]
This clearly implies that $\widetilde{\psi}$ is convex and, since $\mu_H+\mu_\psi\geq 0$,  also $\widetilde{H}(x)\succeq 0$. Moreover, as shown in  \cite[Lemma~1]{potzl2022second}, the algorithm with steps \eqref{eq:PLMSN_step} produces exactly the same iterates for the reformulated problem. Thus, to improve the readability of the paper, we used the assumption $H(x) \succeq 0$ for the results on superlinear convergence and assume throughout the paper that $\psi$  is convex, but bear in mind that these results are valid also when $H(x)$ has negative curvature and $\psi$ is sufficiently strongly convex. Moreover, our results also cover the opposite situation when $\psi$ is weakly convex, i.e., $\mu_\psi<0$, but $H(x)$ has sufficient positive curvature. 
These two scenarios were considered in \cite{potzl2022second} and both of them are also covered by our framework thanks to the transformation above. We underline also that the assumption $H(x) \succeq 0$ is not used to prove our global convergence rate results in \cref{sec:global rates}.
\end{remark}

\subsection{\texorpdfstring{Superlinear convergence via active manifold identification in the\newline finite-dimensional setting}{Superlinear convergence via active manifold identification in the finite-dimensional setting}}%

In \cref{thm:fast_local}, we established that a transition to fast %
superlinear convergence can occur under several  conditions. A consequence of \cref{lem:c2_implies_dm} is that the assumption $f \in C^2$ implies the Dennis--Moré condition \eqref{eq:DM_condition} of \cref{thm:fast_local} if $H\equiv f''$, and therefore also fast superlinear convergence of the iterates. In this subsection, we go further by showing that superlinear convergence can also occur when $f$ is not locally $C^2$, but instead satisfies a combination of \emph{partial smoothness} and \emph{semismoothness} of the gradient. Specifically, in the finite-dimensional setting, we demonstrate that our algorithm identifies an active manifold, along which the function  behaves smoothly, and that this identification enables rapid local convergence even when $\nabla f$ %
is not differentiable at the limit point $x^*$.
For this purpose, we utilise the above-mentioned notion of \emph{partial smoothness}, which originates from the influential work \cite{lewis2002active}, see also \cite{LewisZhang,borwein2006convex,drusvyatskiy2014optimality}. The key idea is to show that, under a partial smoothness assumption, the iterates eventually lie on a smooth manifold along which the objective function $F$ is $C^2$. This phenomenon has a rich history and forms a well-established part of variational analysis. Partial smoothness has been employed in the analysis of first-order methods~\cite{liang2014local,liang2017activity}, in applications areas such as inverse problems \cite{vaiter2017model} and variational inequalities \cite{MR4288533}, and more recently, it has found use in the context of Newton-type algorithms \cite{lewis2013nonsmooth,bareilles2023newton,MR4887740}. While we do not aim to give a comprehensive treatment of partial smoothness here, we refer the reader to the classical references~\cite{borwein2006convex,lewis2022partial,miller2005newton} and focus only on the aspects essential for our convergence analysis.

In this section we work in the finite-dimensional setting. Let $\mathcal{H} = \mathbb{R}^n$ for some $n \in \mathbb{N}$ so that $F \colon \mathbb{R}^n \to \overline{\mathbb{R}}$. We begin by recalling the definition of a $C^2$ manifold in the context of embedded $C^2$ submanifolds relevant to our analysis, %
see \cite[Definition 2.1.1]{berger2012differential}.
\begin{defn}[$C^2$-submanifold]\label{def:manifold}
Let $d\leq n$. A subset $\mathcal{M} \subset \mathbb{R}^n$ is called a $d$-dimensional \textit{$C^2$-submanifold} of $\mathbb{R}^n$ if for every $ x \in \mathcal{M}$ there exists an open neighbourhood $U_x \subset \mathbb{R}^n$ containing $x$ and a $C^2$ map $\psi_x \colon U_x  \to \mathbb{R}^n$ such that $\psi_x(U_x) \subset \mathbb{R}^n$ is open, $\psi$ is a $C^2$--diffeomorphism onto its image and 
\begin{equation*}
    \psi_x(U_x \cap \mathcal{M}) = \psi_x(U_x) \cap E_d
\end{equation*}
where $E_d$ is the $d$-dimensional linear subspace defined as $E_d := \{ x \in \mathbb{R}^n$: $x_{d+1}= \ldots = x_n = 0 \}$.
\end{defn}
We recall the following characterisation (see \cite[Theorem 2.1.2]{berger2012differential} and the proof) of a $C^2$--submanifold which will be important in the forthcoming. The set $ \mathcal{M} \subset \mathbb{R}^n$ is a $d$-dimensional $C^2$-submanifold of $\mathbb{R}^n$ if and only if for every $x \in \mathcal{M}$ there is an open subset $\Omega_x \subset \mathbb{R}^d$ containing $0 \in \mathbb{R}^d$ and an open neighbourhood $U \subset \mathbb{R}^n$ of $x$ in the topology of $\mathbb{R}^n$ and a $C^2$ map
\begin{equation}
\phi_x\colon \Omega_x  \to \mathbb{R}^n \label{eq:immerson_at_x}
\end{equation}
that satisfies the following properties:
\begin{enumerate}[label=(\roman*)]\itemsep=0em
\item  $\phi_x\colon\Omega_x \to \mathbb{R}^n$ is $C^2$, $\phi'(0)$ is injective and $\phi_x(0) = x$ (thus, $\phi_x$ is a $C^2$ immersion around $0$), 
    \item $\phi_x\colon \Omega_x \to \phi_x(\Omega_x) = U \cap \mathcal{M}$ is a homeomorphism with respect to the canonical subspace topology on $\mathcal{M}$,
    \item$\phi_x$ can be chosen (after translation to ensure that $\psi_x(x) = 0$) in the explicit form  
    \begin{equation}
        \phi_x = \psi_x^{-1} \circ i_d \colon \underbrace{ i_d^{-1}(\psi_x(U_x \cap \mathcal{M}))}_{\coloneqq  \Omega_x} \to U_x \cap \mathcal{M},\label{eq:concrete_immersion}
    \end{equation}
    where $\psi_x\colon U_x \to \psi_x(U_x) \subset \mathbb{R}^n$ is the $C^2$--diffeomorphism in \cref{def:manifold} and $i_d\colon\mathbb{R}^d \to \mathbb{R}^n$ denotes the canonical embedding defined by $i_d(x) = (x,0) \in E_d \subset \mathbb{R}^n$. Note that \eqref{eq:concrete_immersion} makes sense as by \cref{def:manifold} we have $\psi_x(U_x \cap \mathcal{M}) \subset E_d$ and therefore the inverse $i_d^{-1}\colon E_d \to \mathbb{R}^d$ is well defined in \eqref{eq:concrete_immersion}.

\end{enumerate}
 The collection $(\phi_x(\Omega_x),\phi_x^{-1})_{x \in \mathcal{M}}$ forms an atlas of $\mathcal{M}$ in the usual sense of differential geometry, cf. the proof of \cite[Theorem 2.2.10.1]{berger2012differential}.  
We make the following assumption for the subsequent two results.
\begin{ass}\label{ass:first_manifolds}
Assume that
    \begin{enumerate}[label=(\roman*)]\itemsep=0em
         \item\label{item:M_is_C2_manifold}         The set \( \mathcal{M} \subset \mathbb{R}^n \) is a \( C^2 \)-submanifold,
        \item\label{item:gradf_C1_on_M} 
  \( \nabla f|_\mathcal{M} \) is of class \( C^1 \), %
        \item\label{item:ball_R} There exists an $R>0$ such that $\grad f$ is Newton differentiable on $B_R(x^*)$ with Newton derivative $H$ and $H$ is uniformly bounded on $B_R(x^*)$.

    \end{enumerate}
\end{ass}

\cref{ass:first_manifolds} \ref{item:ball_R} essentially asks for \cref{ass:semismooth} to hold not only at $x^*$ but on $B_R(x^*)$, and for the Newton derivative to be bounded there in the sense of \eqref{eq:bounded_Hessian}. Now, due to the first item of this assumption, for every \( x \in \mathcal{M} \), if  $\phi_x \colon \Omega_x \to \mathbb{R}^n$ is a $C^2$ map of the form~\eqref{eq:immerson_at_x} mapping \( \Omega_x \) homeomorphically onto \( \phi_x(\Omega_x) \subset \mathcal{M} \), then the composition
\[
\nabla f \circ \phi_x \colon \Omega_x \to \mathbb{R}^n
\]
is of class \( C^1 \).

The following lemma establishes a form of \emph{uniform semismoothness}, which will be instrumental in deriving the superlinear convergence of the algorithm via \cref{thm:fast_local}.

\begin{lem}\label{lem:manifolds_1}
Assume \cref{ass:first_manifolds}. Consider the local \( C^2 \)-immersion 
\[
\phi_{x^*} \colon \Omega_{x^*} \subset \mathbb{R}^d \to \mathcal{M} \subset \mathbb{R}^n
\]
as defined in \eqref{eq:concrete_immersion}. 
Then, there exists a radius \( r > 0 \) such that ${B_r(0)} \subset \Omega_{x^*}$ %
and for every \( \varepsilon > 0 \), there exists \( \delta > 0 \) satisfying %
\begin{equation*}
  \|\nabla f(\phi_{x^*}(w)) - \nabla f(\phi_{x^*}(z)) - H(\phi_{x^*}(z))(\phi_{x^*}(w) - \phi_{x^*}(z))\|_* 
  \leq \varepsilon \|\phi_{x^*}(w)-\phi_{x^*}(z)\|
\end{equation*}
for all \( w, z \in {B_r(0)} \) with \( \|w - z\| \leq \delta \).
 \end{lem}
\begin{proof}%
Without loss of generality, with $U_{x^*}$ as in \cref{def:manifold}, we may assume that
\begin{equation*}
{B_R(x^*)} \subset U_{x^*},
\end{equation*}
otherwise \( R \) can be reduced if necessary (this is possible 
since \( U_{x^*} \) is open and contains \( x^* \)). To show the statement of the lemma, we  fix an arbitrary $\varepsilon>0$. %
Since $\Omega_{x^*}$ is an open neighbourhood of $0 \in \mathbb{R}^d$ %
and $\phi_{x^*}$ is continuous, there exists $r  > 0$ such that
\[
{B_{r}(0)} \subset \Omega_{x^*} \qquad \text{and} \qquad 
\phi_{x^*}({B_{r}(0)}) \subset {B_R(x^*)} \subset U_{x^*}.
\]
By  \cref{ass:first_manifolds} \ref{item:gradf_C1_on_M}, the gradient $\grad f$ is of class $C^1$ on the manifold $\mathcal{M}$, 
hence we have that $\nabla f \circ \phi_{x^*} \colon \Omega_{x^*}\subset \mathbb{R}^{d} \to \mathbb{R}^n$ is of class $C^1$ with Fréchet derivative $(\nabla f \circ \phi_{x^*} )' \colon \Omega_{x^*} \to \mathcal{L}(\mathbb{R}^d,\mathbb{R}^n)$ and also by the characterisation above, the mapping $\phi_{x^*}\colon \Omega_{x^*} \to \mathbb{R}^n$ is $C^2$ with  Fréchet derivative $\phi_{x^*}'\colon \Omega_{x^*} \to \mathcal{L}(\mathbb{R}^d,\mathbb{R}^n)$. In particular the derivatives $\phi_{x^*}'$ and 
$(\nabla f \circ \phi_{x^*})'$ are continuous over ${B_{r}(0)}$ and therefore by the Heine--Cantor theorem, even uniformly continuous. Hence, as in \cref{thm:global_convergence_c2}, we find a $\delta>0$ such that $w,z \in {B_{r}(0)}$ and $\|w-z\| \leq \delta$ implies %
\begin{align*}
    \|(\nabla f \circ \phi_{x^*})'(z) - (\nabla f \circ \phi_{x^*})'(w) \|_* \leq \varepsilon \qquad \text{and}\qquad
    \|\phi_{x^*}'(z) -   \phi_{x^*}'(w) \| \leq \varepsilon.
\end{align*}
 By the fundamental theorem of calculus we conclude that
\begin{align}
    \|\nabla f(\phi_{x^*}(w)) - \nabla f(\phi_{x^*}(z)) - (\nabla f\circ\phi_{x^*})'(z)(w-z)\|_* &\leq \varepsilon \| w-z\| \label{eq:uniform_manifold_smooth_1},\\
       \| \phi_{x^*}(w) - \phi_{x^*}(z) - \phi_{x^*}'(z)(w - z) \| &\leq \varepsilon \|w - z\|.\label{eq:uniform_manifold_smooth_2}
\end{align} 
Using the Newton differentiability of $\grad f$, 
\[(\nabla f\circ \phi_{x^*})'(v) = H(\phi_{x^*}(v)) \phi_{x^*}'(v)
\quad \text{for all $v \in \interior(B_r(0))$},
\]   
by the chain rule\footnote{To be more precise, the left-hand side above is a Frechét derivative, which is equal to the Newton derivative as the former exists, then we apply the chain rule for Newton differentiable functions to get the right-hand side.} for Newton differentiable functions \cite[Proposition 3.8]{ulbrich2011semismooth}, which  is applicable, since for given $v \in \interior({B_r(0)})$ the set $\phi_{x^*}(\interior(B_r(0)))$ is an open environment of $\phi_{x^*}(v)$ contained inside ${B_R(x^*)}$ where the Newton derivative $H(\cdot)$ is uniformly bounded by the constant $M>0$, and also $\phi_{x^*}'(\cdot)$ is uniformly bounded on $\interior(B_{r}(0))$ (by its continuity on the compact set ${B_{r}(0)}$) which is an open neighbourhood of $v \in \interior(B_r(0))$ by choice. Thus we have
\[    \|\nabla f(\phi_{x^*}(w)) - \nabla f(\phi_{x^*}(z)) - H(\phi_{x^*}(z)) \phi_{x^*}'(z)(w-z)\|_* \leq \varepsilon \| w-z\|
\]
for $w,z \in\interior(B_r(0))$ with $\|w-z\| \leq \delta$. Consequently we deduce by \eqref{eq:uniform_manifold_smooth_1}, \eqref{eq:uniform_manifold_smooth_2}, and the triangle inequality,
\begin{align}
  &\|\nabla f(\phi_{x^*}(w)) - \nabla f(\phi_{x^*}(z)) - H(\phi_{x^*}(z))(\phi_{x^*}(w) - \phi_{x^*}(z))\|_* \notag\\
  &\leq \|\nabla f(\phi_{x^*}(w)) - \nabla f(\phi_{x^*}(z)) - H(\phi_{x^*}(z))\phi_{x^*}'(z)(w - z)\|_* \notag\\
  &\quad  + \norm{H(\phi_{x^*}(z))\phi_{x^*}'(z)(w - z) - H(\phi_{x^*}(z))(\phi_{x^*}(w)-\phi_{x^*}(z))}{*} \notag\\
  &\leq \varepsilon\| w-z\| +  \|H(\phi_{x^*}(z))\|_\mathrm{op}\| \phi_{x^*}(w) - \phi_{x^*}(z) - \phi_{x^*}'(z)(w - z) \| \notag
  \\
  &\leq \left(1 + M \right)\varepsilon \| w - z\|. \label{eq:lem_manifold_identification_2} 
\end{align}
Recall that we used the explicit form \eqref{eq:concrete_immersion}, that ensures $\phi_{x^*} = \psi_{x^*}^{-1} \circ i_d\colon \Omega_{x^*} \to \mathbb{R}^n$ with the $C^2$--diffeomorphism $\psi_{x^*}$ from \cref{def:manifold}. In particular, \( \psi_{x^*} \) has a bounded first derivative on the compact set \( \overline{\phi_{x^*}(\interior(B_{r}(0)))} \subset U_x \subset  \mathbb{R}^n \), and is therefore Lipschitz continuous. We conclude that the map $\phi_{x^*}^{-1} = i_d^{-1} \circ \psi_{x^*}\colon \phi_{x^*}(\interior(B_r(0))) \subset \mathcal{M} \to \mathbb{R}^d$ is Lipschitz continuous, as the map $i_d^{-1} \colon E_d \subset \mathbb{R}^n \to \mathbb{R}^d$ is 1-Lipschitz (note again that $i^{-1}_d \circ \psi_{x^*}$ is well defined in this setting as $\psi_{x^*}(x) \subset E_d$ for $x \in \phi_{x^*}(\interior(B_r(0))) \subset  \mathcal{M}$). In particular we find $L>0$ such that 
$$ \|w-z\| =  \|\phi_{x^*} ^{-1}(\phi_{x^*}(w))- \phi_{x^*}^{-1}(\phi_{x^*}(z))\| \leq L\| \phi_{x^*}(w) - \phi_{x^*}(z)\|$$
for all $w,z \in \interior(B_r(0))$. Eventually we deduce from the previous observations and \eqref{eq:lem_manifold_identification_2}, that
\begin{align*}
  \|\nabla f(\phi_{x^*}(w)) - \nabla f(\phi_{x^*}(z)) - H(\phi_{x^*}(z))(\phi_{x^*}(w) - \phi_{x^*}(z))\|_*
  \leq \left(1 + M \right) L \varepsilon \| \phi_{x^*}(w) - \phi_{x^*}(z)\|. 
\end{align*}
By rescaling $\varepsilon>0$ and shrinking $r$ (to obtain the result on the closed ball), the assertion follows.
\end{proof}

\begin{prop}\label{prop:manifolds_lambda_zero}
Let $(x_k)_{k \in \mathbb{N}}$ be the iterates of \cref{alg:proximal_newton} and $x_k \to x^*$ as $k \to \infty$. Let also $H(x_k) \succeq 0 $ for sufficiently large $k$ and suppose that $x_k \in \mathcal{M}$ and $x_+(\lambda_k\slash 2, x_k) \in \mathcal{M}$ for sufficiently large $k$. Assume also \cref{ass:first_manifolds}. Then, we have that %
\begin{equation}
           \frac{ \norm{\nabla f(x_{+}(\lambda_k/2, x_k)) -  \nabla f(x_k) - H(x_k)(x_{+}(\lambda_k/2, x_k)-x_k)}{*}}{\|x_{+}(\lambda_k/2, x_k) - x_k\|} \to 0 \quad \text{as $k \to \infty$} \label{eq:manifold_lem12}
\end{equation}
and hence, $\lambda_k \to 0$ as $k \to \infty$.
\end{prop}  
\begin{proof}
 We follow the argument in \cref{thm:global_convergence_c2}, denoting again \( x_{k,+} \coloneqq  x_+\left( \lambda_k / 2, x_k \right) \). We first show that \eqref{eq:manifold_lem12} holds.  
For this purpose fix an  arbitrary $\varepsilon>0$ and note that, due to \cref{lem:strange_convergence} the convergence of $x_k \to x^*$ also implies $x_{k,+} \to x^*$. Now consider the limit point $x^*$ and the associated immersion $\phi_{x^*} \colon \Omega_{x^*} \to \mathbb{R}^n$ from \eqref{eq:concrete_immersion}. Employing \cref{lem:manifolds_1}, we find a $\delta > 0$ and radius $r>0$ such that $w,z \in B_r(0)$ and $\|z-w\|\leq \delta$ implies
\begin{equation}
  \|\nabla f(\phi_{{x}^*}(w)) - \nabla f(\phi_{{x}^*}(z)) - H(\phi_{{x}^*}(z))(\phi_{{x}^*}(w) - \phi_{{x}^*}(z))\|_* 
  \leq \varepsilon \|\phi_{{x}^*}(w)-\phi_{{x}^*}(z)\|.\label{eq:manifold_lem13}
\end{equation}
Recall that $\phi_{x^*}\colon\Omega_{x^*} \to \phi_{x^*}(\Omega_{x^*})$ is a homeomorphism, and consequently $\phi_{x^*}(B_r(0)) \subset \mathcal{M}$ is %
a neighbourhood (in the topology of $\mathcal{M}$) of $\phi_{x^*}(0) = x^*$. From this and the facts that $x_{k,+} \to x^*$ and $x_k, x_{k,+} \in \mathcal{M}$ for $k \geq k_0$ for some $k_0 \in \mathbb{N}$,  we infer that there is a possibly larger index $k_1\geq k_0$ such that $x_k, x_{k,+} \in \phi_{x^*}(B_r(0))$ for all $k \geq k_1$. For those $k$, by injectivity of $\phi_{x^*}$, we may set $w_{k}\coloneqq  \phi_{x^*}^{-1}(x_k)$ and $w_{k,+}\coloneqq \phi_{x^*}^{-1}(x_{k,+})$.
Then by construction, we have $w_k, w_{k,+} \in B_r(0)$ and by continuity of $\phi_{x^*}^{-1}$ we deduce $w_k \to 0$ and $w_{k,+} \to 0$. Thus, by possibly enlarging $k_1$, we also have $\|w_k - w_{k,+}\| \leq \delta$ for all $k \geq k_1$. These observations combined with \eqref{eq:manifold_lem13}   eventually lead to
\begin{align*}
  &\|\nabla f(x_{k,+}) - \nabla f(x_{k}) - H(x_{k})(x_{k,+} - x_{k})\|_* \\
  &\quad =
  \|\nabla f(\phi_{x^*}(w_{k,+})) - \nabla f(\phi_{x^*}(w_{k})) - H(\phi_{x^*}(w_{k}))(\phi_{x^*}(w_{k,+}) - \phi_{x^*}(w_{k}))\|_* \\
  &\quad \leq \varepsilon \|\phi_{x^*}(w_{k,+}) - \phi_{x^*}(w_{k})\|  = \varepsilon\|x_{k,+} - x_{k} \|
  .\label{eq:manifold_lem14}
\end{align*}
for all $k \geq k_1$. As $\varepsilon>0$ was arbitrary, we infer \eqref{eq:manifold_lem12}. Now we can apply \cref{lem:helper_lambdak_zero} with $\varepsilon_k$ set as the left-hand side of \eqref{eq:manifold_lem12} to get $\lambda_k \to 0$ as $k\to \infty$.%
\end{proof} 
To apply this result and ensure superlinear %
convergence, it is essential that all iterates as well as the limit point eventually lie on a \( C^2 \) manifold --- commonly referred to as the \emph{active manifold}. This is precisely where the concept of \emph{partial smoothness} becomes central. We begin with several foundational definitions, following \cite{LewisZhang}. Let \( C \) be a non-empty convex set. The \emph{subspace parallel to \( C \)} is defined as
\[
\parm\, C \coloneqq  \aff\, C - x \quad \text{for any } x \in C,
\]
where \( \aff\, C \) denotes the affine hull (or affine span) of \( C \). The expression \( \parm\, C \) yields a linear subspace that is parallel to the affine set \( \aff\, C \), and is independent of the particular choice of \( x \in C \).

\begin{defn}[Partial smoothness]\label{def:partial_smoothness}
 We say that a function $g\colon \mathbb{R}^n \to \overline{\mathbb{R}}$ is \textit{$C^2$-partly smooth at $\bar x$ relative to $\mathcal{S}$} if $\mathcal{S}$ is a $C^2$-smooth manifold around $\bar x$ and
\begin{enumerate}[label=(\roman*)]\itemsep=0em
    \item\label{item:PS:RS} (restricted smoothness) $g|_\mathcal{S}$ is $C^2$ around $\bar x$,
    \item\label{item:PS_reg} (regularity) in a neighbourhood of $\bar x$ in $\mathcal{S}$, $g$ is subdifferentially regular and has a subgradient,
    \item\label{item:PS_NS} (normal sharpness) $N_{\mathcal{S}}(\bar x) = \parm \partial g(\bar x)$ where $N_{\mathcal{S}}(\bar x)$ means the normal space to $\mathcal{S}$ at $\bar x$,
    \item\label{item:PS_SC} (subgradient continuity) the  map $\partial g$ is continuous at $\bar x$ relative to $\mathcal{S}$.
\end{enumerate}    
\end{defn}
The normal sharpness condition essentially encodes that the function $g$, which is $C^2$-smooth around $\bar x$ on the manifold, must change drastically along directions that are normal to $\mathcal{S}$. Indeed, another characterisation \cite[Definition 1.3]{miller2005newton} of the normal sharpness condition is that the function $t \mapsto g(\bar x + td)$ is \emph{not} differentiable at $t=0$ for any $d \in N_{\mathcal{S}}(\bar x)$.  An example of a partially smooth function is $g(x) = \norm{x}{\ell^1}$, which is partially smooth at $\bar x$ relative to the manifold $\mathcal{S} = \{ x \in \mathbb{R}^n : I_x \subseteq I_{\bar x}\}$ where the notation $I_x \coloneqq  \{i \in \mathbb{N} : x_i \neq 0 \}$. 
For more examples we refer to e.g. \cite[\S 5.1]{liang2017activity} and \cite{miller2005newton}.

We first show that our algorithm \algname{LeAP-SSN} identifies the active manifold in finite time. For this purpose and for later use, we introduce the next assumption.
\begin{ass}\label{ass:manifolds_main_ass}
    Assume that
    \begin{enumerate}[label=(\roman*)]\itemsep=0em
        \item\label{item:F_PS} $F$ is $C^2$-partially smooth at a stationary point $x^*$ relative to the $C^2$-smooth manifold $\mathcal{M} \subset \mathbb{R}^n$,
        \item\label{item:strict_comp} The following \emph{non-degeneracy} condition holds:
\[0 \in \relint \partial F(x^*),\]
 where $\relint$ means the relative interior.
    \end{enumerate}
\end{ass}
The non-degeneracy condition in \cref{ass:manifolds_main_ass} \ref{item:strict_comp} is also referred to as \emph{strict complementarity}. These assumptions are standard\footnote{ One could explore the possibility to weaken the strict complementarity condition, see \cite{MR3865691}, however a full examination and treatment of these topics is far beyond the scope of this work.} in the literature, see e.g. \cite{Godeme}, \cite[Assumption 1.5]{miller2005newton}. In addition, \cref{prop:finite_time_id} suggests designing an algorithm that switches once finite identification has taken place to utilise the smooth structure, as considered for example in \cite{MR4620237}. 
\begin{prop}[Finite time identification]\label{prop:finite_time_id}
Let $(x_k)_{k \in \mathbb{N}}$ be the iterates of \cref{alg:proximal_newton}, $x_k \to x^*$ and $F(x_k) \to F(x^*)$ as $k \to \infty$, where $x^*$ is a global minimum. Assume also \cref{ass:manifolds_main_ass}.  
Then, we have that $x_k \in \mathcal{M}$ for sufficiently large $k$.

If, in addition, $H(x_k) \succeq 0$ for sufficiently large $k$, the PL condition \cref{ass:PL} holds,  %
$\grad f$ is Lipschitz  on $\interior(B_R(x^*))$ with constant $L_1$, and $H$ is uniformly bounded (in the sense of \eqref{eq:bounded_Hessian}) on $B_R(x^*)$, we also have $x_+(\lambda_k\slash 2, x_k) \in \mathcal{M}$ for sufficiently large $k$.
\end{prop}
\begin{proof}
First observe that $F$ is prox-regular since it is the sum of a continuously Fréchet differentiable function and a convex function, see \cite[Exercise 13.35]{RockafellarWets}. 
From \cref{thm:global_convergence_sublinear_nonconv}, we know that $F'(x_k) \to 0$ as $k\to \infty$. This implies that
$
    \dist(0, \partial F(x_k)) %
    \leq \norm{F'(x_k)}{*}  \to 0,
$ 
and applying \cite[Theorem 4.10]{LewisZhang} we get that $x_k \in \mathcal{M}$ for large enough $k$.

To prove the second claim, let us for simplicity denote $x_{k,+} \coloneqq  x_+(\lambda_k\slash 2, x_k)$. By \cref{lem:strange_convergence} (which is applicable since  $H(x_k) \succeq 0$ for sufficiently large $k$), we have that $x_{k,+} \to x^*$. This means that there exist $K_1$ such that for all $k\geq K_1$, we have $x_k,x_{k,+}  \in \interior(B_R(x^*))$. Using the Lipschitzness of $\grad f$ on $\interior(B_R(x^*))$, the boundedness of $H$ on $B_R(x^*)$,   and boundedness of $\lambda_k$ by \cref{lem:acceptance_inner_nonconv}, we obtain
    \begin{align*}
    \dist(0, \partial F(x_{k,+})) %
    \leq \norm{F'(x_{k,+})}{*}  \overset{\eqref{eq:F_div_def},\eqref{eq:lambda_bound}}&{\leq} (L_1+M)\norm{x_{k,+}-x_k}{}  + \frac{\bar\lambda}{2}\norm{x_{k,+}-x_k}{}  \to 0 \text{ as $k\to \infty$}.
\end{align*}
Using this and the PL condition \cref{ass:PL}, we also obtain $F(x_{k,+}) \to F(x^*)$. Applying again \cite[Theorem 4.10]{LewisZhang} we get that $x_+(\lambda_k\slash 2, x_k) \in \mathcal{M}$ for large enough $k$. %
\end{proof}

We are now in a position to state and prove the main theorem of this section, which is the sister result of \cref{thm:global_superlinear}.
\begin{theorem}[Global superlinear convergence under partial smoothness]\label{thm:fast_convergence_partial_smoothness}
Let, for problem \eqref{eq:opt},  \cref{ass:basic,ass:subproblem,ass:main_inequality,ass:PL} hold. Let also $(x_k)_{k \in \mathbb{N}}$ be the iterates of \cref{alg:proximal_newton} with arbitrary starting point $x_0\in \dom F$. Let additionally %
\cref{ass:first_manifolds,ass:manifolds_main_ass} hold. Assume also that $F$ is locally convex around  $x^*$ and that $H(x_k) \succeq 0$ for sufficiently large $k$. 

Then, $x^*$ is a global minimum and
\[x_k \to x^*\text{ superlinearly}, \qquad F(x_k)-F^* \to 0\text{ superlinearly}, \qquad\norm{F'(x_k)}{*} \to 0\text{ superlinearly as $k\to \infty$}.\]%
\end{theorem}
\begin{proof}
Note that we are in the regime of \cref{thm:global_convergence_nonconv_PL}. Also, thanks to \cref{ass:first_manifolds} \ref{item:ball_R}, by \cite[Proposition 2.3]{MR4431970} we have that $\grad f$ is Lipschitz on $\interior(B_R(x^*))$, so the full result of \cref{prop:finite_time_id} is available. The theorem now follows from collecting the implications of \cref{thm:global_convergence_nonconv_PL} (which gives $x_k \to x^*$, $F(x_k) \to F(x^*)$, and that $x^*$ is a global minimum), and in turn \cref{prop:finite_time_id} (which gives $x_k, x_+(\lambda_k\slash 2, x_k) \in \mathcal{M}$ for sufficiently large $k$) and \cref{prop:manifolds_lambda_zero} (which gives $\lambda_k \to 0$) and applying \cref{thm:fast_local}.
\end{proof}
\begin{remark}\label{rem:DM_holds_manifolds}
    In particular, under the conditions of the above result, we see that the DM condition \eqref{eq:DM_condition} is satisfied. Indeed, we proved in \cref{prop:manifolds_lambda_zero} the condition \eqref{eq:uniform_ss_condition_convex}, which is equivalent to the DM condition \eqref{eq:DM_condition} as shown in \cref{thm:fast_local}.
\end{remark}
For the result of \cref{thm:fast_convergence_partial_smoothness}, the fact that we need $\grad f|_{\mathcal{M}}$ to be $C^1$ from \cref{ass:first_manifolds} \ref{item:gradf_C1_on_M} is inconvenient since in combination with the normal sharpness condition in \cref{ass:manifolds_main_ass} \ref{item:F_PS}, it means that  we require $\psi$ to be non-zero and to provide the active manifold $\mathcal{M}$ on which we need $\grad f$ to be $C^1$. This is because we always assume $f \in C^1$ and hence $f$ cannot satisfy the normal sharpness condition (unless $N_{\mathcal{M}}(x^*) = \emptyset$) on a submanifold as $f$ is differentiable everywhere on its domain. However,   our algorithm identifies the manifold irrespective of whether $\grad f|_{\mathcal{M}}$ is $C^1$ or not. Nevertheless, we hope future work can excise this assumption.

\subsubsection{Example}\label{subsec:manifold_example} In order to build intuition and gain a geometric understanding, we discuss the results of this subsection through a small academic example, postponing more realistic applications to future work. For this purpose we define $F\colon \mathbb{R}^2 \to \mathbb{R}$ via $F(x) = f(x) + \psi (x)$ where for $x = (x_1,x_2)$
\begin{equation*}
    f(x_1,x_2) =  \|x\|^2 + \max(0,x_1)^2, \quad \psi(x_1,x_2) = |x_1|.
\end{equation*}
We directly observe that $f\colon\mathbb{R}^2 \to \mathbb{R}$ is globally $C^{1,1}(\mathbb{R}^2,\mathbb{R})$, with gradient
\begin{equation*}
    \nabla f(x_1,x_2) = 
    \begin{bmatrix} 
    2x_1 + 2\max(0,x_1) \\
    2x_2  
    \end{bmatrix}.
\end{equation*}
Clearly, the function $\nabla f\colon \mathbb{R}^2 \to \mathbb{R}^2$ is Newton differentiable with Newton derivative $H(x) \in \mathbb{R}^{2 \times 2}$ defined by %
\begin{equation*}
    H(x)=  
    \begin{bmatrix} 
    2 + 2 \chi_{[0,+\infty)}(x_1) &0 \\
    0 & 2   
    \end{bmatrix},
\end{equation*}
where the notation $\chi_A\colon \HH \to \mathbb{R}$ denotes the characteristic function on the set $A$, i.e., $\chi_A(x) = 1$ if $x \in A$ and $\chi_A(x)=0$ if $x \not\in A$. We have $H(x) \succeq 0$ and $H(x)$ is globally bounded in the sense of \eqref{eq:bounded_Hessian}. We further observe that the function $F$ is strongly convex as the sum of a convex and the strongly convex function $\| \cdot\|^2$ on $\mathbb{R}^2$. Strong convexity implies that it satisfies the global PL condition \cref{ass:PL} and has the global unique minimiser $x^* = (0,0)$. Note that it is not possible to find a ball $B_R(x^*)$ such that $f$ is $C^2$ on $B_R(x^*)$ and therefore the results of \cref{thm:global_convergence_c2} and \cref{lem:c2_implies_dm} (which is used to verify the Dennis--Moré condition, in order to apply \cref{thm:fast_local}) are not applicable.
Now consider the subdifferential of $F$ at $x^*$: we directly obtain by the sum rule %
\begin{equation}
    \partial F(x^*) 
    =  [-1,1] \times\{0\}. \label{eq:subdiff_manifold_example}
\end{equation}
Given the global solution $x^*$, we may set $\mathcal{M}\coloneqq  \{0\} \times(-1,1)$. Then $F$ is $C^2$-partly smooth at $x^*$ relative to $\mathcal{M}$. To see this, we argue as follows. First, \cref{def:partial_smoothness} \ref{item:M_is_C2_manifold} is immediate as we see that $F$ is $C^2$ and $\nabla f$ is $C^1$ on $\mathcal{M}$:
\begin{equation*}
    F|_{\mathcal{M}}(x) = x_2^2, \quad \text{and} \quad  \nabla f|_{\mathcal{M}}(x_1,x_2) = 
    \begin{bmatrix} 
    0\\
    2x_2  
    \end{bmatrix}.
\end{equation*}
Moreover, $F$ is convex and therefore subdifferentially regular, i.e., \cref{def:partial_smoothness} \ref{item:PS_reg} is satisfied. Moreover $N_\mathcal{M}(x^*) = \mathbb{R} \times\{0\} = \parm \partial F(x^*)$ according to \eqref{eq:subdiff_manifold_example}, giving \cref{def:partial_smoothness} \ref{item:PS_NS}. To show continuity of $\partial F$ at $x^*$ relative to $\mathcal{M}$, according to \cite[Note 2.9 (b)]{lewis2002active} it is enough to show that for any sequence $x_k \in \mathcal{M}$ with $x_k \to x^*$ and any $v^* \in \partial F(x^*)$ there is a sequence of subgradients $v_k \in \partial F(x_k)$ with $v_k \to v^*$ (here we note that again by the sum rule, for $x=(0,x_2)$ belonging to $\mathcal{M}$,  $\partial F(x) = [-1,1]\times \{2x_2\}$).
This is easy to prove: for given $x_k=(0,x_{k,2}) \in \mathcal{M}$ and $v^* = (v_1^*,0)$ for some $v_1^*\in [-1,1]$ (every $v^* \in \partial F(x^*)$ is of this form by construction), we simply set $v_k = (v_1^*,2x_{k,2}) \in \partial F(x_k)$ and deduce $v_k\to v^*$ as $x_k \to x^*$. Consequently also  \cref{def:partial_smoothness} \ref{item:PS_SC} holds. Eventually in order to apply \cref{thm:fast_convergence_partial_smoothness}, we need to show non-degeneracy, i.e. $0 \in \relint \partial F(x^*)$. %
Recalling  $\partial F(x^*)  = [-1,1] \times \{0\}$ and $\aff(\partial F(x^*)) = \mathbb{R}\times \{0\}$ and using the definition of the relative interior,
\begin{equation*}
    \relint \partial F(x^*) = \{ v \in [-1,1] \times \{0\} : \text{$\exists \varepsilon>0$ s.t. $\interior(B_\varepsilon(v)) \cap \aff( \partial F(x^*) ) \subseteq  \partial F(x^*)$} \} = (-1,1) \times \{0\}. 
\end{equation*}
Hence $0 \in \relint \partial F(x^*)$ and we may apply \cref{thm:fast_convergence_partial_smoothness} to obtain superlinear convergence of our method on this problem.
\begin{figure}
    \centering
    \begin{minipage}[t]{0.25\textwidth}
    \centering
    \includegraphics[width=\linewidth]{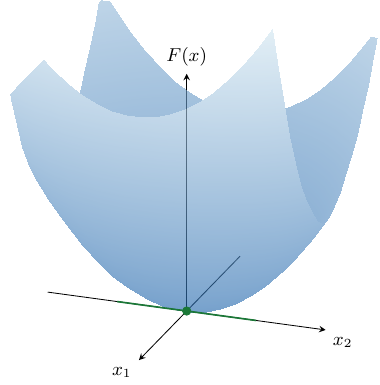} 
    \small (a) 
  \end{minipage}
  \hspace{0.05\textwidth}
  \begin{minipage}[t]{0.25\textwidth}
    \centering
     \includegraphics[width=\linewidth]{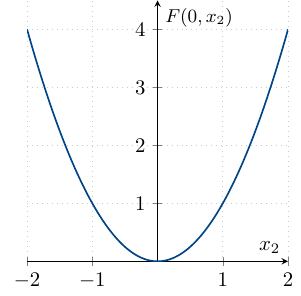}
    \small (b) 
  \end{minipage}
  \hspace{0.05\textwidth}
  \begin{minipage}[t]{0.25\textwidth}
    \centering
    \includegraphics[width=\linewidth]{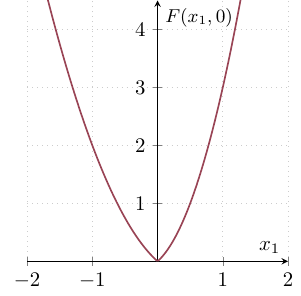}
    \small (c) 
  \end{minipage}
  \caption{Depiction of the function $F: \mathbb{R}^2 \to \mathbb{R}$ as described in the example in \cref{subsec:manifold_example}. 
In panel~(a), the manifold $\mathcal{M}$ is shown in green along with the minimiser $x^* = (0, 0)$. 
Panel~(b) illustrates the $C^2$ behaviour of $F$ when restricted to the manifold. 
Panel~(c) highlights the sharpness of the function along the normal direction $N_{\mathcal{M}}(x^*)$.}

    \label{fig:manifold_example}
\end{figure}

    \section{Numerical Experiments}\label{sec:numerics}

\noindent \textbf{Data availability.} Software to generate all the tables and figures in this section may be found at \cite{github:leapssn}. The specific version used in this paper is archived on Zenodo \cite{zenodo}. The implementation is written in Python. The first two examples utilise the Firedrake library \cite{FiredrakeUserManual} to implement the finite element discretisations as well as PETSc \cite{petsc} for its sparse linear algebra routines. The final example adapts code found in the \texttt{super-newton} library \cite{super-newton}.

In the first two examples, we assume $\Omega \subset \mathbb{R}^2$ to be a bounded and convex domain. As is standard, we denote the Lebesgue spaces by $L^p(\Omega)$, $p \in [1,\infty]$ and write $W^{s,p}(\Omega)$, $s\geq 0$, $p\in[1,\infty]$ to denote the Sobolev spaces  \cite{Adams1975}. We define the Sobolev Hilbert spaces as $H^s(\Omega) \coloneqq W^{s,2}(\Omega)$. We denote the space of compactly supported smooth functions by $C_c^\infty(\Omega)$. For a vector-valued function $v \colon \Omega \to \mathbb{R}^2$ where each component lives in the space $H^s(\Omega)$ (or $C^\infty_c(\Omega)$), we say that $v \in H^s(\Omega;\mathbb{R}^2)$ (respectively $C^\infty_c(\Omega;\mathbb{R}^2)$).

\subsection{A contact problem via a Moreau--Yosida regularisation}
\label{sec:examples:signorini}

We consider a Signorini contact problem \cite{fichera1963} where a partially clamped linearly elastic beam, deformed by gravity, must lie above a nontrivial obstacle. In particular we fix a rectangular domain $\Omega = (0,5) \times (0,1)$ and seek a displacement $u \colon \Omega \to \mathbb{R}^2$ that satisfies %
\begin{subequations}
\label{eq:signorini}
\begin{align}
    \min_{u \in H^1(\Omega;\mathbb{R}^2)} \int_\Omega \mu_L |\varepsilon(u)|^2 + \frac{\lambda_L}{2} |\operatorname{div}(u)|^2 - g \cdot u \, \mathrm{d}x,\\
    u = (0,0)^\top \; \text{on} \; \Gamma_D = \{ (x,y) \in \partial \Omega : x=0 \; \text{or} \; x=5\},\label{signorini:2} \\
    (u - \varphi) \cdot n_u \geq 0 \; \text{on} \; \Gamma_S = \{(x,y) \in \partial \Omega : y = 0\}. \label{signorini:3}
\end{align}
Here $\mu_L = 76.92$ and $\lambda_L = 115.38$ are the Lam\'e coefficients, $\varepsilon(u) = (\nabla u + (\nabla u)^\top)/2$ denotes the symmetrised gradient, $g = (0,-10)^\top$ is the body force due to gravity, and the Dirichlet boundary condition, \labelcref{signorini:2}, models an elastic beam that is clamped on both the left- and right-hand side. Inequality \labelcref{signorini:3} models the Signorini (non-penetration) condition at the bottom of the beam enforcing that the beam must lie above a rigid body with the shape $(0,5)\times(0,1/4)$ under the deformation
\begin{align}
\varphi(x,y) = (0, -1/2-2\sin(\pi x)/5)^\top.
\end{align}
Finally $n_u$  denotes the outward normal of the \emph{deformed} elastic body as given by
\begin{align}
n_u = \frac{(I + \nabla u)^{-\top} \tilde{n}}{|(I + \nabla u)^{-\top} \tilde{n}|} \;\; \text{where} \; \tilde{n} = (0,-1)^\top \; \text{and $I$ is the $2 \times 2$ identity matrix}.
\end{align} 
\end{subequations}
The pointwise constraint can be difficult to enforce directly. Hence, the constraint is relaxed and we include a nonsmooth Moreau--Yosida regularisation to the formulation. For a given  penalisation parameter $\gamma >0$, the relaxed problem seeks $u_\gamma \colon \Omega \to \mathbb{R}^2$ satisfying
\begin{align}
\label{eq:signorini-my}
\begin{gathered}
    u_\gamma = \argmin_{v \in H^1(\Omega;\mathbb{R}^2)} F(v) \;\;\text{subject to} \;\; v = (0,0)^\top \; \text{on} \; \Gamma_D,\\
\text{where} \; F(v) = \int_\Omega  \mu_L |\varepsilon(v)|^2 + \frac{\lambda_L}{2} |\operatorname{div}(v)|^2 - g \cdot v \, \mathrm{d}x + \frac{\gamma}{2} \int_{\Gamma_S} \max(0, (v-\varphi)\cdot n_v)^2 \, \mathrm{d}s.
\end{gathered}
\end{align}
Thus we fix $\psi \equiv 0$ in \eqref{eq:opt}. %
We mesh the domain into 24,000 simplices and discretise \labelcref{eq:signorini-my} with a continuous piecewise linear finite element method for $u$ resulting in 24,682 degrees of freedom \cite[Ch.~7.4]{ern2021:I}. A discretised approximate solution, when $\gamma=10^6$, is plotted in \cref{fig:signorini:a}.

In order to minimise the violation of the non-penetration condition, we wish to consider a large penalisation parameter $\gamma$. However, the larger the penalty, the more difficult the problem becomes to solve. In particular a Newton method, even coupled with a linesearch, will struggle to converge within 200 iterations. This effect is demonstrated in \cref{fig:signorini:b} where we compare \cref{alg:proximal_newton} against a Newton solver without a linesearch, with an $\ell^2$-minimising linesearch \cite[Alg.~2]{Brune2015} ($\ell^2$-Newton) and with a backtracking linesearch \cite[Ch.~6.3]{dennis1996}  (backtracking Newton). From a zero initial guess, we found that even for $\gamma$-values as low as $10^3$, the Newton solver failed to converge. The backtracking and $\ell^2$-Newton solvers were more robust, achieving convergence up to $\gamma=10^4$ and $\gamma=10^3$, respectively, but failing once $\gamma \geq 10^5$. By contrast, \cref{alg:proximal_newton} successfully converged up until $\gamma=10^6$.

\begin{figure}[h!]
\begin{subfigure}[c]{0.49\textwidth}
\includegraphics[width = \textwidth]{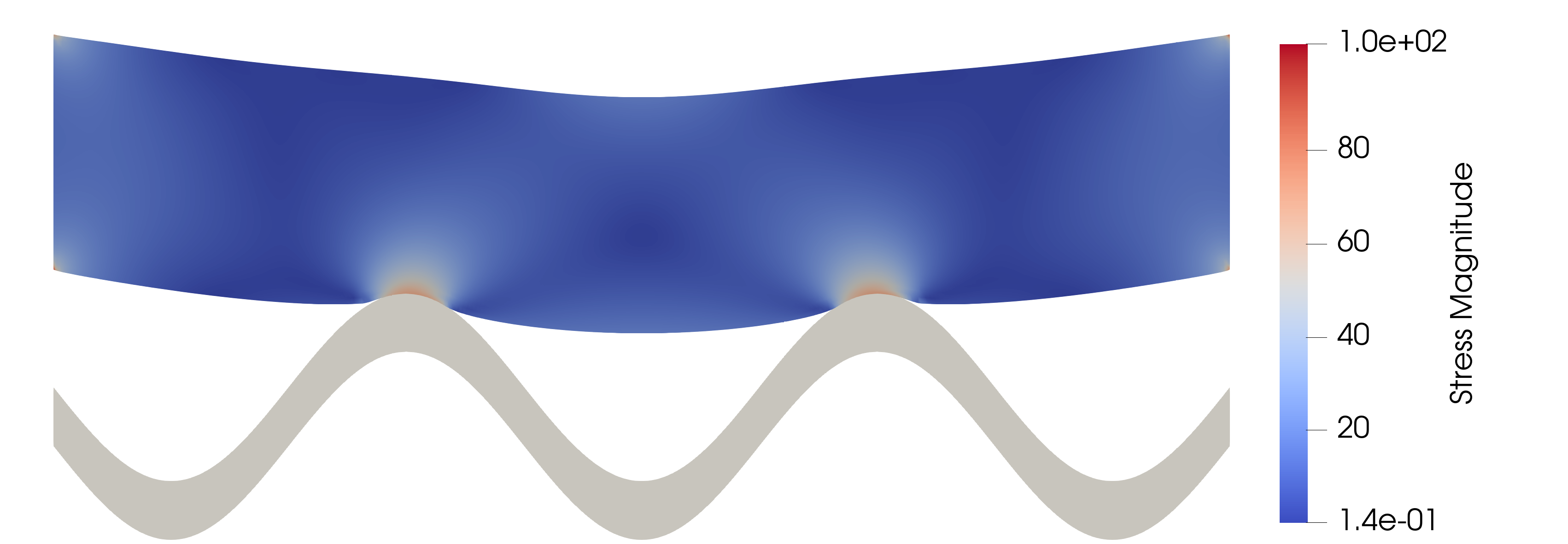}
\caption{Approximate solution when $\gamma=10^6$.}
\label{fig:signorini:a}
\end{subfigure}~
\begin{subfigure}[c]{0.3\textwidth}
\renewcommand{\arraystretch}{1.1}
\centering
\small
\begin{tabular}{|l|c|c|c|c|c|}
\hhline{~|-----}
\rowcolor{lightgray!10} \multicolumn{1}{c|}{\cellcolor{lightgray!00}} & \multicolumn{5}{c|}{$\gamma$}\\
\hhline{|-|-|-|-|-|-|}
\rowcolor{lightgray!10} \multicolumn{1}{|c|}{Solver} & $10^3$ & $10^4$  &$10^5$ &$10^6$ &$10^7$\\
\hline
\cref{alg:proximal_newton} &  23 & 41 &  81 & 185 & - \\
Newton &  - & - &  -& - &-  \\
$\ell^2$-Newton &  63 & - &  -& - &  -\\
Backtracking Newton &  70 & 149 &  -& - &  -\\
\hline
\end{tabular}
\caption{Solver iteration counts.}
\label{fig:signorini:b}
\end{subfigure}
\caption{Results for the contact problem of \cref{sec:examples:signorini}. (Left) The discretised solution to the relaxed Signorini problem \labelcref{eq:signorini-my} with $\gamma=10^6$. The colourmap indicates the magnitude of the stress in the elastic beam and the obstacle is plotted in grey. (Right) The number of linear system solves (including those leading to rejected updates in \cref{alg:proximal_newton}) to reach a residual vector $\ell^2$-norm ($\|F'\|_{\ell^2}$) of $10^{-8}$ for increasing values of $\gamma$. The solvers are initialised at a zero initial guess $u = (0,0)^\top$ for each penalty parameter $\gamma$. A dashed line indicates a failure to converge within 200 linear system solves. We choose the parameters $\alpha = 1/2$, $\beta = 1/4$ and $\Lambda_0 = 1$  for \cref{alg:proximal_newton} and use the $H^1$-norm, $\| u \|^2  =   \|u\|^2_{L^2(\Omega)} + \|\nabla u\|^2_{L^2(\Omega)}$, for the primal space norm and its corresponding discrete dual norm for $\| \cdot \|_*$.
}
\label{fig:signorini}
\end{figure}

In \cref{fig:signorini:convergence} we plot the reduction in the objective against the iteration number $k$. For $\gamma=10^6$, we include the theoretical global convergence rate predicted by \labelcref{eq:conv_rate_product}, where $\alpha=1/2$ and we estimate $D_0$ by $D_0 \approx \max_{i}\{\|u^* - u_i\|_{H^1(\Omega)}\} = \|u^*\|_{H^1(\Omega)}$. This gives us the estimate $F(u_k)-F^* \leq 4 k^{-1} (\Pi_{j=0}^{k-1} \lambda_j)^{1/k} \|u^*\|^2_{H^1(\Omega)}$.  For this example, we see that the algorithm performs orders of magnitude better than what is theoretically guaranteed. We also observe eventual local superlinear convergence for all values of $\gamma$.

\begin{figure}[h!]
\centering
\includegraphics[width =0.5 \textwidth]{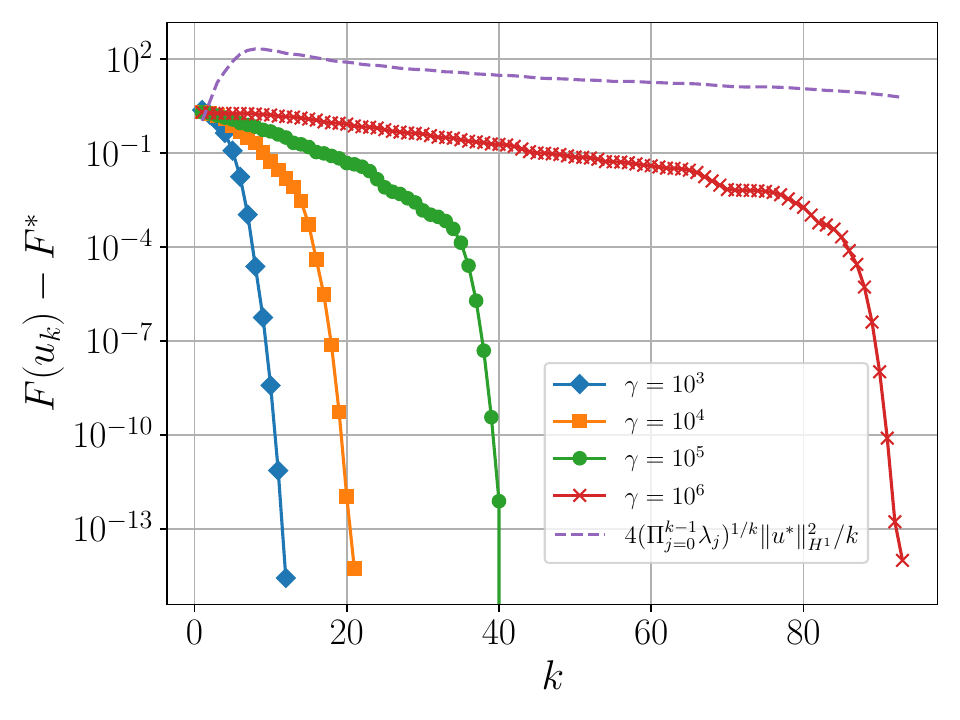}
\caption{The reduction in the difference between the current and final objective with increasing iteration number $k$ when using \cref{alg:proximal_newton} to minimise \labelcref{eq:signorini-my} for various choices of $\gamma$. We observe local superlinear convergence. The dashed line indicates the theoretical global convergence rate when $\gamma=10^6$.}
\label{fig:signorini:convergence}
\end{figure}

\subsection{Total variation image restoration}
\label{sec:examples:image-restoration}
In this example we consider an image restoration problem. The domain is the unit square $\Omega = (0,1)^2$ and we are given a polluted image represented by a function $\omega \coloneqq g + \sigma \eta$ for some $g \colon \Omega \to [0,1]$, $\sigma > 0$, and normally distributed noise $\eta \sim \mathcal{N}(0,1)$. Our goal is to recover an approximation to the original image represented by $g$. One approach is to consider an optimisation problem where the minimiser is regularised by the total variation seminorm. In other words, we seek $u\colon \Omega \to \mathbb{R}$ that minimises
\begin{align}
\min_{u \in \mathrm{BV}(\Omega)} \frac{1}{2}  \|u - \omega\|^2_{L^2(\Omega)}  + \delta |u|_{\mathrm{TV}},
\label{eq:image:1}
\end{align}
where $|u|_{\mathrm{TV}}$ denotes the total variation of $u$,
\begin{align*}
|u|_{\mathrm{TV}} \coloneqq \sup \left\{ \int_\Omega u \operatorname{div}(v) \, \mathrm{d}x : v \in C^\infty_c(\Omega; \mathbb{R}^2), \|v\|_{L^\infty(\Omega)} \leq 1\right\},
\end{align*}
and $\mathrm{BV}(\Omega)$ denotes the space of functions of bounded variation \cite{kunisch2004}.  Regularising $u$ with its total variation, as opposed to the $H^1$-norm, is known to better reduce blurring and better preserve edges \cite{rudin1992,kunisch2004}. However, it is difficult to directly numerically tackle \labelcref{eq:image:1}. Hence, we follow the approach outlined in \cite{kunisch2004} and introduce the equivalent Fenchel predual problem which seeks $p\colon \Omega \to \mathbb{R}^2$ that minimises
\begin{align}
\min_{p \in H_0(\operatorname{div})} \frac{1}{2} \|\operatorname{div}(p) + \omega\|^2_{L^2(\Omega)} \;\; 
\text{subject to} \;\; -\delta {\bf 1} \leq p \leq \delta {\bf 1} \; \text{a.e.~in} \; \Omega,
\label{eq:image:2}
\end{align}
where ${\bf 1} = (1,1)^\top$. Here $H_0(\operatorname{div})$ denotes the space
\begin{align*}
H_0(\operatorname{div}) \coloneqq \{ p \in L^2(\Omega; \mathbb{R}^2) : \operatorname{div}(p) \in L^2(\Omega), \; p \cdot \tilde{n} = 0 \; \text{on} \; \partial \Omega \},
\end{align*}
where $\tilde{n}$ is the outward normal to $\partial \Omega$. The minimisers $u$ of \labelcref{eq:image:1} and $p$ of \labelcref{eq:image:2} are then related by $u = \operatorname{div}(p) + \omega$ \cite[Theorem~2.2]{kunisch2004}.

To approximate the minimiser of \labelcref{eq:image:2}, we first mesh the domain into $400 \times 400$ uniform quadrilateral cells and employ a first-order $H_0(\operatorname{div})$-conforming Raviart--Thomas ($\mathrm{RT}_1$) discretisation for $p$ \cite[Chapter~14.5.2]{ern2021:I}. To enforce the pointwise box constraints, we regularise the problem with a nonsmooth Moreau--Yosida regularisation with penalisation parameter $\gamma>0$. We also include a broken $H^1$-smoothing term controlled by $\epsilon > 0$. Altogether we seek $p_{h} \colon \Omega \to \mathbb{R}^2$ that satisfies
\begin{align*}
\begin{gathered}
p_{h} = \argmin_{q_h \in \mathrm{RT}_1} F(q_h) \;\; \text{subject to} \;\; q_h \cdot \tilde{n} = 0 \; \text{on } \partial \Omega,\\
\text{where} \; F(q) = \frac{1}{2} \|\operatorname{div}(q) + \omega\|^2_{L^2(\Omega)} 
+ \frac{\gamma}{2} \| \max({\bf 0},  q-\delta{\bf 1})\|^2_{L^2(\Omega)}  + \frac{\gamma}{2} \| \min({\bf 0},  q+\delta{\bf 1})\|^2_{L^2(\Omega)}
+ \epsilon | q|^2_{H^1(\mathcal{T}_h)}.
\end{gathered}
\end{align*}
Here $\mathcal{T}_h$ is the triangulation of the domain and $|q|^2_{H^1(\mathcal{T}_h)} \coloneqq \sum_{K \in \mathcal{T}_h} |\nabla q|^2_{L^2(K)}$. Note that we set $\psi \equiv 0$ in \eqref{eq:opt}. %

We now fix the model parameters $\delta = 10^{-4}$, $\epsilon = 10^{-1}$, and $\sigma = 0.06$. We pick $g$ to be the piecewise constant finite element projection of the  Shepp--Logan phantom image \cite{shepp1974}. We will consider various choices for $\gamma$.

We plot the original and polluted images as well as the restoration in \cref{fig:image-restoration:a}. We also provide the approximate solutions' peak signal-to-noise ratio (PSNR) as defined by
$\operatorname{PSNR}(u) =     10 \log_{10}( |\Omega| \cdot \| u - g\|^{-2}_{L^2(\Omega)} )$. As in the previous example, in \cref{fig:image-restoration:b}  we compare \cref{alg:proximal_newton} against a Newton solver without a linesearch, with an $\ell^2$-minimising linesearch \cite[Algorithm~2]{Brune2015}, and a backtracking linesearch \cite[Chapter~6.3]{dennis1996} for a range of values of $\gamma$.
We observe that the classical Newton method is surprisingly effective requiring the fewest linear system solves when $\gamma \leq 10^6$. However, it fails to converge once $\gamma=10^7$. The other three solves do converge. Aside from when $\gamma=10^4$, backtracking Newton requires the largest number of iterations. \cref{alg:proximal_newton} and $\ell^2$-Newton have a similar performance throughout.

\begin{figure}[h!]
\begin{subfigure}[c]{0.49\textwidth}
\includegraphics[width = \textwidth]{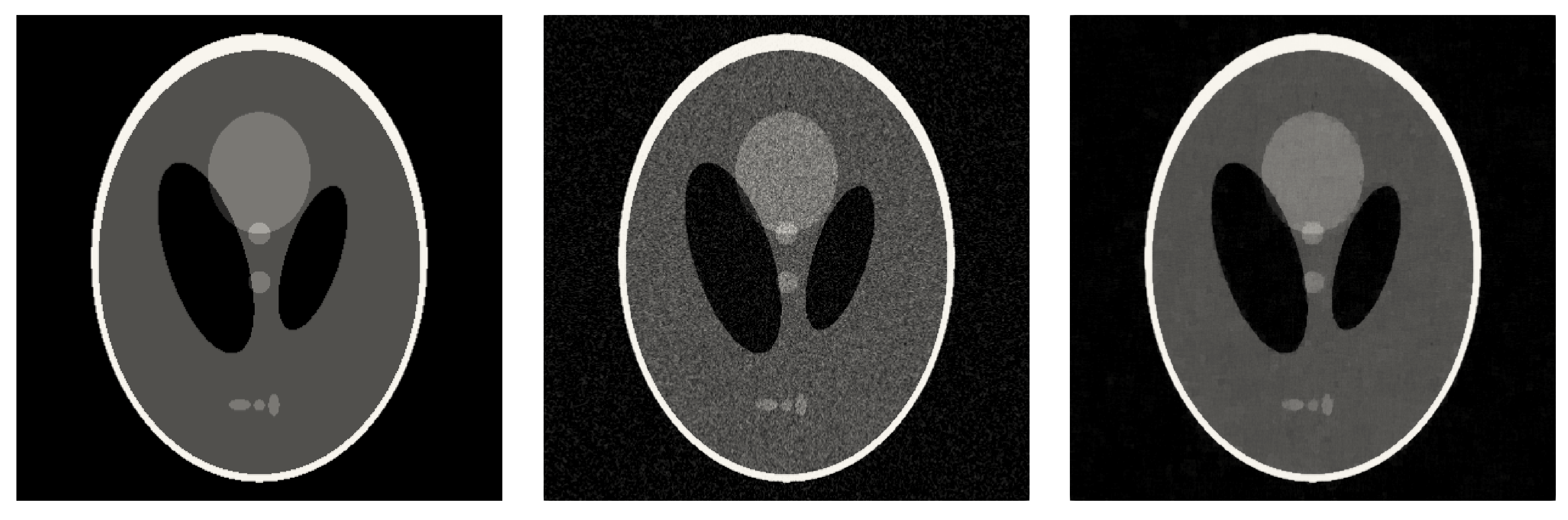}
\caption{The original, polluted, and restored images.}
\label{fig:image-restoration:a}
\end{subfigure}~
\begin{subfigure}[c]{0.3\textwidth}
\renewcommand{\arraystretch}{1.1}
\centering
\small
\begin{tabular}{|l|c|c|c|c|}
\hhline{~|----}
\rowcolor{lightgray!10} \multicolumn{1}{c|}{\cellcolor{lightgray!00}} & \multicolumn{4}{c|}{$\gamma$}\\
\hhline{|-|-|-|-|-|}
\rowcolor{lightgray!10} \multicolumn{1}{|c|}{Solver} & $10^4$ & $10^{5}$  &$10^{6}$ &$10^{7}$\\
\hline
\cref{alg:proximal_newton} &  17 & 22 & 31 & 55   \\
Newton & 10  & 13 &  18 & -   \\
$\ell^2$-Newton & 11 & 22 & 36 & 53 \\
Backtracking Newton & 10  & 33 & 74 & 127 \\
\hline
\end{tabular}
\caption{Solver iteration counts.}
\label{fig:image-restoration:b}
\end{subfigure}
\caption{Results for the image restoration problem of \cref{sec:examples:image-restoration} with parameters $\delta = 10^{-4}$, $\sigma = 0.06$, and $\epsilon = 10^{-1}$. (Left) The original ($g$), polluted ($\omega$, PSNR: 27.97) and restored ($u = \operatorname{div}(p) + \omega$, PSNR: 39.10) images when $\gamma=10^7$. (Right) The number of linear system solves (including those leading to rejected updates in \cref{alg:proximal_newton}) to reach a residual vector $\ell^2$-norm ($\|F'\|_{\ell^2}$) of $10^{-8}$ for increasing values of $\gamma$. We use the initial guess $p = \delta{\bf 1}$ for each value of $\gamma$. \cref{alg:proximal_newton} is initialised with $\alpha = \beta = 10^{-4}$ and $\Lambda_0 = 64$. We use the norm $\| p \|^2  = \|p\|^2_{L^2(\Omega)} + \| \operatorname{div} p\|^2_{L^2(\Omega)} + |p|^2_{H^1(\mathcal{T}_h)}$ for the primal space norm and its corresponding discrete dual norm for $\| \cdot \|_*$.  A dashed line indicates a failure to converge within 500 linear system solves.
}
\label{fig:image-restoration}
\end{figure}

\subsection{Support Vector Machine classification}
\label{sec:examples:svc}

Our final example is a Support Vector Machine classification problem from the field of machine learning \cite{cortes1995}. Given a number of points $x_1, \dots, x_\ell \in \mathbb{R}^n$ for some $\ell, n \in \mathbb{N}$ where each point $x_i$ is classified into one of two groups denoted by $y_i \in \{-1,1\}$, $i=1,\dots,\ell$, the aim is to find an $\omega \in \mathbb{R}^n$ and $b \in \mathbb{R}$ that defines a hyperplane $\omega^\top x + b =0$ which separates the two groups. As in \cite{yin2019}, given a $\gamma>0$, the L2-loss SVM model seeks $(\omega, b) \in \mathbb{R}^n \times \mathbb{R}$ that minimises 
\begin{align*}
F(\omega,b) = \frac{1}{2} \|\omega\|^2_{\ell^2} + \gamma \sum_{i=1}^\ell \max(1 - y_i (\omega^\top x_i + b), 0)^2.
\end{align*}
We visualise the solution to one data set where $n=2$ and $\ell=10,000$ in \cref{fig:svc:a}. In \cref{fig:svc:b}, we test \cref{alg:proximal_newton} for $\ell=10,000$ with $n \in \{2,20,200, 2000\}$ and $\gamma \in \{10^{-4}, 10^{-2}, 1, 10^2, 10^4\}$. The data $x_i$ and $y_i$, $i \in \{1,\dots,\ell\}$ are generated by a reproducible classification problem generator from the Python package \texttt{scikit-learn} \cite{scikit-learn}. We observe that \cref{alg:proximal_newton} always converges for the parameters we consider. More iterations are required as both $n$ and $\gamma$ become larger.

\begin{figure}[h!]
\centering
\begin{subfigure}[c]{0.4\textwidth}
\includegraphics[width = \textwidth]{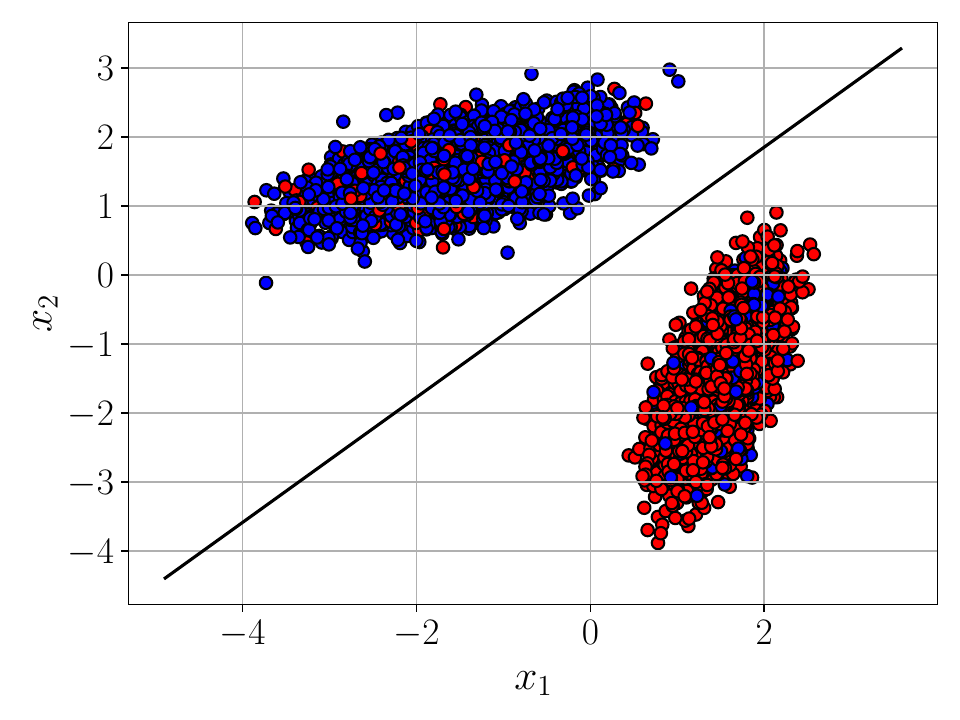}
\caption{The data points and separating hyperplane.}
\label{fig:svc:a}
\end{subfigure}~\qquad
\begin{subfigure}[c]{0.3\textwidth}
\renewcommand{\arraystretch}{1.1}
\centering
\small
\begin{tabular}{|l|c|c|c|c|}
\hhline{~|----}
\rowcolor{lightgray!10} \multicolumn{1}{c|}{\cellcolor{lightgray!00}} & \multicolumn{4}{c|}{$n$}\\
\hhline{|-|-|-|-|-|}
\rowcolor{lightgray!10} \multicolumn{1}{|c|}{$\gamma$} & 2 & 20  & 200 & 2000\\
\hline
$10^{-4}$ &  5 & 4 & 7 & 10  \\
$10^{-2}$  & 6  & 10 & 18 & 21   \\
$10^{0}$  & 7 & 21 & 24 & 30 \\
$10^{2}$  & 8  & 27 & 31 & 36 \\
$10^{4}$ & 9  & 34 & 37 & 43 \\
\hline
\end{tabular}
\caption{Solver iteration counts.}
\label{fig:svc:b}
\end{subfigure}
\caption{Results for the Support Vector Machine classification problem of \cref{sec:examples:svc}. (Left) Plot of the data points $x_\ell \in \mathbb{R}^2$ and the separating line $\omega^\top x + b$ when $n=2$ and $\gamma=10^2$. A red and blue dot indicates a value of $y_i=1$ and $y_i=-1$, respectively. (Right) The number of linear system solves (including those leading to rejected updates) for \cref{alg:proximal_newton} to reach a residual vector $\ell^2$-norm ($\|F'\|_{\ell^2}$) of $10^{-6}$ for various values of $\gamma$ and $n$. \cref{alg:proximal_newton} is initialised with $\alpha = \beta = 10^{-1}$, $\Lambda_0 = 3 \gamma \cdot (\sum_{i=1}^\ell\|x_i\|^2_{\ell^2})^{1/2}$, and $(\omega_0, b_0) = ({\bf 1}/2, 1/2)$. We use the Euclidean norm $\| z \|^2  = \| z \|_*^2 = \|z\|^2_{\ell^2}$ for the primal and dual norms.
}
\label{fig:svc}
\end{figure}

\section*{Conclusion}
In this paper, we introduced our SSN-type algorithm \algname{LeAP-SSN} (\cref{alg:proximal_newton}) and showed that it possesses classical local superlinear convergence as well as global convergence rates and, in some settings, accelerates to global superlinear convergence. Notably, these results are achieved with a relaxation of the standard strong convexity assumption to a PL condition. A key to our global convergence rates is \cref{ass:main_inequality}, which is satisfied for a large class of semismooth functions. Another important feature of our algorithm is its adaptivity not only to the parameters of the function class (such as $\mu$ or $L$), but also to the class itself. In other words, we have one algorithm for all the considered settings without requiring knowledge of whether the function is convex or not, whether it satisfies the PL condition or not, nor whether the function is semismooth or not.  We also explored how partial smoothness and active manifold identification can be utilised in the nonsmooth context. We concluded with some numerical experiments in both finite and infinite-dimensional contexts. We believe that this paper is a good starting point for future research on algorithms that combine global convergence rates that are well studied in the world of second-order methods for functions with Lipschitz Hessians and improved asymptotic convergence rates well developed in the works on SSN-type methods.

\printbibliography

\end{document}